\documentclass[11pt,leqno]{amsart}
\makeatletter
\AtBeginDocument{%
  \let\original@@tocwrite=\@tocwrite
  \newif\ifAHenVflag
  \def\jlreq@uniqtoken{\jlreq@uniqtoken}
  \def\jlreq@endmark{\jlreq@endmark}
  \long\def\jlreq@getfirsttoken#1#{\jlreq@getfirsttoken@#1\bgroup\jlreq@endmark}
  \long\def\jlreq@getfirsttoken@#1#2\jlreq@endmark#3\jlreq@endmark{#1}
  \renewcommand{\@tocwrite}[2]{%
    \begingroup
      \AHenVflagfalse 
      \@ifempty{#2}{}{%
        \expandafter\expandafter\expandafter\ifx\jlreq@getfirsttoken#2\jlreq@uniqtoken{}\jlreq@endmark\Sectionformat\expandafter\@firstoftwo\else\expandafter\@secondoftwo\fi
        {%
          \def\Sectionformat##1##2{\@ifempty{##1}{}{\AHenVflagtrue}}%
          #2
        }{\AHenVflagtrue}%
      }%
      \def\@tempa{}%
      \ifAHenVflag\def\@tempa{\original@@tocwrite{#1}{#2}}\fi
    \expandafter\endgroup
    \@tempa
  }%
}
\makeatother

\usepackage{a4wide}
\usepackage{amssymb}
\usepackage{amsthm}
\usepackage{amsmath}
\theoremstyle{plain}

\newtheorem{theorem}{Theorem}[section]
\newtheorem{corollary}[theorem]{Corollary}
\newtheorem{lemma}[theorem]{Lemma}
\newtheorem{proposition}[theorem]{Proposition}

\newtheorem{proposition-definition}[theorem]{Proposition-Definition}
\theoremstyle{definition}
\newtheorem{definition}[theorem]{Definition}
\newtheorem{question}[theorem]{Question}
\theoremstyle{remark}
\newtheorem{remark}[theorem]{Remark}

\newtheorem{example}[theorem]{Example}

\DeclareMathOperator{\Ind}{Ind}
\DeclareMathOperator{\Mod}{Mod}
\DeclareMathOperator{\ind}{c-Ind}
\DeclareMathOperator{\Hom}{Hom}

\DeclareMathOperator{\val}{val}

\DeclareMathOperator{\Lie}{Lie}
\DeclareMathOperator{\Char}{char}
\DeclareMathOperator{\Coind}{Coind}
\newcommand*{\St}{\mathrm{St}}
\newcommand*{\charone}{\mathbf{1}}

\newcommand*{\id}{\mathrm{id}}

\newcommand*{\aff}{\mathrm{aff}}

\numberwithin{equation}{section}

\author{N. Abe}
\address[N. Abe]{Department of Mathematics, Hokkaido University, Kita 10, Nishi 8, Kita-Ku, Sapporo, Hokkaido, 060-0810, Japan}
\thanks{The first-named author was supported by JSPS KAKENHI Grant Number 26707001.}
\email{abenori@math.sci.hokudai.ac.jp}

\author{G. Henniart} 
\address[G. Henniart]{Universit\'e de Paris-Sud, Laboratoire de Math\'ematiques d'Orsay, Orsay cedex F-91405 France;
CNRS, Orsay cedex F-91405 France}
\email{Guy.Henniart@math.u-psud.fr}

\author{M.-F. Vign\'eras}
\address[M.-F. Vign\'eras]{Institut de Math\'ematiques de Jussieu, 175 rue du Chevaleret, Paris 75013 France}
\email{vigneras@math.jussieu.fr}

\title{On pro-$p$-Iwahori invariants of $R$-representations of reductive $p$-adic groups}
\date{}

\begin{document} 

\begin{abstract}
Let $F$ be locally compact field with residue characteristic $p$, and $\mathbf{G}$ a connected reductive $F$-group.
Let $\mathcal{U}$ be a pro-$p$ Iwahori subgroup of $G = \mathbf{G}(F)$.
Fix a commutative ring $R$.
If $\pi$ is a smooth $R[G]$-representation, the space of invariants $\pi^{\mathcal{U}}$ is a right module over the Hecke algebra $\mathcal{H}$ of $\mathcal{U}$ in $G$.

Let $P$ be a parabolic subgroup of $G$ with a Levi decomposition $P = MN$ adapted to $\mathcal{U}$.
We complement previous investigation of Ollivier-Vign\'eras on the relation between taking $\mathcal{U}$-invariants and various functor like $\Ind_P^G$ and right and left adjoints.
More precisely the authors' previous work with Herzig introduce representations $I_G(P,\sigma,Q)$ where $\sigma$ is a smooth representation of $M$ extending, trivially on $N$, to a larger parabolic subgroup $P(\sigma)$, and $Q$ is a parabolic subgroup between $P$ and $P(\sigma)$.
Here we relate $I_G(P,\sigma,Q)^{\mathcal{U}}$ to an analogously defined $\mathcal{H}$-module $I_\mathcal{H}(P,\sigma^{\mathcal{U}_M},Q)$, where $\mathcal{U}_M = \mathcal{U}\cap M$ and $\sigma^{\mathcal{U}_M}$ is seen as a module over the Hecke algebra $\mathcal{H}_M$ of $\mathcal{U}_M$ in $M$.
In the reverse direction, if $\mathcal{V}$ is a right $\mathcal{H}_M$-module, we relate $I_\mathcal{H}(P,\mathcal{V},Q)\otimes \ind_\mathcal{U}^G\charone$ to $I_G(P,\mathcal{V}\otimes_{\mathcal{H}_M}\ind_{\mathcal{U}_M}^M\charone,Q)$.
As an application we prove that if $R$ is an algebraically closed field of characteristic $p$, and $\pi$ is an irreducible admissible representation of $G$, then the contragredient of $\pi$ is $0$ unless $\pi$ has finite dimension.
\end{abstract}
\maketitle

 \setcounter{tocdepth}{2}  
 
 \tableofcontents
    
\section{Introduction}   
\subsection{}
The present paper is a companion to \cite{arXiv:1703.05599} and is similarly inspired by the classification results of  \cite{MR3600042}; however it can be read independently.
We recall the setting. We have a non-archimedean locally compact field $F$ of residue characteristic $p$ and a connected reductive $F$-group $\mathbf G$.  We fix  a commutative ring $R$ and study the smooth $R$-representations of $G = \mathbf{G}(F)$.

In \cite{MR3600042} the irreducible  admissible $R$-representations of $G$ are classified in terms of supersingular ones when $R$ is an algebraically closed field of characteristic $p$. That classification is expressed in terms of   representations $I_G(P,\sigma,Q)$, which make sense for any $R$. In that notation,
 $P$ is a parabolic subgroup of $G$ with a Levi decomposition $P=MN$ and $\sigma$ a smooth $R$-representation of the Levi subgroup $M$; there is a maximal parabolic subgroup $P(\sigma)$ of $G$ containing $P$ to which $\sigma$ inflated to $P$ extends to a representation $e_{P(\sigma)}(\sigma)$, and $Q$ is a parabolic subgroup of $G$ with $P\subset Q \subset P(\sigma)$. Then
$$I_G(P,\sigma,Q)=\Ind_{P(\sigma)}^G(e_{P(\sigma)}(\sigma)\otimes \St_Q^{P(\sigma)})$$
where $\Ind$ stands for parabolic induction and $\St_Q^{P(\sigma)}= \Ind_Q^{P(\sigma)}R/ \sum\Ind_{Q'}^{P(\sigma)}R$, the sum being over parabolic subgroups $Q'$ of $G$ with $Q\subsetneq Q'\subset P(\sigma)$. 
Alternatively, $I_G(P,\sigma,Q)$ is the quotient of $\Ind_{P(\sigma)}^G(e_{P(\sigma)}(\sigma))$ by $ \sum\Ind_{Q'}^{G}e_{Q'}(\sigma)$ with $Q'$ as above, where $e_{Q}(\sigma)$ is the restriction of $e_{P(\sigma)}(\sigma)$ to $Q$, similarly for $Q'$.

In \cite{arXiv:1703.05599} we mainly studied what happens to $I_G(P,\sigma,Q)$ when we apply to it, for a parabolic subgroup $P_1$ of $G$, the left adjoint of $\Ind_{P_1}^G$, or its right adjoint. Here we tackle a different question. We fix a pro-$p$ parahoric subgroup $\mathcal{U}$ of $G$ in good position with respect to $P$, so that in particular $\mathcal{U}_M=\mathcal{U}\cap M$ is a  pro-$p$ parahoric subgroup of $M$. One of our main goals is to identify the $R$-module 
$I_G(P,\sigma,Q)^\mathcal{U}$ of $\mathcal{U}$-invariants, as a right module over the Hecke  algebra $\mathcal{H}=\mathcal{H}_G$ of $\mathcal{U}$ in $G$ - the convolution  algebra on the double coset space $\mathcal{U} \backslash G/\mathcal{U}$ -  in terms on the module $\sigma^{\mathcal{U}_M}$ over the Hecke  algebra $ \mathcal{H}_M$ of $\mathcal{U}_M$ in $M$. That goal is achieved in section \ref{S:9.0}, Theorem~\ref{thm:main}.

\subsection{} The initial work has been done in \cite[\S 4]{arXiv:1703.04921} where $(\Ind_{P }^G\sigma)^\mathcal{U}$ is identified. Precisely, writing $M^+$ for the monoid  of  elements $m\in M$ with $m(\mathcal{U} \cap N) m^{-1} \subset \mathcal{U} \cap N$, the subalgebra  $\mathcal{H}_{M^+}$  of $\mathcal{H}_M$  with support in $M^+$,  has a natural algebra embedding $\theta$ into $\mathcal{H}$ and \cite[Proposition 4.4]{arXiv:1703.04921} identifies  $(\Ind_{P }^G\sigma)^\mathcal{U}$ with $\Ind_{\mathcal{H}_M}^{\mathcal{H}} (\sigma^{\mathcal{U}_M})=\sigma^{\mathcal{U}_M}\otimes_{\mathcal{H}_{M^+}}\mathcal{H}$.  So in a sense, this paper is a sequel to \cite{arXiv:1703.04921} although some of our results here are used in \cite[\S 5]{arXiv:1703.04921}.

As $I_G(P,\sigma,Q)$ is only a subquotient of $\Ind_{P }^G\sigma$ and taking $\mathcal{U}$-invariants is only left exact, it is not straightforward to describe $I_G(P,\sigma,Q)^\mathcal{U}$ from the previous result. However, that takes care of the parabolic induction step, so in a first approach we may assume $P(\sigma)=G$ so that $I_G(P,\sigma,Q)=e_{G}(\sigma)\otimes \St_Q^{G}$. The crucial case is when moreover $\sigma$ is $e$-minimal, that is, not an extension 
$e_M(\tau)$ of a smooth $R$-representation $\tau$ of a proper Levi subgroup of $M$. That case is treated first and the general case in section \ref{S:9.0} only.

\subsection{}\label{S:1.3} To explain our results, we need more notation. We choose a maximal $F$-split torus $T$ in $G$, a minimal parabolic subgroup $B=ZU$ with Levi component $Z$ the $G$-centralizer of $T$. We assume that $P=MN$ contains $B$ and $M$ contains $Z$,  and that $\mathcal{U}$ corresponds to an alcove in the apartment associated to $T$ in the adjoint building of $G$. It turns out that when $\sigma$ is $e$-minimal, the set $\Delta_M$ of simple roots of $T$ in $\Lie N$ is orthogonal to its complement in the set $\Delta$ of simple roots of $T$ in $\Lie U$. We assume until the end of this section \S \ref{S:1.3}, that $\Delta_{M}$ and 
 $\Delta_2=\Delta \setminus   \Delta_M$ are orthogonal. If $M_2$ is the Levi subgroup - containing $Z$ -  corresponding to $\Delta_2$, both $M$ and $M_2$ are normal in $G$, $M\cap M_2=Z$ and $G=M_1M_2$. Moreover the normal subgroup $M'_2$ of $G$ generated by $N$ is included in $M_2$ and $G=MM'_2$. 
  
 We  say that a right $\mathcal{H}_M$-module  $\mathcal{V}$ is extensible to $\mathcal{H}$ if $T_z^M $ acts trivially on  $\mathcal{V}$  for $z\in Z\cap M'_2$ (\S \ref{S:8.2}). In this case,  we show that there is a natural structure of right $\mathcal{H}$-module $e_{\mathcal{H}}(\mathcal{V})$ on $\mathcal{V}$ such that $T_g\in \mathcal{H}$   corresponding to $\mathcal{U} g \mathcal{U} $ for $g\in   M'_2$ acts as in the trivial character of $G$ (\S \ref{S:8.1}). We call $e_\mathcal{H}(\mathcal{V})$ the extension of $\mathcal{V}$ to $\mathcal{H}$ though $\mathcal{H}_M$ is not a subalgebra of $\mathcal{H}$.  That notion is already present in \cite{arXiv:1406.1003_accepted}  in the case where $R$ has characteristic $p$. Here we extend the construction to any $R$ and prove some more properties. In particular we produce an $\mathcal{H}$-equivariant embedding $e_\mathcal{H}(\mathcal{V})$ into $\Ind_{\mathcal{H}_M}^{\mathcal{H}} \mathcal{V}$ (Lemma \ref{lemma:sigmaemb}). 
  If  $Q$ is a parabolic subgroup of $G$ containing $P$, we go further and put on $e_\mathcal{H}(\mathcal{V})\otimes_R(\Ind_Q^GR)^\mathcal{U}$ and $e_\mathcal{H}(\mathcal{V})\otimes_R(\St_Q^G )^\mathcal{U}$  structures of $\mathcal{H}$-modules (Proposition \ref{lemma:tensor} and Corollary~\ref{cor:VS}) - note that $\mathcal{H}$ is not a group algebra and there is no obvious notion of tensor product of $\mathcal{H}$-modules.

 If $\sigma$ is an $R$-representation of $M$ extensible to $G$, then its extension $e_G(\sigma)$ is simply obtained by letting $M'_2$ acting trivially on the space of $\sigma$; moreover it is clear that  $\sigma^{\mathcal{U}_M}$   is extensible to $\mathcal{H}$, and one shows easily that   $e_G(\sigma)^\mathcal{U}=e_\mathcal{H}(\sigma^{\mathcal{U}_M})$ as an $\mathcal{H}$-module (\S \ref{S:8.4}). Moreover, the natural inclusion of $\sigma$ into $\Ind_P^G\sigma$ induces on taking pro-$p$ Iwahori invariants an embedding $e_\mathcal{H}(\sigma^{\mathcal{U}_M})\to (\Ind_{P }^G\sigma)^\mathcal{U}$ which, via the isomorphism of \cite{arXiv:1703.04921}, yields exactly the above embedding of $\mathcal{H}$-modules of $e_\mathcal{H}(\sigma^{\mathcal{U}_M})$ into $\Ind_{\mathcal{H}_M}^{\mathcal{H}} (\sigma^{\mathcal{U}_M})$.

 Then we show that the $\mathcal{H}$-modules $(e_G(\sigma)\otimes_R\Ind_Q^GR)^\mathcal{U}$ and $e_\mathcal{H}(\sigma^{\mathcal{U}_M})\otimes_R(\Ind_Q^GR)^\mathcal{U}$  are equal, and similarly 
 $(e_G(\sigma)\otimes_R\St_Q^G)^\mathcal{U} $and $e_\mathcal{H}(\sigma^{\mathcal{U}_M})\otimes_R(\St_Q^G)^\mathcal{U} $ are equal (Theorem~\ref{thm:main8}).

\subsection{} \label{S:1.4}We turn back to the general case where we do not assume that  $\Delta_{M}$ and 
 $ \Delta \setminus   \Delta_M$ are orthogonal. Nevertheless, given a  right $\mathcal{H}_M$-module  $\mathcal{V}$, there exists a 
largest  Levi subgroup $M(\mathcal{V})$  of $G$ - containing $Z$ - corresponding to $\Delta \cup \Delta_1$ where $\Delta_1$ is a subset of $ \Delta \setminus   \Delta_M$ orthogonal to $\Delta_M$, such that $\mathcal{V}$ extends 
  to a right $\mathcal{H}_{M(\mathcal{V})}$-module $e_{M(\mathcal{V})}(\mathcal{V})$ with the notation of  section \eqref{S:1.3}. For any parabolic subgroup $Q$ between $P$ and $P(\mathcal{V})=M(\mathcal{V})U$ we put (Definition  \ref{def:Htriple})
  $$I_\mathcal{H}(P,\mathcal{V},Q)= \Ind_{\mathcal{H}_M}^\mathcal{H} (e_{M(\mathcal{V})}(\mathcal{V})\otimes_R (\St^{M(\mathcal{V})}_{Q\cap M(\mathcal{V}))})^{\mathcal{U}_{M(\mathcal{V})}}).$$
  We refer to Theorem \ref{thm:main} for the description of the right $\mathcal{H}$-module $I_G(P,\sigma,Q)^\mathcal{U}$ for any smooth $R$-representation $\sigma$ of $\mathcal{U}$. As a special case, it says that when $\sigma$ is  $e$-minimal then $P(\sigma)\supset P(\sigma^{\mathcal{U}_M})$ and if moreover $P(\sigma)=P(\sigma^{\mathcal{U}_M})$ then $I_G(P,\sigma,Q)^\mathcal{U}$ is isomorphic to
  $I_\mathcal{H}(P,\sigma^{\mathcal{U}_M},Q)$.
  
  \begin{remark}
  In \cite{arXiv:1406.1003_accepted} are attached similar $\mathcal{H}$-modules to $(P,\mathcal{V},Q)$; here we write them $CI_\mathcal{H}(P,\mathcal{V},Q)$ because their definition uses, instead of  $\Ind_{\mathcal{H}_M}^\mathcal{H}$ a different kind of induction, which we call coinduction. In loc.\ cit.\ those modules are use to give, when $R$ is an algebraically closed field of characteristic $p$, a classification of simple $\mathcal{H}$-modules in terms of supersingular  modules - that classification is similar to the classification of irreducible admissible $R$-representations of $G$ in \cite{MR3600042}. Using the comparison between induced and coinduced modules established in \cite[4.3]{MR3437789} for any $R$, our corollary \ref{cor:isoCI} expresses  $CI_\mathcal{H}(P,\mathcal{V},Q)$ as a module $I_\mathcal{H}(P_1,\mathcal{V}_1,Q_1)$; consequently we show in \S \ref{S:9.5} that the classification of \cite{arXiv:1406.1003_accepted} can also be expressed in terms of modules $I_\mathcal{H}(P,\mathcal{V},Q)$.
  \end{remark}

\subsection{} In a reverse direction one can associate to a right $\mathcal{H}$-module $\mathcal{V}$ a smooth $R$-representation $\mathcal{V} \otimes_\mathcal{H} R[\mathcal{U}\backslash G]$ of $G$  (seeing $\mathcal{H}$ as the endomorphism ring of the $R[G]$-module $R[\mathcal{U}\backslash G]$).

If  $\mathcal{V}$ is a right $\mathcal{H}_M$-module, we construct, again using \cite{arXiv:1703.04921}, a natural $R[G]$-map
$$I_\mathcal{H}(P,\mathcal{V},Q)\otimes_\mathcal{H} R[\mathcal{U}\backslash G]\to \Ind_{P(\mathcal{V})}^G (e_{M(\mathcal{V})}(\mathcal{V})\otimes_R \St^{M(\mathcal{V})}_{Q\cap M(\mathcal{V})}),
$$
with the notation of \eqref{S:1.4}. We show in \S \ref{S:10} that it is an isomorphism under a mild assumption on the $\mathbb Z$-torsion in $\mathcal{V}$; in particular it is an isomorphism if $p=0$ in $R$.

\subsection{}
In the final section \S \ref{S:9}, we turn back to the case where $R$ is an algebraically closed field of characteristic $p$.
We prove that the smooth dual of an irreducible admissible $R$-representation $V$ of $G$ is $0$ unless $V$ is finite dimensional - that result is new if $F$ has positive characteristic, a case where 
the proof of Kohlhaase \cite{kohlhasse-smooth-duality} for $\Char (F)=0$ does not apply. Our proof first reduces to the case where $V$ is supercuspidal (by \cite{MR3600042}) then uses again the $\mathcal{H}$-module $V^\mathcal{U}$.

 \section{Notation, useful facts and preliminaries}\label{S:2}
 \subsection{The group $G$ and its standard parabolic subgroups $P=MN$}\label{S:2.1}
 In all that follows, $p$ is a prime number, $F$ is a local field with finite residue field $k$ of characteristic $p$;
 We denote an algebraic group over $F$ by a bold letter, like $\mathbf H$, and use the same ordinary letter for the group of $F$-points, $H=\mathbf H(F)$.
  We fix a connected reductive $F$-group $\mathbf G$. We fix a maximal $F$-split subtorus $\mathbf T$ and write $\mathbf Z$ for its $\mathbf G$-centralizer; we also fix a minimal parabolic subgroup $\mathbf B$ of $\mathbf G$  with Levi component $\mathbf Z$, so that $\mathbf B=\mathbf  Z  \mathbf U$ where $ \mathbf U$ is the unipotent radical of  $\mathbf B$. Let $X^*( \mathbf T)$ be the group of $F$-rational characters of  $\mathbf T$ and $\Phi$ the subset of roots of  $\mathbf T$ in the Lie algebra of  $\mathbf G$. Then $ \mathbf B$ determines a subset $\Phi^+$ of positive roots - the roots of $\mathbf T$ in the Lie algebra of $\mathbf U$-  and a subset of simple roots $\Delta$. The $\mathbf G$-normalizer  $\mathbf N_{\mathbf G}$ of $\mathbf T$ acts on $X^*( \mathbf T)$ and through that action,  $\mathbf N_{\mathbf G}/ \mathbf Z$ identifies with the Weyl group of the root system $\Phi$.
Set $\mathcal N:=\mathbf N_{\mathbf G}(F)$ and note that  $\mathbf N_{\mathbf G}/ \mathbf Z \simeq \mathcal N/Z$; we write $\mathbb W $ for  $\mathcal N/Z$.

A standard parabolic subgroup  of $\mathbf G$ is a parabolic $F$-subgroup containing $\mathbf B$. Such a  parabolic subgroup $\mathbf P$  has a unique Levi subgroup $\mathbf M$ containing $\mathbf Z$, so that $\mathbf P=\mathbf M\mathbf N$ where $\mathbf N$ is the unipotent radical of $\mathbf P$ -  we also call  $\mathbf M$ standard.  By a common abuse of language to describe the preceding situation, we simply say ``let $P=MN$ be a standard parabolic subgroup of $G$''; we sometimes write $N_P$ for $N$ and $M_P$  for $M$. The parabolic subgroup of $G$ opposite to $P$ will be written $\overline P$ and its unipotent radical $\overline N$, so that $\overline P=M\overline N$, but beware that $\overline P$ is not standard !  We write  $\mathbb W_M$ for the Weyl group $(M\cap\mathcal{N})/Z$.

If $\mathbf P=\mathbf M\mathbf N$ is a standard parabolic subgroup of $G$, then $\mathbf M\cap \mathbf B$ is a minimal parabolic subgroup of $\mathbf M$. If $\Phi_M$ denotes the set of roots of  $\mathbf T$ in the Lie algebra of $\mathbf M$, with respect to $\mathbf M\cap \mathbf B$  we have   $\Phi_M^+=\Phi_M\cap \Phi^+ $ and $\Delta_M=\Phi_M\cap \Delta$.  We also write  $\Delta_P$ for  $\Delta_M$ as $P$ and $M$ determine each other, $P=MU$. Thus we obtain a bijection $P\mapsto \Delta_P$ from standard parabolic subgroups of $G$ to subsets of $\Delta$, with $B$ corresponds to $\Phi$ and $G$ to $\Delta$. If $I$ is a subset of $\Delta$, we sometimes denote by $P_I=M_IN_I$ the corresponding standard parabolic subgroup of $G$. If $I=\{\alpha\}$ is a singleton, we write $P_\alpha=M_\alpha  N_\alpha$. We note a few useful properties.  If $P_1$ is another standard parabolic subgroup of $G$, then $P\subset P_1$ if and only if $\Delta_P\subset \Delta_{P_1}$; we have $\Delta_{P\cap P_1}= \Delta_P\cap \Delta_{P_1}$ and the parabolic subgroup corresponding to $\Delta_P\cup \Delta_{P_1}$ is  the subgroup $\langle P,P_1\rangle$  of $G$ generated by $ P $ and $P_1$. The standard parabolic subgroup of $M$ associated to $\Delta_M \cap \Delta_{M_1}$ is $M\cap P_1=(M\cap M_1)(M\cap N_1)$ \cite[Proposition 2.8.9]{MR794307}.
It is convenient to write $G'$ for the subgroup of $G$ generated by the unipotent radicals of the parabolic subgroups; it is also the normal subgroup of $G$ generated by $U$, and we have $G=ZG'$.  For future references, we give now a useful lemma extracted from \cite{MR3600042}:

\begin{lemma}\label{lemma:ZG'}The group $Z\cap  G'$ is generated by the $Z\cap M'_\alpha$,  $\alpha $ running through $\Delta $.
\end{lemma}
\begin{proof}
Take $I = \emptyset$ in \cite[II.6.Proposition]{MR3600042}.
\end{proof}

Let $v_F$ be the normalized valuation of $F$.
For each $\alpha\in  X^*(T)$, the homomorphism $x\mapsto v_F(\alpha(x)):T\to\mathbb Z$ extends uniquely to a homomorphism $Z\to\mathbb Q$ that we denote in the same way.  This defines a homomorphism $Z\xrightarrow{v} X_*(T)\otimes \mathbb  Q$ such that $\alpha(v(z))= v_F(\alpha(z))$ for $z\in Z, \alpha \in X^*(T)$. 

\bigskip An interesting situation occurs when $\Delta=I\sqcup J$ is the union of two orthogonal subsets $I$ and $J$. In that case, $G'=M'_IM'_J$, $M'_I$ and $M'_J$ commute with each other, and their intersection is finite and central in $G$  \cite[II.7 Remark 4]{MR3600042}.

\subsection{$I_G(P,\sigma, Q)$ and minimality}\label{S:2.22} We recall from \cite{MR3600042} the construction of $I_G(P,\sigma,Q)$, our main object of study.  

%
%
Let $\sigma$ be an $R$-representation of $M$ and $P(\sigma)$ be the standard parabolic subgroup with
\[
\Delta_{P(\sigma)}=\{\alpha \in \Delta\setminus \Delta_P\ | \ Z\cap M'_\alpha \ \text{acts trivially on } \ \sigma\}\cup \Delta_P.  
\]
This is the largest parabolic subgroup $P(\sigma)$ containing $P$ to which $\sigma$ extends, here $N\subset P$ acts on $\sigma$ trivially.
Clearly when $P\subset Q \subset P(\sigma)$, $\sigma$ extends to $Q$ and the extension is denoted by $e_Q(\sigma)$.
The restriction of $e_{P(\sigma)}(\sigma)$ to $Q$ is $e_Q(\sigma)$.
If there is no risk of ambiguity, we write
\[
e(\sigma)=e_{P(\sigma)}(\sigma).
\]
\begin{definition}\label{def:Gtriple} An {\bf $R[G]$-triple} is a triple $(P,\sigma,Q)$ made out of  a standard parabolic subgroup $P=MN$ of $G$, a smooth $R$-representation of $M$, and a parabolic subgroup $Q$ of $G$ with $P\subset Q \subset P(\sigma)$.  To an $R[G]$-triple $(P,\sigma,Q)$ is associated a smooth  $R$-representation of $G$:  
 $$I_G(P,\sigma,Q) =\Ind_{P(\sigma)}^G(e(\sigma)\otimes \St_Q^{P(\sigma)})$$
where $\St_Q^{P(\sigma)}$ is the quotient of $ \Ind_Q^{P(\sigma)}\charone$, $\charone$ denoting the trivial $R$-representation of   $Q$,  by the sum  of its subrepresentations $\Ind_{Q'}^{P(\sigma)}\charone$, the sum being over the set  of  parabolic subgroups $Q'$ of $G$ with $Q\subsetneq Q'\subset P(\sigma)$. 
\end{definition}

Note that $I_G(P,\sigma,Q) $ is naturally isomorphic to the quotient of $\Ind_{Q}^G(e_Q(\sigma))$ by the sum of its subrepresentations $\Ind_{Q'}^G(e_{Q'}(\sigma))$ for $Q\subsetneq Q'\subset P(\sigma)$ by Lemma 2.5.  

 \bigskip It might happen that $\sigma$ itself has the form $e_{P}(\sigma_1)$ for some standard parabolic subgroup $P_1=M_1N_1$ contained in $P$  and some $R$-representation $\sigma_1$ of $M_1$. 
 In that case, $P(\sigma_1)=P(\sigma)$ and $e(\sigma)=e(\sigma_1)$. We say that $\sigma$   is  {\bf $e$-minimal} if  $\sigma=e_{P}(\sigma_1)$ implies $P_1=P, \sigma_1=\sigma$.  
 
 \begin{lemma}[{\cite[Lemma~2.9]{arXiv:1703.05599}}]
 \label{lemma:min}Let $P=MN $ be a standard parabolic subgroup of $G$ and  let $\sigma$ be an  $R$-representation of $M$. There exists a unique standard parabolic subgroup  $P_{\min,\sigma}=M_{\min,\sigma}N_{\min,\sigma}$  of $G$ and a unique  $e$-minimal    representation of $\sigma_{\min}$ of $M_{\min,\sigma}$ with $\sigma= e_P(\sigma_{\min})$. Moreover $P(\sigma)=P(\sigma_{\min})$ and $e(\sigma)=e(\sigma_{\min})$.
  \end{lemma}

\begin{lemma} \label{lemma:2.2}Let $P=MN $ be a standard parabolic subgroup of $G$ and  $\sigma$ an $e$-minimal $R$-representation of $M$. Then $\Delta_P$ and $\Delta_{P(\sigma)}\setminus \Delta_P$ are orthogonal. 
\end{lemma}
 
 That comes from \cite[II.7 Corollary~2]{MR3600042}. That corollary of loc.\ cit.\ also shows that when $R$ is a field and $\sigma$ is supercuspidal, then $\sigma$ is $e$-minimal.
 Lemma \ref{lemma:2.2} shows that $\Delta_{P_{\min,\sigma}}$ and $\Delta_{P(\sigma_{\min})}\setminus\Delta_{P_{\min,\sigma}}$ are orthogonal. 
 
 Note that when   $\Delta_P$ and $\Delta_\sigma$ are orthogonal of union $\Delta=\Delta_P\sqcup \Delta_\sigma$, then $G=P(\sigma)= M M'_{ \sigma}$ and $e(\sigma) $ is the $R$-representation of $G$ simply obtained by extending $\sigma$ trivially on $M'_{ \sigma}$.

\begin{lemma}[{\cite[Lemma~2.11]{arXiv:1703.05599}}]
 \label{lemma:2.3} Let $(P,\sigma,Q)$ be an $R[G]$-triple.  Then $(P_{\min,\sigma},\sigma_{\min},Q)$ is an  $R[G]$-triple
  and $I_G(P,l\sigma,Q) =I_G(P_{\min,\sigma},\sigma_{\min},Q)$.
\end{lemma}

  \subsection{Pro-$p$ Iwahori Hecke algebras}\label{S:2.3}
 We fix a  special parahoric subgroup $\mathcal K$ of $G$ fixing a special vertex $x_0$ in the apartment $\mathcal A$ associated to $T$ in the  Bruhat-Tits building of the adjoint group of $G$.    We let $\mathcal B$  be the  Iwahori subgroup fixing the alcove $\mathcal C$ in  $\mathcal A$ with vertex $x_0$ contained in the Weyl chamber (of vertex $x_0$) associated to $B$. We let $\mathcal U$ be the pro-$p$ radical of $\mathcal B$ (the   pro-$p$ Iwahori subgroup).  The pro-$p$ Iwahori Hecke ring $\mathcal{H}=\mathcal{H}(G,\mathcal U) $ is the convolution ring of compactly supported functions $G\to \mathbb Z$ constant on the double classes of $G$ modulo $\mathcal U$. We denote by $T(g)$ the characteristic function of $\mathcal{U} g \mathcal{U}$ for $g\in G$, seen as an element of $\mathcal{H}$. Let $R$ be a commutative ring. The pro-$p$ Iwahori Hecke $R$-algebra $\mathcal{H}_{M,R}$ is $R\otimes_{\mathbb Z}\mathcal{H}_{M}$. We will follow the custom to still denote by $h$ the natural image $1\otimes h$ of $h\in \mathcal{H}$ in  $\mathcal{H}_R$.

  For $P=MN$ a standard parabolic subgroup of $G$,  the similar  objects  for  $M $ are indexed by $M$, we have $\mathcal K_M=\mathcal K\cap M, \mathcal B_M=\mathcal B\cap M, \mathcal U_M=\mathcal U\cap M$,  the pro-$p$ Iwahori Hecke ring $\mathcal{H}_M=\mathcal{H}(M,\mathcal{U}_M)  $, $T^M(m)$ the characteristic function of $\mathcal{U}_M m \mathcal{U}_M$ for $m\in M$, seen as an element of  
  $\mathcal{H}_M $. 
    The pro-$p$ Iwahori group $\mathcal{U}$ of $G$ satisfies the  Iwahori decomposition with respect to $P$:  $$\mathcal U= \mathcal{U}_N\mathcal U_M\mathcal U_{\overline N} ,$$ 
 where $\mathcal{U}_N= \mathcal{U} \cap N, \mathcal U_{\overline N}=\mathcal U\cap \overline N$.
 The linear map
   \begin{equation}\label{eq:theta} \mathcal{H}_{M }\xrightarrow{\theta} \mathcal{H}, \quad  \theta(T^{M}(m))=T(m) \quad (m\in M)\end{equation}  does not respect the product. But if we introduce the monoid $M^+$ of elements $m\in M$ contracting $\mathcal U_N$,  meaning $m\mathcal U_N m^{-1}\subset \mathcal U_N $, and the submodule  $ \mathcal{H}_{M^+ }\subset \mathcal{H}_M$ of  functions with support in $M^+$, we have \cite[Theorem~1.4]{MR3437789}:
   
\bigskip {\sl  $  \mathcal{H}_{M^+ }$ is a subring of $\mathcal{H}_M$ and $\mathcal{H}_M$ is the localization of $ \mathcal{H}_{M^+ }$ at an element $\tau^M\in \mathcal{H}_{M^+}$ central and invertible in $\mathcal{H}_M$, meaning     $\mathcal{H}_M= \cup_{n\in \mathbb N} \mathcal{H}_{M^+}(\tau^M)^{-n}$. The map $ \mathcal{H}_{M }\xrightarrow{\theta} \mathcal{H}$ is injective and its  restriction $\theta|_{ \mathcal{H}_{M^+ }}$  to  $\mathcal{H}_{M^+ }$ respects the product. 

 These  properties  are also true when $(M^+, \tau^M)$ is  replaced by its inverse   $(M^-, (\tau^M)^{-1})$ where $M^-=\{m^{-1}\in M \ | \ m\in M^+\}$.}


\section{Pro-$p$ Iwahori invariants of $I_G(P,\sigma,Q)$}\label{S:8}
    
\subsection{Pro-$p$ Iwahori Hecke algebras: structures}

We supplement here the notations of \S \ref{S:2.1} and \S \ref{S:2.3}.  The subgroups $Z^0=Z\cap \mathcal K= Z\cap \mathcal B$ and $Z^1=Z\cap \mathcal U$ are normal in $\mathcal N$ and we put
$$W= \mathcal N/ Z^0, \ W(1)= \mathcal N/ Z^1,\  \Lambda= Z/ Z^0, \  \Lambda(1)= Z/ Z^1, \ Z_k=Z^0/Z^1.$$
We have $\mathcal N = (\mathcal N\cap \mathcal K) Z$ so that we see the finite Weyl group $\mathbb W=\mathcal N/Z$ as the subgroup $(\mathcal N\cap \mathcal K)/Z^0$ of $W$; in this way $W$ is the semi-direct product $\Lambda \rtimes \mathbb W$.  The   image $W_{G'}=W'$ of $\mathcal N\cap G'$ in $W$ is an affine Weyl group generated by  the set $S^{\aff}$ of affine reflections determined by the walls of the alcove $\mathcal C$. The group $W'$ is normal in $W$ and  $W$  is the semi-direct product $W'\rtimes \Omega$ where $\Omega$ is the image in $W$ of the normalizer $\mathcal N_{\mathcal C}$ of  $\mathcal C$ in $\mathcal N$.  The length function $\ell$ on the affine Weyl system $(W',S^{\aff})$ extends to a length function on $W$ such that $\Omega$ is the set of elements of length $0$. We also view $\ell$ as a function of $W(1)$ via the quotient map $W(1)\to W$.
We write
\begin{equation}\label{eq:hattildew}(\hat w,\tilde w, w)\in \mathcal N\times W(1) \times W \ \text{corresponding via the quotient maps } \ \mathcal N\to W(1)\to W. 
\end{equation}
When  $w=s$ in $S^{\aff}$ or more generally $w$ in $W_{G'}$, we will most of the time  choose $\hat w$ in $ \mathcal N \cap G' $ and $\tilde w$ in    the image ${}_1W_{G'}$ of $\mathcal N\cap G'$ in $W(1)$.

\bigskip We are now ready to describe the pro-$p$ Iwahori Hecke ring $\mathcal{H}=\mathcal{H}(G,\mathcal{U})$ \cite{MR3484112}. We have $G=\mathcal{U} \mathcal N \mathcal{U}$ and 
for $n,n'\in \mathcal N$ we have $\mathcal{U} n \mathcal{U}=\mathcal{U} n' \mathcal{U}$ if and only if $nZ^1=n'Z^1$.  For  $n\in \mathcal N$ of image $w\in W(1)$ and  $g\in \mathcal{U} n \mathcal{U}$ we denote $T_w=T(n)=T(g)$ in  $\mathcal{H}$. The relations among the basis elements $(T_w)_{w\in W(1)}$ of $\mathcal{H}$  are:
\begin{enumerate}
\item Braid relations : $T_w T_{w'}=T_{ww'}$ for $w,w'\in W(1)$ with $ \ell(ww')=\ell(w)+\ell(w')$.
\item Quadratic relations : $T_{\tilde s}^2= q_sT_{\widetilde{s}^2} + c_{\tilde s} T_{\tilde s}$
\end{enumerate}
for $\tilde s\in W(1)$ lifting  $s\in S^{\aff}$, where $q_s=q_G(s)=|\mathcal{U}/\mathcal{U}\cap \hat s \mathcal{U}  (\hat s)^{-1}|$ depends only on $s$, and $ c_{\tilde s}=\sum_{t\in Z_k} c_{\tilde s}(t) T_t$  for integers $c_{\tilde s}(t)\in \mathbb N$ summing to $q_s-1$.  

We shall   need the   basis elements $(T^*_w)_{w\in W(1)}$  of $\mathcal{H}$ defined by:
\begin{enumerate}
\item $T^*_w=T_w$  for $w \in W(1)$ of length $\ell(w)=0$.
\item $T^*_{\tilde s}=T_{\tilde s}- c_{\tilde s}$ for $\tilde s\in W(1)$ lifting $s\in S^{\aff}$.
\item $T^*_{ww'}=T^*_wT^*_{w'}$  for $w,w'\in W(1)$ with $ \ell(ww')=\ell(w)+\ell(w')$.
\end{enumerate}\bigskip We need more notation for the  definition of the admissible  lifts of $S^{\aff}$ in $\mathcal N_G$.  Let  $s\in S^{\aff}$ fixing a face  $\mathcal C_s$ of the alcove $\mathcal C$ and $\mathcal K_s$  the parahoric subgroup of $G$ fixing $\mathcal C_s$. The theory of Bruhat-Tits associates to $\mathcal C_s$   a certain  root  $\alpha_s\in \Phi^+$  \cite[\S 4.2]{MR3484112}.  We consider  
the group  $G'_s$ generated by  $U_{\alpha_s}\cup U_{-\alpha_s}$ where  $U_{\pm \alpha_s}$  the root subgroup of $\pm \alpha_s$   (if $2 \alpha_s\in \Phi$, then $U_{2 \alpha_s}\subset  U_{ \alpha_s}$) and the group $\mathcal G'_s$  generated by $\mathcal{U}_{\alpha_s} \cup \,\mathcal{U}_{-\alpha_s}$ where
 $\mathcal{U}_{\pm \alpha_s}= U_{\pm \alpha_s}\cap \mathcal K_s$.  When $u\in  \mathcal{U}_{ \alpha_s}-\{1\}$, the intersection  $\mathcal N_G \cap \mathcal{U}_{-\alpha_s}u \,\mathcal{U}_{-\alpha_s}$    (equal to $\mathcal N_G \cap U_{-\alpha_s}u U_{-\alpha_s}$  \cite[6.2.1 (V5)]{MR0327923} \cite[\S 3.3 (19)]{MR3484112})  possesses  a  single element $n_s(u)$. 
The group $Z'_s=Z\cap \mathcal G'_s$ is  contained in  $Z\cap \mathcal K_s=Z^0$; its image in $Z_k$ is denoted by $Z'_{k,s}$. 

The elements   $n_s(u)$  for  $u\in  \mathcal{U}_{ \alpha_s}-\{1\}$ are   the admissible lifts of $s$ in $\mathcal N_G$; their images in $W(1)$ are  the admissible lifts of $s$ in $W(1)$. 
By  \cite[Theorem  2.2, Proposition 4.4]{MR3484112},
when $\tilde s\in W(1)$ is an admissible lift of $s$, $c_{\tilde s}(t)=0$ if $t\in Z_k\setminus Z'_{k,s}$, and 
\begin{equation}\label{eq:cs}
c_{\tilde s}\equiv (q_s-1)|Z'_{k,s}|^{-1}\sum _{t\in Z'_{k,s}} T_t \quad \mod  p.
\end{equation}
The admissible lifts  of   $ S$ in $\mathcal N_G$ are contained in $\mathcal N_G \cap \mathcal K$ because  $\mathcal K_s \subset \mathcal K$ when $s\in S$. 

\begin{definition}\label{defn:admissible_lift} An admissible lift of the finite Weyl group $\mathbb W$  in  $\mathcal N_G $ is a map $w\mapsto \hat w:\mathbb W\to \mathcal N_G \cap \mathcal K$ such that $\hat s$ is admissible for all $s\in S$ and $\hat w=\hat w_1 \hat w_2$ for $w_1,w_2\in \mathbb W$ such that  $w=w_1 w_2$ and $\ell(w)=\ell(w_1)+\ell(w_2)$.
 \end{definition}
 Any choice of admissible lifts   of $S $ in $\mathcal N_G\cap \mathcal K$ extends  uniquely to an admissible lift of $\mathbb W$   (\cite[IV.6]{MR3600042}, \cite[Proposition 2.7]{arXiv:1703.04921}).
   
\bigskip Let $P=MN$ be a standard parabolic subgroup of $G$.   The groups $Z, Z^0=Z\cap \mathcal K_M=Z\cap \mathcal B_M, Z^1=Z\cap \mathcal U_M$ are the same for $G$ and $M$, but $\mathcal N _M= \mathcal N\cap M$  and $M\cap G'$ are subgroups of $ \mathcal N$ and $G'$. The monoid $M^+$ (\S \ref{S:2.3}) contains  $(\mathcal N _M\cap \mathcal K)$ and is equal to 
$M^+=\mathcal{U}_M \mathcal N _{M^+}\mathcal{U}_M$ where $\mathcal N _{M^+}= \mathcal N\cap M^+$. An element $z\in Z$ belongs to $M^+$ if and only if $v_F(\alpha(z))\geq 0$ for all $\alpha \in \Phi^+\setminus \Phi_M^+$  (see \cite[Lemme 2.2]{MR3437789}).
Put $W_M = \mathcal{N}_M/Z^0$ and $W_M(1) = \mathcal{N}/Z^1$.

Let $\epsilon=+$ or $\epsilon=-$. We denote by $W_{M^\epsilon}$  the images of $\mathcal N _{M^\epsilon}$ in $W_M, W_M(1)$.
We see the groups $W_M, W_M(1), {}_1W_{M'}$ as subgroups of $W, W(1), {}_1W_{G'}$. 
As $\theta$ (\S \ref{S:2.3}), the linear injective map 
\begin{equation}\label{eq:theta*} \mathcal{H}_{M }\xrightarrow{\theta^*} \mathcal{H}, \quad  \theta^*(T_w^{M,*})= T^*_w, \quad (w\in W_M(1)),
\end{equation} respects the product on the subring  $ \mathcal{H}_{M^\epsilon } $.  
 Note that $\theta$ and $\theta^*$ satisfy the obvious transitivity property with respect to a change of parabolic subgroups.

\subsection{Orthogonal case}\label{S:8.0}
Let us examine the case where $\Delta_M$ and $\Delta\setminus \Delta_M$ are orthogonal, writing $M_2=M_{\Delta\setminus \Delta_M}$ as in \S \ref{S:1.3}.

From $M\cap M_2=Z$ we get $W_M\cap W_{M_2}=\Lambda, W_M(1)\cap W_{M_2}(1)= \Lambda(1)$, the semisimple building of $G$ is the product of those of $M$ and $M_2$ and $S^{\aff}$ is the disjoint union of $S_M^{\aff}$ and $S_{M_2}^{\aff}$, the group $W_{G'}$ is the direct product of  $W_{M'}$ and  $W_{M_2'}$. 
For $\tilde s\in W_M(1)$ lifting $s\in S_M^{\aff}$, the elements $T^M_{\tilde s}\in \mathcal{H}_M$ and $T_{\tilde s}\in \mathcal{H}$ satisfy the same quadratic relations.  
  A word of caution is necessary for the lengths $\ell_M$ of $W_M $ and  $\ell_{M_2}$ of $W_{M_2} $ different from  the restrictions   of the length $\ell$ of $W_M$, for example  $\ell_M(\lambda)=0$  for $\lambda \in \Lambda \cap W_{M_2'}$.  
  
 \begin{lemma}We have $\Lambda =(W_{M^\epsilon} \cap  \Lambda) (W_{M_2'}\cap  \Lambda)$.
 \end{lemma} 
  \begin{proof} We prove the lemma for $\epsilon=-$. The case $\epsilon=+$ is similar. 
  The map $v:Z\to X_*(T)\otimes \mathbb Q$ defined in \S \ref{S:2.1} is trivial on $Z^0$   and we also write $v$ for the resulting homomorphism on $\Lambda$. 
  For  $\lambda  \in \Lambda$ there exists $\lambda_2\in W_{M'_2}\cap \Lambda $ such that $\lambda \lambda_2\in W_{M^-}$, or equivalently $\alpha (v(\lambda \lambda_2))\leq 0$
  for all $\alpha \in \Phi^+\setminus \Phi_M^+ =\Phi_{M_2}^+$.   It suffices to have the inequality for  $\alpha \in \Delta_{M_2}$.  The matrix $(\alpha(\beta^\vee))_{\alpha, \beta\in \Delta_{M_2}}$ is invertible, hence there exist  $n_\beta \in \mathbb Z$ such that $\sum_{\beta \in \Delta_{M_2}} n_\beta \alpha(\beta^\vee) \leq - \alpha (v(\lambda))$ for all $\alpha \in \Delta_{M_2}$.    As $v (W_{M'_2}\cap \Lambda)$ contains  $ \oplus _{\alpha \in \Delta_{M_2}} \mathbb Z \alpha^\vee$ where $\alpha^\vee$ is the coroot of $\alpha$ \cite[after  formula (71)]{MR3484112}, there exists $\lambda_2\in W_{M'_2}\cap \Lambda$ with $v(\lambda_2) = \sum_{\beta\in\Delta_{M_2}}n_\beta\beta^\vee$.
 \end{proof} 
  
  The  groups 
 $\mathcal  N\cap M'$ and $ \mathcal  N\cap M_2'$ are normal in $\mathcal N$, and $\mathcal N= (\mathcal N\cap M') \mathcal N_{\mathcal C}( \mathcal  N\cap M_2') =Z (\mathcal N\cap M')( \mathcal  N\cap M_2')$, and  
 $$W= W_{M'}\Omega W_{M_2'}=W_MW_{M_2'}=W_{M^+} W_{M_2'}= W_{M^-} W_{M_2'}$$
The first two equalities are clear, the equality $W_MW_{M_2'}=W_{M^\epsilon} W_{M_2'}$ follows from 
  $W_M= \mathbb W_M \Lambda $,  $\mathbb W_M \subset W_{M^\epsilon}$ and  the lemma.
 The inverse image in $W(1)$ of these groups are
     \begin{equation}\label{Wdec}W(1)= {}_1W_{M'}\, \Omega(1)\,  {}_1W_{M_2'}=   W_M(1) \, {}_1W_{M_2'}= W_{M^+}(1) \, {}_1W_{M_2'}= W_{M^-}(1) \, {}_1W_{M_2'}.
 \end{equation}

  We  recall  the function $q_G(n)=q(n)= | \mathcal U /( \mathcal U \cap  n^{-1}\mathcal U n)|$ on $\mathcal N$ \cite[Proposition 3.38]{MR3484112}  and we extend to  $\mathcal N$ the functions 
   $ q_M$ on $\mathcal N\cap M$ and  $q_{M_2}$ on $\mathcal N\cap M_2$:
\begin{equation}\label{eq:q} q_{M}(n) = | \mathcal U_{M}/( \mathcal U_{M}\cap  n^{-1}\mathcal U_{M} n)| ,\quad  q_{M_2}(n) = | \mathcal U_{M_2}/( \mathcal U_{M_2}\cap  n^{-1}\mathcal U_{M_2} n)| .
\end{equation}
The functions $q, q_M, q_{M_2}$  descend  to   functions on $W(1) $ and on $W$,  also denoted by $q, q_M, q_{M_2}$.

\begin{lemma} \label{lemma:0} Let $n\in \mathcal N$ of image $w\in W$. We have  
\begin{enumerate}
\item $q(n)=q_M(n) q_{M_2}(n)$.
\item $q_M(n)=q_M(n_M)$ if  $n=n_Mn_2$, $n_M\in \mathcal N \cap M,  n_2\in  \mathcal N \cap M'_2$ and similarly when $M$ and $M_2$ are permuted.
\item $ q( w)=1 \Leftrightarrow q_M(\lambda w_{M})=q_{M_2}(\lambda w_{M_2})=1 $, if   $w= \lambda  w_{M}  w_{M_2}, (\lambda, w_{M}, w_{M_2})\in \Lambda\times  \mathbb W_{M}\times \mathbb W_{M_2}$.
\item On the coset $(\mathcal N \cap M'_2) \mathcal N_{\mathcal C} n$, $q_M$ is constant equal to  $q_{M}(n_{M'})$  for any element $n_{M'}\in  M' \cap (\mathcal N \cap M'_2) \mathcal N_{\mathcal C} n$. A similar  result is true when  $M$ and $M_2$ are permuted.
\end{enumerate}\end{lemma}

\begin{proof}   The product map 
\begin{equation}\label{eq:U}Z^1 \prod_{\alpha \in \Phi_{M,red}} \mathcal U_{\alpha} \prod_{\alpha \in \Phi_{M_2,red}}  \mathcal U_{\alpha}\to \mathcal U
\end{equation} with $ \mathcal U_{\alpha}=  U_{\alpha}\cap  \mathcal U $, is a homeomorphism.  
We have   $\mathcal U_{M }=Z^1 \mathcal Y_{M' }$, $\mathcal U_{M'}= (Z^1\cap M') \mathcal Y_{M' }$  where  $\mathcal Y_{M'}=\prod_{\alpha \in \Phi_{M,red}} \mathcal U_{\alpha} $ and $\mathcal N \cap M'_2$ normalizes  $ \mathcal Y_{M'} $. Similar results are true  when $M$ and $M_2$ are permuted, and $\mathcal{U}=\mathcal{U}_{M'}\mathcal{U}_{M_2}=  \mathcal{U}_{M}\mathcal{U}_{M'_2}$. 
 
  Writing $\mathcal N=Z (\mathcal N\cap M')( \mathcal  N\cap M_2')$ (in any order), we see that the product map 
\begin{equation}\label{eq:nUn} Z^1
( \mathcal Y_{M'}\cap  n^{-1}\mathcal Y_{M'}  n )(\mathcal Y_{M'_2}\cap   n^{-1}\mathcal Y_{M'_2} n)\to \mathcal U \cap n^{-1}\mathcal U n
\end{equation} 
is an homeomorphism. The inclusions induce  bijections 
\begin{equation}\label{eq:UY}\mathcal Y_{M'}/( \mathcal Y_{M'}\cap  n^{-1}\mathcal Y_{M'} n)\simeq  \mathcal U_{M' }/ (\mathcal U_{M' }\cap n^{-1}\mathcal U_{M'} n)\simeq   \mathcal U_{M}/( \mathcal U_{M}\cap  n^{-1}\mathcal U_{M} n) ,
\end{equation} 
similarly for $M_2$, 
and also a bijection \begin{align}\label{eq:UnU}\mathcal U / (\mathcal U \cap n^{-1}\mathcal U n)&\simeq \mathcal Y_{M'_2}/ (\mathcal Y_{M'_2}\cap  n^{-1}\mathcal Y_{M'_2}  n)\times   (\mathcal Y_{M'}/( \mathcal Y_{M'}\cap  n^{-1}\mathcal Y_{M'} n). \end{align}
The assertion  (1) in the lemma  follows from  \eqref{eq:UY}, \eqref{eq:UnU}.

 The assertion (2)  follows from \eqref{eq:nUn}; it  implies the assertion (3).

A subgroup of $\mathcal N$ normalizes $\mathcal{U}_M$  if and only if it  normalizes $\mathcal Y_{M'}$ by  \eqref{eq:UY} if and only if $q_M=1$ on this group. The group $\mathcal N \cap M'_2 $ normalizes $\mathcal Y_{M'}$ because the elements of $M'_2$ commute with those of $M'$ and  $q_M$ is trivial on $\mathcal N_{\mathcal C} $ by (2). Therefore  the group  $(\mathcal N \cap M'_2) \mathcal N_{\mathcal C} $ normalizes $\mathcal{U}_M$. The coset  $ (\mathcal N \cap M'_2) \mathcal N_{\mathcal C} n$  contains  an element $n_{M'}\in  M'$.  For  $x\in (\mathcal N \cap M'_2) \mathcal N_{\mathcal C}$,   $ (xn_{M'})^{-1} \mathcal{U} xn_{M'}= n_{M'}^{-1} \mathcal{U} n_{M'}$ hence  $q_{M}(xn_{M'})=q_{M}(n_{M'})$.
 \end{proof}     
 \subsection{Extension of an $\mathcal{H}_M$-module   to $\mathcal{H}$}\label{S:8.2}

 This section is inspired by   similar results for the pro-$p$ Iwahori Hecke algebras over an algebraically closed field field of characteristic $p$~\cite[Proposition 4.16]{arXiv:1406.1003_accepted}. 
We keep the setting of \S  \ref{S:8.0} and we introduce ideals:

\begin{itemize}
\item $\mathcal J_\ell$ (resp.\ $\mathcal J_r$) the  left (resp.\ right) ideal of $\mathcal{H}$ generated by $T_w^*-1_\mathcal{H}$ for all $w\in {}_1W_{M'_2}$,
\item $\mathcal J_{M,\ell}$ (resp.\ $\mathcal J_{M,r}$) the left (resp.\ right) ideal of $\mathcal{H}_M$ generated by $T_\lambda^{M,*}-1_{\mathcal{H}_M}$ for all $\lambda$ in $ {}_1W_{M'_2}\cap W_M(1) = {}_1W_{M'_2}\cap \Lambda(1)$.
\end{itemize}    
\noindent  The next proposition shows that the  ideals  $ \mathcal J_\ell= \mathcal J_{r}$ are equal and similarly $ \mathcal J_{M, \ell}= \mathcal J_{M,r} $. After the proposition, we will drop the indices $\ell$  and $r$.     
  
\begin{proposition}\label{prop:ideal} The ideals $ \mathcal J_\ell$ and $\mathcal J_{r}$ are equal to the submodule $\mathcal J'$ of $\mathcal{H}$ generated by  $T^*_w- T^*_{ww_2}$ for   all $w\in W(1)$ and $w_2\in  {}_1W_{M'_2}$.

The ideals $ \mathcal J_{M, \ell}$ and $  \mathcal J_{M,r} $ are  equal to the submodule $\mathcal J'_M$of $\mathcal{H}_M$ generated by  $T^{M,*}_w- T^{M,*}_{w\lambda_2}$  for  all $w\in W_M(1)$ and $\lambda_2\in  \Lambda (1) \cap {}_1W_{M'_2} $.
 \end{proposition}

  \begin{proof}  (1)  We  prove $\mathcal J_{\ell} =  \mathcal J'$. 
 Let $w\in W(1), w_2\in {}_1W_{M'_2}$. We prove by induction on the length of $w_2$ that $T^*_w(T^*_{w_2}-1)\in  \mathcal J'$. This is obvious when $\ell(w_2)=0$ because $T^*_wT^*_{w_2}=T^*_ {ww_2}$. Assume that $\ell(w_2)=1$ and put $s=w_2$. If $\ell(ws)=\ell(w)+1$, as before $T^*_w(T^*_{s}-1)\in  \mathcal J'$ because $T^*_wT^*_{s}=T^*_ {ws}$. Otherwise $\ell(ws)=\ell(w)-1$ and 
 $T^*_w =T^*_{ws^{-1}}T_s^*$ hence  
 $$T^*_w(T^*_{s}-1)=T^*_{ws^{-1}}(T^*_{s})^2-T^*_w=T^*_{ws^{-1}}(q_s T^*_{s^2} - T^*_{s} c_s)-T^*_w=
q_s T^*_{ws } - T^*_w (c_s+1).
$$
Recalling from \ref{S:2.3} that  $c_s+1= \sum_{t\in Z'_k} c_s(t)T_t$ with $c_s(t)\in \mathbb N$ and $\sum_{t\in Z'_k} c_s(t)=q_s$,  
$$q_s T^*_{ws } - T^*_w (c_s+1)= \sum_{t\in Z'_k} c_s(t)(T^*_{ws } - T^*_wT^*_t)= \sum_{t\in Z'_k} c_s(t)(T^*_{ws } - T^*_{ws  s^{-1}t}) \in  \mathcal J'.
$$
Assume  now that $\ell(w_2)>1$. Then, we factorize $w_2=xy$  with $x,y \in {}_1W_{M_2}$ of length $\ell(x), \ell(y)< \ell(w_2)$ and $ \ell(w_2)=\ell(x)+\ell(y)$. The element   $T^*_w(T^*_{w_2}-1)=T^*_w T^*_x(T^*_y-1)+T^*_w(T^*_x-1)$ lies in $\mathcal J'$ by induction.

Conversely, we prove  $T^*_{w w_2}-T^*_w\in \mathcal J_{\ell}$. We factorize $w=xy$ with $y\in {}_1W_{M_2}$ and $x\in {}_1W_{M'} \Omega(1)$. Then, we have $\ell(w)=\ell(x)+\ell(y)$ and $\ell(ww_2)=\ell(x)+\ell(yw_2)$. Hence
$$T^*_{ww_2}-T^*_w=T_x^*(T^*_{yw_2}-T^*_y)=T_x^*(T^*_{yw_2}-1)-T_x^*(T^*_{y }-1) \in  \mathcal J_{\ell}.$$
 This ends the proof of $\mathcal J_{\ell} = \mathcal J'$. 

By the same argument,  the right ideal $\mathcal J_{r}$ of $\mathcal{H}$   is equal to the  submodule  of $\mathcal{H}$ generated by  $T^*_{ w_2w }-T^*_w$ for all  $w\in W(1)$ and $ w_2\in {}_1W_{M'_2}$. But this latter submodule is equal to $\mathcal J'$ because ${}_1W_{M'_2}$ is normal in $W(1)$.
 Therefore we proved $\mathcal J'= \mathcal J_{r}=\mathcal J_{\ell} $.
 
(2) Proof of the second  assertion.  We prove $ \mathcal J_{M,\ell}= \mathcal J'_M$. The proof is easier than in (1)  because for $w\in W_M(1)$ and  $\lambda_2\in {}_1W_{M'_2}\cap  \Lambda(1)$, we have  $\ell(w \lambda_2)=\ell(w)+\ell(\lambda_2)$ hence $T^{M,*}_w (T^{M,*} _{ \lambda_2}-1)= T^{M,*}_{w \lambda_2}-T^{M,*}_w$. We have also   $\ell( \lambda_2 w)=\ell(\lambda_2)+\ell(w)$ hence $(T^{M,*} _{ \lambda_2}-1)T^{M,*}_w= T^{M,*}_{ \lambda_2 w}-T^{M,*}_w$ hence  $ \mathcal J_{M,r}$    is equal to the submodule of $\mathcal{H}_M$ generated by  $T^{M,*}_{  \lambda_2 w}-T^{M,*}_w$ for all  $w\in W_M(1)$ and  $\lambda_2\in {}_1W_{M'_2}\cap  \Lambda(1)$. This latter submodule is  $\mathcal J'_M$, 
  as  ${}_1W_{M'_2} \cap \Lambda (1)={}_1W_{M'_2} \cap W_M (1) $ is normal in $W_M(1)$.  Therefore $\mathcal J'_M= \mathcal J_{M,r}=\mathcal J_{M,\ell}$. 
  \end{proof}

By Proposition \ref{prop:ideal},  a basis of $\mathcal J$ is $T^*_w- T^*_{ww_2}$ for $w $ in  a system of representatives of $W(1)/  {}_1W_{M'_2}$, and $w_2\in  {}_1W_{M'_2} \setminus \{1\}$. Similarly a basis of $\mathcal J_M$ is $T^{M,*}_w- T^{M,*}_{w\lambda_2}$ for $w$ in a system of representatives of $W_M(1)/ ( \Lambda (1) \cap {}_1W_{M'_2})$.  and  $\lambda_2\in (\Lambda (1) \cap {}_1W_{M'_2} )\setminus \{1\}$.

\begin{proposition} \label{prop:ext} The natural ring  inclusion  of $\mathcal{H}_{M^- }$ in $\mathcal{H}_{M}$ and the  ring inclusion of   $\mathcal{H}_{M^- }$ in $\mathcal{H}$ via $\theta^*$ induce  ring isomorphisms
$$  \mathcal{H}_{M}/\mathcal J_M \xleftarrow{\sim} \mathcal{H}_{M^-}/(\mathcal J_M \cap \mathcal{H}_{M^-} )\xrightarrow{\sim}  \mathcal{H}/\mathcal J.$$
\end{proposition}

\begin{proof}  (1) The left map  is   obviously injective. We prove the surjectivity.  Let  $ w\in W_M(1)$.  Let $ \lambda_2 \in {}_1W_{M'_2}\cap \Lambda(1)$ such that $ w \lambda_2^{-1}\in W_{M^-}(1)$ (see \eqref{Wdec}). We have  $T^{M,*}_{ w \lambda_2^{-1}}\in \mathcal{H}_{M^-}$ and $T^{M,*}_{ w}= T^{M,*}_{ w \lambda_2^{-1}}T^{M,*}_ {\lambda_2 }= T^{M,*}_{ w \lambda_2^{-1}}+T^{M,*}_{ w \lambda_2^{-1}}(T^{M,*}_{  \lambda_2 }-1)$. Therefore $T^{M,*}_{ w} \in  \mathcal{H}_{M^-}+\mathcal J_M $. As $w$ is arbitrary, $\mathcal{H}_M= \mathcal{H}_{M^-}+\mathcal J_M $.

(2) The right  map   is surjective:
 let $w\in W(1) $ and $w_2\in {}_1W_{M'_2} $ such that $ww_2^{-1} \in W_{M^-}(1)$ (see \eqref{Wdec}).  Then
 $T^{*}_{\hat w}- T^{*}_{w w_2^{-1}  }  \in \mathcal J$ with the same arguments than in (1), using   Proposition \ref{prop:ideal}. Therefore $\mathcal{H}= \theta^*(\mathcal{H}_{M^-}) + \mathcal J$.
 
We prove the injectivity:  $\theta^*(\mathcal{H}_{M^-}) \cap  \mathcal J= \theta^* (\mathcal{H}_{M^-} \cap  \mathcal J_M)$. Let $\sum_{w\in W_{M^-}(1)} c_w T_{w}^{M,*}$, with $c_w\in \mathbb Z$,  be an element of $\mathcal{H}_{M^-}$. Its image by $\theta^*$ is $\sum_{w\in W (1)} c_w T_{w}^{*}$ where we have set $c_w=0$ for $w\in W(1)\setminus W_{M^-}(1)$. We have $\sum_{w\in W (1)} c_w T_{w}^{*} \in \mathcal J$ if and only if 
 $\sum _{w_2\in {}_1W_{M'_2}} c_{ww_2}=0 $ for all $w\in  W(1)$. If $c_{ww_2}\neq 0$ then $w_2 \in {}_1W_{M'_2}\cap W_{M}(1)$, that is,  $w_2  \in {}_1W_{M'_2}\cap \Lambda(1)$. The sum 
 $\sum _{w_2\in {}_1W_{M'_2}} c_{ww_2}$ is equal to $  \sum _{\lambda_2\in {}_1W_{M'_2}\cap \Lambda(1)} c_{w\lambda_2} $. By Proposition \ref{prop:ideal}, $\sum_{w\in W (1)} c_w T_{w}^{*} \in \mathcal J$ if and only if  $\sum_{w\in W_{M^-}(1)} c_w T_{w}^{M,*}\in  \mathcal J_M$. 
\end{proof} 
 
 We  construct a ring isomorphism 
 $$e^*:  \mathcal{H}_{M}/\mathcal J_M \xrightarrow{\sim}    \mathcal{H}/\mathcal J $$ by using  Proposition \ref{prop:ext}.  For any $w\in W(1)$, $T^{ *}_w +\mathcal J=  e^*(T^{M,*}_{w_{M^-}}+\mathcal J_M)$ where  $w_{M^-}\in W_{M^-}(1)\cap w  \ {}_1 W_{M'_2} $ (see \eqref{Wdec}),  because by  Proposition \ref{prop:ideal}, $T^{ *}_w +\mathcal J= T^{ *}_{w _{M^-}} +\mathcal J$ and  $T^{ *}_{w_{M^-}} +\mathcal J= e^*(T^{M,*}_{w_{M^-}}+\mathcal J_M)$ by construction of $e^*$.
We check that  $e^*$   is induced by $\theta^*$:
\begin{theorem} \label{thm:ext0} The  linear  map  $\mathcal{H}_{M  } \xrightarrow{\theta^*}\mathcal{H}$ induces a ring   isomorphism 
$$e^*:  \mathcal{H}_{M}/\mathcal J_M \xrightarrow{\sim}    \mathcal{H}/\mathcal J.$$
\end{theorem}

\begin{proof} Let $w\in W_M(1)$. We have  to show that $T^{ *}_w +\mathcal J=  e^*(T^{M,*}_{w}+\mathcal J_M)$. We saw above that $T^{ *}_w +\mathcal J=  e^*(T^{M,*}_{w_{M^-}}+\mathcal J_M)$ with 
 $w=w_{M^-}\lambda_2$ with $\lambda_2\in {}_1W_{M'_2} \cap W_M(1) $. 
 As $\ell_M(\lambda_2)=0$,  $T^{M,*}_{w}=T^{M,*}_{w_{M^-}} T^{M,*}_{\lambda_2} \in T^{M,*}_{w_{M^-}} +\mathcal J_M$. Therefore
$T^{M,*}_{w_{M^-}}+\mathcal J_M=   T^{M,*}_{w }+\mathcal J_M $, this ends the proof of the theorem.
 \end{proof}

  We wish now to compute $e^*$ in terms of the $T_w$ instead of the $T_w^*$.

 \begin{proposition}\label{prop:extT} Let  $w\in W(1)$. 
 Then, 
 $  T_w+\mathcal J=    
e^*(T^M_{w_M} q_{M_2}(w)+\mathcal J_M),$ 
for any   $w_M\in W_M(1)  \cap w \ {}_1 W_{M'_2}$.
 \end{proposition}

  \begin{proof}  
  The element $w_M$ is unique modulo right multiplication by an element  
  $\lambda_2\in W_M(1)\cap  {}_1 W_{M'_2}  $  of length $\ell_M(\lambda_2)=0$  and $ T^M_{w_M} q_{M_2}(w)+\mathcal J_M $  does not depend on the choice of $w_M$. We choose  a decomposition (see \eqref{Wdec}):
$$w=\tilde s_1\ldots \tilde s_a u \tilde s_{a+1}\ldots \tilde s_{a+b},\quad \ell(w)=a+b,$$
 for $u\in \Omega(1)$, $\tilde s_i \in {}_1W_{M'}$ lifting $s_i\in S_M ^{\aff}$ for $1\leq i \leq a$ and $\tilde s_i \in {}_1W_{M'_2}$  lifting $s_i\in S_{M_2}^{\aff}$  for $a+1\leq i \leq a+b$, and  we choose  $u_M\in W_M(1)$ such that   $u\in u_M \, {}_1 W_{M'_2}$.  Then $$w_M =\tilde s_1\ldots \tilde s_a u_M \in W_M(1)  \cap w \ {}_1 W_{M'_2}$$ and  $q_{M_2}(w)=q_{M_2}(\tilde s_{a+1}  \ldots  \tilde s_{a+b})$ (Lemma \ref{lemma:0} 4)).  
 We check first the proposition in three simple cases:

Case 1. Let $w=\tilde s \in {}_1W_{M'}$  lifting  $s\in S_{M}^{\aff}$; we have  $T_{\tilde s} +\mathcal J= e^*(T^M_{\tilde s} + \mathcal J_M)$ because $T_{\tilde s}^*-e^*(T_{\tilde s}^{M,*})\in  \mathcal J$, $T_{\tilde s}= T_{\tilde s}^* + c_{\tilde s}  $, $T^M_{\tilde s}= T_{\tilde s}^{M,*} + c_{\tilde s} $ and $1= q_{M_2}(\tilde s) $.  

Case 2. Let  $w=u \in W(1)$ of length $\ell(u)=0$ and $u_M\in W_M(1)$ such that   $u\in u_M \, {}_1 W_{M'_2}$. We have $\ell_M(u_M)=0$ and $q_{M_2}(u )=1$ (Lemma \ref{lemma:0}).
We deduce 
$T_u+\mathcal J= e^*(T^M_{u_M} + \mathcal J_M)$ because $ T^*_u+\mathcal J=T^*_{u_M} + \mathcal J= e^*(T_{u_M}^{M,*}+  \mathcal J_M)$,   and $ T_u = T^*_u,T^M_{u_M} = T^{M,*}_{u_M} $.

 Case 3.   Let $w=\tilde s \in {}_1W_{M'_2}$  lifting    $s\in S_{M_2}^{\aff}$; we have  $T_{\tilde s} +\mathcal J= e^*(q_{M_2}(\tilde s) + \mathcal J_M)$ because $T_{\tilde s}^*-1, c_{\tilde s} - (q_s-1) \in \mathcal J$, $T_{\tilde s}= T_{\tilde s}^* + c_{\tilde s} \in q_s +\mathcal J$ and $q_s= q_{M_2}(\tilde s) $.

In general, the braid relations $T_w =T_{\tilde s_1}\ldots \tilde T_{s_a}T_ u T_{\tilde s_{a+1}}\ldots T_{\tilde s_{a+b}}$ give  a similar product decomposition of $T_w+\mathcal J$,  and  the  simple cases 1, 2, 3 imply that  $T_w+\mathcal J$ is equal to \begin{align*} &e^*(T^M_{\tilde s_1}+\mathcal J_M)\ldots e^*(T^M_{\tilde s_a}+\mathcal J_M) e^*(T^M_{u_M} + \mathcal J_M)e^*(q_{M_2}(\tilde s_{a+1}) + \mathcal J_M)\ldots e^*(q_{M_2}(\tilde s_{a+b}) + \mathcal J_M)\\
&=e^*(T^M_{w_M}q_{M_2}(w)+ \mathcal J_M).
\end{align*} 
The proposition is proved.
    \end{proof}

  Propositions  \ref{prop:ideal},  \ref{prop:ext}\ref{prop:extT}, and  Theorem \ref{thm:ext0} are valid over any commutative ring $R $ (instead of $\mathbb Z$). 

The two-sided ideal of $\mathcal{H}_R$ generated by $T_w^*-1$ for all $w\in {}_1W_{M'_2}$ is $\mathcal J_R=\mathcal J \otimes_{\mathbb Z} R$, 
the two-sided ideal of $\mathcal{H}_{M,R}$ generated by $T_\lambda^*-1$ for all $\lambda\in {}_1W_{M'_2}\cap \Lambda(1)$ is $\mathcal J_{M,R}=\mathcal J_M \otimes_{\mathbb Z} R$, and we get as in Proposition \ref{prop:ext} isomorphisms
$$  \mathcal{H}_{M,R}/\mathcal J_{M,R} \xleftarrow{\sim} \mathcal{H}_{M^-,R}/(\mathcal J_{M,R} \cap \mathcal{H}_{M^-,R} )\xrightarrow{\sim}  \mathcal{H}_R /\mathcal J_R,$$
giving an isomorphism $  \mathcal{H}_{M,R}/\mathcal J_{M,R}\to \mathcal{H}_R /\mathcal J_R$ induced by $\theta^*$.  Therefore, we have an isomorphism from the category of  right $\mathcal{H}_{M,R}$-modules  where $\mathcal J_M$ acts by $0$ onto the category of right $\mathcal{H}_R$-modules  where $\mathcal J$ acts by $0$.

\begin{definition}\label{def:ext}  A right $\mathcal{H}_{M,R}$-module $\mathcal V$  where $\mathcal J_M$ acts by $0$ is called extensible to $\mathcal{H}$. The corresponding $\mathcal{H}_R$-module  where $\mathcal J$ acts by $0$ is called its extension to $\mathcal{H}$ and denoted by $e_{\mathcal{H}}(\mathcal V)$ or $e(\mathcal V)$.\end{definition}

With the element basis $T^*_w$,   $\mathcal V$   is extensible to $\mathcal{H}$ if and only if \begin{equation}\label{eq:condition} vT_{\lambda_2}^{M,*}= v \ \text{for all } \  v\in \mathcal V \ \text{and} \ \lambda_2\in {}_1W_{M'_2} \cap \Lambda(1).\end{equation}
The   $\mathcal{H}$-module structure   on the  $R$-module $e(\mathcal V)=\mathcal V$  is determined by  
\begin{equation}\label{eq:structure}  vT_{w_2}^*= v, \quad  vT_w^*= v T^{M,*}_w , \quad \text{for all } \ v\in  \mathcal V,  w_2\in {}_1W_{M'_2}, w\in W_M(1).\end{equation}
It is also determined by the action of $T_w^*$ for $w\in  {}_1W_{M'_2} \cup W_{M^+}(1)$ (or $w\in  {}_1W_{M'_2} \cup  W_{M^-}(1)$). Conversely,  a right $\mathcal{H}$-module  $\mathcal W$ over $R$ is extended from an $\mathcal{H}_M$-module if and only if 
  \begin{equation}\label{eq:structure2}  vT_{w_2}^*= v,
 \quad \text{for all } \ v\in  \mathcal W,  w_2\in {}_1W_{M'_2}.\end{equation}
 
In terms of the  basis elements $T_w$ instead of $T_w^*$,  this says:
 
  \begin{corollary}\label{cor:extT} A  right $\mathcal{H}_M$-module $\mathcal V$ over $R$ is extensible to $\mathcal{H}$ if and only if \begin{equation}\label{eq:conditionT} vT_{\lambda_2}^{M}= v \ \text{for all } \  v\in \mathcal V \ \text{and} \ \lambda_2\in {}_1W_{M'_2} \cap \Lambda(1).\end{equation}
Then,  the  structure of $\mathcal{H}$-module  on the  $R$-module $e(\mathcal V)=\mathcal V$ is determined by
\begin{equation}\label{eq:structureT}  vT_{w_2}= v q_{w_2}, \quad  vT_w= v T^{M}_w q_{M_2}(w), \quad \text{for all } \ v\in  \mathcal{V},  w_2\in {}_1W_{M'_2}, w\in W_M(1).\end{equation}
($W_{M^+}(1)$ or  $W_{M^-}(1)$ instead of $W_M(1)$ is enough.) 
 A right $\mathcal{H}$-module  $\mathcal W$ over $R$ is extended from an $\mathcal{H}_M$-module if and only if 
  \begin{equation}\label{eq:structure2T}  vT_{w_2}= vq_{w_2},
 \quad \text{for all } \ v\in  \mathcal W,  w_2\in {}_1W_{M'_2}.\end{equation}
\end{corollary}

   \subsection{$\sigma^{\mathcal U_M}$ is extensible to $\mathcal{H}$ of extension $e(\sigma^{\mathcal U_M})=e(\sigma)^{\mathcal U}$}\label{S:8.1}
    Let $P=MN$ be a standard parabolic subgroup of $G$ such that 
$\Delta_P $ and $ \Delta\setminus \Delta_P$ are orthogonal,  and  $\sigma$ a smooth $R$-representation of $M$ extensible to $G$.   Let $P_2=M_2N_2$ denote the standard parabolic subgroup of $G$ with  $\Delta_{P_2}= \Delta\setminus \Delta_P $.

Recall that $G=MM'_2$, that  $M\cap M'_2=Z\cap M'_2$ acts trivially on $\sigma $,  $e(\sigma)$ is the representation of $G$ equal to $\sigma$ on $M$ and trivial on $M'_2$.
  We will describe the $\mathcal{H}$-module $e(\sigma)^{\mathcal U}$ in this section. We first consider $e(\sigma)$ as a subrepresentation of $\Ind_P^G\sigma$. 
For $v\in \sigma$, let $f_v\in (\Ind_P^G\sigma)^{M'_2}$ be the unique function with value $v$ on $M'_2$.  Then, the map
\begin{equation}\label{eq:HVt}v\mapsto f_{ v}:\sigma \to \Ind_P^G\sigma 
\end{equation}
 is the natural $G$-equivariant embedding of $e(\sigma)$ in $\Ind_P^G\sigma $. As $\sigma^{\mathcal U_M}=e(\sigma)^{\mathcal U }$ as $R$-modules,  the image of  $e(\sigma)^{\mathcal U }$ in $(\Ind_P^G\sigma)^\mathcal U$ is made out of the $f_v$ for $v\in  \sigma^{\mathcal U_M}$.

 We now recall the explicit description of $(\Ind_P^G\sigma)^\mathcal U$. For each $d\in \mathbb W_{M_2}$, we fix a lift  $\hat d\in {}_1W_{M'_2}$ and for $v\in \sigma^{\mathcal U_M}$ let  $f_{P\hat d\mathcal U, v}\in (\Ind_P^G\sigma)^{\mathcal U}$ for the function with support contained in $P\hat d\mathcal U$ and value $v$ on $\hat d\mathcal U$. As $Z\cap M'_2$ acts trivially on $\sigma$, 
 the function $f_{P\hat d\mathcal U, v}$ does not depend on the choice of the lift $\hat d\in {}_1W_{M'_2}$ of $d$.
  By \cite[Lemma 4.5]{arXiv:1703.04921}: 
  
  \bigskip {\sl The map
 $  \oplus_{d\in \mathbb W_{M_2}}\sigma^{\mathcal U_M} \to (\Ind_P^G\sigma)^{\mathcal U}$ 
given on each $d$-component by $v\mapsto  f_{P\hat d\mathcal U, v}$, is an  $\mathcal{H}_{M^+}$-equivariant isomorphism where $\mathcal{H}_{M^+}$ is seen as a subring of $\mathcal{H}$ via $\theta$, and   induces an $\mathcal{H}_R$-module isomorphism}
\begin{equation}\label{eq:OV}v\otimes h \mapsto  f_{P\mathcal U, v}h:  \sigma^{\mathcal U_M}\otimes_{\mathcal{H}_{M^+}, \theta}\mathcal{H} \to (\Ind_P^G\sigma)^{\mathcal U}.
\end{equation}

In particular  for $v\in \sigma^{\mathcal U _M}$, $v\otimes  T(\hat d)$ does not depend on the choice of the lift $\hat d\in {}_1W_{M'_2}$ of $d$ and 
\begin{equation}\label{eq:ovvP}f_{P\hat d\mathcal U, v} =f_{P\mathcal U, v} T(\hat d).
\end{equation}
 As  $G$ is the disjoint union of $P \hat d \mathcal U$ for $d\in  \mathbb W_{M_2}$, we have $f_v= \sum_{d\in \mathbb W_{M_2}}f_{P\hat d\mathcal U, v}$
 and  $f_v$ is the image of  $v\otimes e _{M_2}$  in \eqref{eq:OV}, where
\begin{equation}\label{eq:e}
e _{M_2}= \sum_{d\in \mathbb W_{M_2}} T(\hat d).
\end{equation}
Recalling \eqref{eq:HVt} we get:
\begin{lemma}\label{lemma:sigmaemb} The map 
$v\mapsto v\otimes e _{M_2}: e( \sigma)^{\mathcal U}\to \sigma^{\mathcal U_M}\otimes_{\mathcal{H}_{M^+}, \theta}\mathcal{H} $ is an $\mathcal{H}_R$-equivariant embedding.
\end{lemma}
\begin{remark}
The trivial map $v\mapsto v\otimes 1_\mathcal{H}$  is not an $\mathcal{H}_R$-equivariant embedding.
\end{remark}

We  describe the action of $T(n)$ on $e( \sigma)^{\mathcal U}$ 
for  $n\in  \mathcal N$. By definition for $v\in e(\sigma)^{\mathcal U}$,  \begin{equation}\label{eq:vTn}vT(n)=\sum_{y\in \mathcal U / (\mathcal U \cap n^{-1}\mathcal U n)}  yn^{-1}v.\end{equation}

 \begin{proposition}\label{prop:eu} We have $vT(n)= vT^M(n_M) q_{M_2}(n)$ for any 
  $n_N\in \mathcal N \cap M$ is such that    $n=n_M  (\mathcal N \cap M'_2)$.
\end{proposition}

\begin{proof}   
The  description \eqref{eq:UnU} of $ \mathcal U / (\mathcal U \cap n^{-1}\mathcal U n)$ gives $$vT(n)=\sum_{y_1\in \mathcal U_{M} / (\mathcal U_{M} \cap n^{-1}\mathcal U_{M} n)}y_1 \sum_{y_2\in \mathcal U_{M'_2} / (\mathcal U_{M'_2} \cap n^{-1}\mathcal U_{M'_2} n)}  y_2n^{-1} v.$$
 As $M'_2$ acts trivially on $e(\sigma)$, we obtain 
\begin{equation*} vT(n)=q_{M_2}(n) \sum_{y_1\in \mathcal U_{M} / (\mathcal U_{M} \cap n^{-1}\mathcal U_{M} n)}y_1n_{M}^{-1} v=q_{M_2}(n) \, v T^M(n_{M}).
\end{equation*} \end{proof}

\begin{theorem} \label{thm:ouf} Let  $\sigma$  be  a   smooth $R$-representation of $M$. If   $P(\sigma)=G$, then   $\sigma^{\mathcal U_M}$ is extensible to $\mathcal{H}$ of extension $e (\sigma^{\mathcal U_M})= e(\sigma)^{\mathcal U}$. Conversely, if  $\sigma^{\mathcal U_M}$ is extensible to $\mathcal{H}$ and generates $\sigma$, then  $P(\sigma)=G$.
\end{theorem}
 
 \begin{proof}  (1) 
 The $\mathcal{H}_M$-module  $\sigma^{\mathcal U_M}$ is extensible to $\mathcal{H}$ if and only if $Z\cap  M'_2 $ acts trivially on 
 $ \sigma^{\mathcal U_M}$.   Indeed, for $ v\in  \sigma^{\mathcal U_M}, z_2\in  Z\cap  M'_2 $,  $$ vT^M(z_2)=\sum_{y\in \mathcal U_M / (\mathcal U_M \cap z_2^{-1}\mathcal U _Mz_2)}  yz_2^{-1}v=\sum_{y\in \mathcal Y_M / (\mathcal Y_M \cap z_2^{-1}\mathcal Y _Mz_2)}  yz_2^{-1}v =z_2^{-1}v,$$
 by \eqref{eq:vTn}, then \eqref{eq:UnU}, then the fact that $z_2^{-1} $  commutes with the elements of  $\mathcal Y_M$. 

 (2)   $P(\sigma)=G$   if and only if   $Z \cap  M'_2 $ acts trivially on  $\sigma$ (the group  $Z \cap  M'_2 $ is generated by $Z \cap \mathcal  M'_\alpha $ for $\alpha \in \Delta_{M_2}$ by Lemma~\ref{lemma:ZG'}). The $R$-submodule $\sigma^{Z \cap  M'_2}$ of elements fixed by $Z \cap  M'_2$ is stable by $M$, because $M=ZM'$, the elements of $M'$ commute with those of $Z \cap  M'_2$ and $Z$ normalizes $Z \cap  M'_2$.

(3)  Apply (1) and (2) to get  the theorem except 
the equality $e (\sigma^{\mathcal U_M})= e(\sigma)^{\mathcal U}$ when 
$P(\sigma)=G$ which follows from  Propositions \ref{prop:eu} and \ref{prop:extT}.  \end{proof}

   Let  ${\bf 1 }_M$ denote  the trivial representation of $M$ over $R$ (or ${\bf 1}$ when there is no ambiguity on $M$).  The right $\mathcal{H}_R$-module $({\bf 1 }_G)^{\mathcal U}={\bf 1 }_{\mathcal{H}}$ (or ${\bf 1 }$  if there is no ambiguity)   is  the trivial right $\mathcal{H}_R$-module:  for $w\in W_M(1)$, $T_w=q_w \id $ and $T^{*}_w=\id$ on ${\bf 1 }_{\mathcal{H}}$.

\begin{example}\label{ex:extrivial}    The $\mathcal{H}$-module $(\Ind_{P}^G  {\bf 1 }) ^{\mathcal U}$ is the extension of the $\mathcal{H}_{M_2}$-module $ (\Ind_{M_2\cap B}^{M_2}  {\bf 1  }) ^{\mathcal U_{M_2}}$.  
Indeed,  the representation $\Ind_{P}^G  {\bf 1  }$ of $G$ is trivial on $N_2$, as $G=M M'_2$ and $N_2\subset M'$ (as $\Phi= \Phi_{M}\cup \Phi_{M_2}$). For $g=mm'_2 $ with $m\in M, m'_2\in M'_2$ and   $n_2 \in N_2 $, we have  $Pgn_2 =Pm'_2n_2 =Pn_2  m'_2=Pm'_2=Pg$. The group $M_2\cap B=  M_2\cap P$ is the standard minimal parabolic subgroup of $M_2$ and $(\Ind_{P}^G  {\bf 1  })|_{M_2}=\Ind_{M_2\cap B}^{M_2}  {\bf 1  }$.   Apply Theorem \ref{thm:ouf}:
\end{example}

  \subsection{The $\mathcal{H}_R$-module $e(\mathcal V)\otimes_R (\Ind_Q^G {\bf 1})^\mathcal U $}\label{S:8.4}
Let $P=MN$ be a standard parabolic subgroup of $G$ such that  $\Delta_P$ and $\Delta\setminus \Delta_P$ are orthogonal,  let $\mathcal V$ be  a right  $\mathcal{H}_{M,R}$-module  which is extensible to $\mathcal{H}_R$ of extension $e(\mathcal{V})$ and let $Q $ be a parabolic subgroup of $G$ containing $P$. 

We define on the  $R$-module $e(\mathcal V)\otimes_R (\Ind_Q^G {\bf 1})^\mathcal U $ a structure of   right  $\mathcal{H}_R$-module:

\begin{proposition}\label{lemma:tensor}
\begin{enumerate}
\item The diagonal action of  $T_w^*$ for $w\in W(1)$ on $e(\mathcal V)\otimes_R (\Ind_Q^G {\bf 1})^\mathcal U $
 defines  a structure of   right  $\mathcal{H}_R$-module. 
\item The action of the $T_w$    is also diagonal and satisfies:
 $$((v\otimes f)T_w, (v\otimes f)T^*_w)= ( vT_{uw_{M'}} \otimes fT_{uw_{M'_2}} , vT^*_{uw_{M'}} \otimes fT^*_{uw_{M'_2}}) ,$$ 
where $w=u w_{M'} w_{M'_2}$ with 
$u\in W(1), \ell(u)=0, w_{M'}\in  {}_1W_{M'}, w_{M'_2}\in  {}_1W_{M'_2}$.\end{enumerate}
\end{proposition}

\begin{proof} If the lemma is true for $P$ it is also true for $Q$, because 
the $R$-module $e(\mathcal V)\otimes_R (\Ind_Q^G {\bf 1})^\mathcal U $
 naturally embedded in $e(\mathcal V)\otimes_R (\Ind_P^G {\bf 1})^\mathcal U $  is stable by the action of  $\mathcal{H}$ defined in the lemma. So, we suppose $Q=P$.
 
Suppose that $T^*_w$ for $w\in W(1)$ acts on $e(\mathcal V)\otimes_R (\Ind_P^G {\bf 1})^\mathcal U $ as in (1). The   braid relations obviously hold. The quadratic relations hold because $T^*_s$ with $s\in {}_1S^{\aff}$,  acts trivially either on $e(\mathcal V)$ or on $(\Ind_P^G {\bf 1})^\mathcal U $.  Indeed,  ${}_1S^{\aff}= {}_1S_M^{\aff} \cup {}_1S_{M_2}^{\aff}$,   $T^*_s$ for  $s\in 
{}_1S_M^{\aff}$,  acts trivially on  $(\Ind_P^G \bf 1)^\mathcal U$ which is   extended from a $\mathcal{H}_{M_2}$-module (Example \ref{ex:extrivial}), and  $T^*_s$ for   $s\in 
{}_1S_{M'_2}^{\aff}$,  acts trivially on  $ e(\mathcal V)$ which is extended from a $\mathcal{H}_{M }$-module.  This proves  (1).
 
We describe now the action of $T_w$ instead of $T_w^*$ on the $\mathcal{H}$-module $e(\mathcal V)\otimes_R (\Ind_Q^G {\bf 1})^\mathcal U $. Let $w\in W(1)$. We write $ w=u w_{M'} w_{M'_2}= u w_{M'_2} w_{M'}$ with $u\in W(1), \ell(u)=0, w_{M'}\in  {}_1W_{M'}, w_{M'_2}\in  {}_1W_{M'_2}$. We have $\ell (w) =\ell(w_{M'})+\ell(w_{M'_2})$ hence $T_w= T_u T_{w_{M'}}T_{w_{M'_2}}$.
 
 For $w=u$, we have $T_u=T^*_u$ and $  (v\otimes f)T _u= (v\otimes f)T^*_u= vT^*_u\otimes f T^*_u=  vT _u\otimes f T _u$.
 
  For $ w=w_{M'}$, $ (v\otimes f)T^*_w= vT^*_w \otimes f$;   in particular for  $s\in 
{}_1S_M^{\aff}$,  $c_s=\sum_{t\in Z_k\cap {}_1W_{M'}} c_s(t) T^*_t$,  we have $ (v\otimes f)T_s= (v\otimes f)(T^*_s+c_s)=v(T^*_s+c_s)\otimes f=vT_s \otimes f$. Hence $ (v\otimes f)T_w= vT_w \otimes f$.

For $w=w_{M'_2}$, we have similarly  $ (v\otimes f)T^*_w= v \otimes fT^*_w$ and $ (v\otimes f)T_w= v \otimes fT_w$.
   \end{proof} 
   
\begin{example} \label{ex:toto} 
 Let  $\mathcal X$ be a right $\mathcal{H}_R$-module. Then  ${\bf 1 }_{\mathcal{H}}\otimes_R \mathcal X$ where  the  $T^{*}_w $ acts  diagonally is a  $\mathcal{H}_R$-module isomorphic to  $\mathcal X$. But the action of the $T_w$ on ${\bf 1 }_{\mathcal{H}}\otimes_R \mathcal X$ is not diagonal. 
\end{example}

It is known \cite{MR3402357} that 
$(\Ind_{Q'}^G{\bf 1})^{\mathcal U}$  and  $(\St_Q^G)^{\mathcal U}$ are free $R$-modules  and that $ (\St_Q^G)^{\mathcal U}$ is  the cokernel of the natural $\mathcal{H}_R$-map
  \begin{align} \label{eq:StQU}\oplus _{Q\subsetneq Q'}(\Ind_{Q'}^G{\bf 1})^{\mathcal U}\to (\Ind_Q^G{\bf 1} )^{\mathcal U}
 \end{align}
  although the   invariant functor $(-) ^{\mathcal U}$ is only left exact. 

\begin{corollary}\label{cor:VS} The diagonal action of  $T_w^*$ for $w\in W(1)$ on $e(\mathcal V)\otimes_R ( \St_Q^G)^\mathcal U $
 defines  a structure of   right  $\mathcal{H}_R$-module satisfying Proposition \ref{lemma:tensor} (2).
\end{corollary}

  \section{Hecke module   $I_\mathcal{H}(P,\mathcal V,Q)$ }  \label{S:9.0}
\subsection{Case $\mathcal{V}$ extensible to $\mathcal{H}$} \label{S:9.1} Let $P=MN$ be a standard parabolic subgroup of $G$ such that  $\Delta_P$ and $\Delta\setminus \Delta_P$ are orthogonal, $\mathcal{V}$ a right $\mathcal{H}_{M,R}$-module extensible to $\mathcal{H}_R$ of extension $e(\mathcal{V})$,  and   $Q$ be a  parabolic subgroup  of $G$ containing $P$.  As $Q$ and $M_Q$ determine each other: $Q=M_QU$, we denote also
$\mathcal{H}_{M_Q}=\mathcal{H}_{Q}$ and $\mathcal{H}_{M_Q,R}=\mathcal{H}_{Q,R}$ when $Q\neq P, G$. When $Q=G$ we  drop $G$  and   we denote $e_\mathcal{H}(\mathcal{V})=e(\mathcal{V})$ when $Q=G$.

\begin{lemma} \label{lemma:trans}  $\mathcal V$ is extensible to an $\mathcal{H}_{Q,R}$-module $e_{\mathcal{H}_{Q}}(\mathcal V)$.
\end{lemma}

\begin{proof}  This is straightforward. By Corollary \ref{cor:extT}, 
  $\mathcal V$  extensible to $\mathcal{H}$ means that  $T^{M,*}(z)$ acts trivially  on  $\mathcal V$ for all $z \in \mathcal N_{M'_{2}}\cap Z$.  
We  have $M_Q=M M'_{2,Q} $ with   $M'_{2,Q}\subset M_Q\cap M'_2 $ and  
 $ \mathcal N_{M'_{2,Q}} \subset \mathcal N_{M'_{2}} $; hence  $T^{M,*}(z)$ acts trivially  on  $\mathcal V$  for all $z \in \mathcal N_{M'_{2,Q}}\cap Z$ meaning that
$\mathcal V$ is extensible to $\mathcal{H}_{Q}$.  \end{proof} 

\begin{remark} 
 We cannot say   that $e_{\mathcal{H}_{Q}}(\mathcal V)$ is extensible to $\mathcal{H}$ of extension $e(\mathcal{V})$ when
the set of roots
 $\Delta_Q$ and $\Delta \setminus \Delta_Q$ are not  orthogonal (Definition \ref{def:ext}).
 \end{remark}  
 Let $Q'$ be an arbitrary  parabolic subgroup  of $G$ containing $Q$. We are going to define a  $\mathcal{H}_R$-embedding  $\Ind_{\mathcal{H}_{Q'}}^\mathcal{H}(e_{\mathcal{H}_{Q'}}(\mathcal V))  \xrightarrow{\iota(Q,Q') } \Ind_{\mathcal{H}_{Q}}^\mathcal{H}(e_{\mathcal{H}_{Q}}(\mathcal V))= e_{\mathcal{H}_{Q}}(\mathcal V)\otimes_{ \mathcal{H}_{M_Q^+},\theta}\mathcal{H}$  defining a  $\mathcal{H}_R$-homomorphism
      $$\oplus_{Q\subsetneq Q'\subset G}  \Ind_{\mathcal \mathcal{H}_{Q'}}^\mathcal{H}( e_{\\mathcal{H}_{Q'}}
(\mathcal V))\to \Ind_{\mathcal{H}_Q}^\mathcal{H} ( e_{\mathcal{H}_Q}(\mathcal V))$$
 of cokernel isomorphic to   $e(\mathcal V)\otimes_R (\St_Q^G)^\mathcal U$. In the extreme case $(Q,Q')=(P,G)$, the $\mathcal{H}_R$-embedding $e(\mathcal{V}) \xrightarrow{\iota(P,G)}\Ind_{\mathcal{H}_{M}}^\mathcal{H}(\mathcal V)$ is  given in the following lemma where  $f_G$ and  $f_{P\mathcal U}\in (\Ind_P^G {\bf 1})^{\mathcal U}$ of $P\mathcal U$  denote the characteristic functions of $G$ and $P\mathcal U$, $f_G=f_{P\mathcal U}e_{M_2}$ (see \eqref{eq:e}).   
 
 \begin{lemma} \label{lemma:ovv}  There is a natural  $\mathcal{H}_R$-isomorphism   $$v\otimes  1 _{\mathcal{H}}\mapsto v\otimes  f_{P\mathcal U} :\Ind_{\mathcal{H}_{M}}^\mathcal{H}(\mathcal V)= \mathcal V\otimes_{ \mathcal{H}_{M^+},\theta}\mathcal{H}\xrightarrow{\kappa_P}  e(\mathcal V)\otimes_R (\Ind_P^G {\bf 1})^\mathcal U, $$ 
 and compatible $\mathcal{H}_R$-embeddings 
 \begin{align*}&v\mapsto v\otimes   f_G : e(\mathcal V) \xrightarrow{} e(\mathcal V)\otimes_R (\Ind_P^G {\bf 1})^\mathcal U,\\
&v\mapsto v\otimes  e_{M_2}: e(\mathcal V) \xrightarrow{\iota(P,G)} \Ind_{\mathcal{H}_{M}}^\mathcal{H}(\mathcal V). \end{align*}  \end{lemma}
\begin{proof} We show  first that   the  map 
\begin{align}\label{eq:f0}v\mapsto v\otimes f_{P\mathcal U}: \mathcal V \xrightarrow{} e(\mathcal V)\otimes_R (\Ind_P^G {\bf 1})^\mathcal U 
\end{align}
   is  $\mathcal{H}_{M^+}$-equivariant. 
Let $ w\in W_{M^+}(1)$. We write $w=u w_{M'} w_{M'_2} $ as in  Lemma \ref{lemma:tensor} (2), so that $f_{P\mathcal U}T_w=     f_{P\mathcal U} T_{uw_{M'_2}}$. We have  $ f_{P\mathcal U} T_{uw_{M'_2}} =f_{P\mathcal U}$ because  ${}_1 W_{M'}\subset  W_{M^+}(1)\cap  W_{M^-}(1)$ hence   $uw_{M'_2}= w w_{M'}^{-1} \in W_{M^+}(1)$ and 
in  $  {\bf 1}_{\mathcal{H}_M}\otimes _{\mathcal{H}_{M^+},\theta}\mathcal{H}$ we have 
$(1\otimes 1 _{\mathcal{H}})  T_{uw_{M'_2}} = 1T^M_{uw_{M'_2}} \otimes 1 _{\mathcal{H}}$, and  $T^M_{uw_{M'_2}}$ acts trivially in   $  {\bf 1}_{\mathcal{H}_M}$ because  $\ell_M(uw_{M'_2})=0$.
 We deduce 
     $(v\otimes f_{P\mathcal U})T_w=   vT_w \otimes f_{P\mathcal U}T_w=
  v T^M_{w}\otimes f_{P\mathcal U}$.

By adjunction \eqref{eq:f0} gives an   $\mathcal{H}_R$-equivariant  linear map
\begin{equation}\label{eq:ovv1}v\otimes 1_{\mathcal{H}}\mapsto v \otimes f_{P\mathcal U}:
\mathcal V\otimes_{ \mathcal{H}_{M^+},\theta}\mathcal{H}\xrightarrow{\kappa_P} e(\mathcal V)\otimes_R (\Ind_P^G {\bf 1})^\mathcal U.
\end{equation}
We prove that $\kappa_P$ is an isomorphism. Recalling $ \hat d \in \mathcal N \cap M'_2, \tilde d \in {}_1 W_{M'_2}$ lift $d$, one knows that \begin{equation}\label{eq:ovv2}{\mathcal V}\otimes_{\mathcal{H}_{M^+},\theta}\mathcal{H}= \oplus _{d\in \mathbb W_{M_2}} {\mathcal V} \otimes T_{\tilde d}, \quad e(\mathcal V)\otimes_R (\Ind_P^G {\bf 1})^\mathcal U=\oplus _{d\in \mathbb W_{M_2}} \mathcal V   \otimes f_{P\hat d \mathcal U},
\end{equation}
where  each summand is isomorphic to $\mathcal V$. The left equality follows from \S 4.1  and Remark 3.7 in \cite{MR3437789}  recalling that  $w\in \mathbb W_{M_2} $ is   of minimal length in its coset $\mathbb W_Mw=w\mathbb W_M$ as $\Delta_M$ and $\Delta_{M_2}$ are orthogonal; for the second  equality   see \S \ref{S:8.1}  \eqref{eq:ovvP}.   We have  $\kappa_P(v\otimes T_{\tilde d})=(v \otimes f_{P\mathcal U})T_{\tilde d} =  v \otimes f_{P\mathcal U} T_{\tilde d} $  (Lemma \ref{lemma:tensor}). Hence   $\kappa_P$  is an isomorphism.

We consider   the composite map   
$$ v\mapsto v\otimes 1 \mapsto  v\otimes f_{P \mathcal U}e_{M_2}  :e(\mathcal V) \to e(\mathcal V)\otimes_R {\bf 1}_{\mathcal{H}} \to e(\mathcal V) \otimes_R (\Ind_P^G {\bf 1})^\mathcal U,$$
where the right map is  the tensor product $e(\mathcal V)\otimes_R - $ of the $\mathcal{H}_R$-equivariant embedding ${\bf 1}_{\mathcal{H}}\to (\Ind_P^G {\bf 1})^\mathcal U$ sending $1_R$ to $ f_{P \mathcal U}e_{M_2} $ (Lemma \ref{lemma:sigmaemb}); this map
  is injective because  $(\Ind_P^G\mathbf{1})^{\mathcal{U}}/\mathbf{1}$ is a free $R$-module;
  it is $\mathcal{H}_R$-equivariant  for the diagonal action of the $T_w^*$ on the tensor products  (Example \ref{ex:toto} for the first map). By compatibility with (1), we get the $\mathcal{H}_R$-equivariant embedding  $v\mapsto v\otimes e_{M_2}: e(\mathcal V) \xrightarrow{\iota(P,G)} \Ind_{\mathcal{H}_{M}}^\mathcal{H}(\mathcal V)$.
  \end{proof} 
 
 For a general $(Q,Q')$ the  $\mathcal{H}_R$-embedding  
 $\Ind_{\mathcal \mathcal{H}_{Q'}}^\mathcal{H}(e_{\mathcal \mathcal{H}_{Q'}}(\mathcal V))  \xrightarrow{\iota(Q,Q')} \Ind_{\mathcal{H}_{Q}}^\mathcal{H}(e_{\mathcal{H}_{Q}}(\mathcal V))$ is given in the next proposition generalizing Lemma \ref{lemma:ovv}.  The element $e_{M_2}$ of $\mathcal{H}_R$ appearing in the definition of  $\iota(P,G')$    is replaced in the definition of $\iota(Q,Q')$ by an element   $\theta_{Q'}(e_{Q} ^{Q'})\in \mathcal{H}_R$ that we define first. 
 
 Until the end of \S \ref{S:9.0}, we fix an admissible lift  $w\mapsto \hat w: \mathbb W\to  \mathcal N \cap \mathcal K$ (Definition \ref{defn:admissible_lift}) and  $\tilde w$ denotes the image of  $\hat w$ in $W(1)$. We denote  $\mathbb W_{M_{Q}}= \mathbb W_Q$   and by ${}^{\mathbb W_Q}\mathbb W $  the  set of $w\in\mathbb W $ of minimal length in their coset $\mathbb W_Qw$. The group $G$ is the disjoint union of  $Q \hat d \mathcal{U}$ for  $d$ running through $ {}^{\mathbb W_Q}\mathbb W $  \cite[Lemma 2.18 (2)]{arXiv:1703.04921}. 
  \begin{equation}\label{eq:QQ'U}Q'\mathcal{U}=\sqcup_{d\in  {}^{\mathbb W_ Q}\mathbb W_{ Q'}} Q \hat d \mathcal{U},
   \end{equation}  
Set 
  \begin{equation}\label{eq:eQQ'}
  e_{Q} ^{Q'} =\sum_{d \in  {}^{\mathbb W_ Q}\mathbb W_{ Q'}} T^{M_{Q'}}_{\tilde d}.
 \end{equation} 
 We write $e_{Q}^{G} =e_Q$.  We have $e_{P} ^{Q} =\sum_{d\in \mathbb W_{M_{2,Q}}} T_{\tilde d}^{M_{Q}}$.
\begin{remark}
 Note that      ${}^{\mathbb W_M}\mathbb W=\mathbb W_{M_2}$  and $e_P= e_{M_2}$, where $M_2$  is the standard Levi subgroup of $G$ with  $\Delta_{M_2}=\Delta\setminus \Delta_{M}$, as $\Delta_M$ and $\Delta\setminus \Delta_{M}$ are orthogonal.   More generally,  $ {}^{\mathbb W_ Q}\mathbb W_{M_{Q'}}= {}^{\mathbb W_{M_{2,Q}}}\mathbb W_{M_{2,Q'}} $  where $M_{2,Q'}=M_2\cap M_{Q'}$.
\end{remark}
 Note that    $e_{Q} ^{Q'} \in \mathcal{H}_{M^+}\cap  \mathcal{H}_{M^-}$. We consider the linear map 
$$\theta_{Q}^{Q'}: \mathcal{H}_{Q }\to \mathcal \mathcal{H}_{Q'} \quad  T_w^{M_Q }\mapsto T_w^{M_{Q '}} \quad (w\in W_{M_Q}(1)).$$
 We write $\theta_{Q}^{G} =\theta_Q$ so $\theta_{Q} (T_w^{M_Q })=T_w $. When $Q=P $ this is the map $\theta$ defined earlier. Similarly we denote by $\theta_{Q}^{Q',*}$ the linear map sending the $T_w^{M_Q,*}$ to  $T_w^{M_{Q '},*}$ and $\theta_{Q}^{G,*} =\theta_Q^*$.  We have
 \begin{equation}\label{eq:eQQ'2}
\theta_{Q'}(e_{Q} ^{Q'}) =\sum_{d \in   {}^{\mathbb W_ Q}\mathbb W_{ Q'}} T _{\tilde d},\quad   \theta_{Q'}(e_{P} ^{Q'} )=\theta_Q(e_{P} ^{Q}) \theta_{Q'} (e_{Q} ^{Q'}).
 \end{equation} 
 
   \begin{proposition}  \label{prop:ovvv}   There exists an $\mathcal{H}_R$-isomorphism  
 \begin{equation}\label{eq:cal}
 v\otimes 1_{\mathcal{H}} \mapsto v\otimes f_{Q\mathcal U}:
 \Ind_{\mathcal{H}_{Q}}^\mathcal{H}( e_{\mathcal{H}_{Q}}(\mathcal V))=
  e_{\mathcal{H}_{Q}}(\mathcal V) \otimes_{ \mathcal{H}_{M_Q^+},\theta} \mathcal{H} \xrightarrow{\kappa_Q}
   e(\mathcal V)\otimes_R (\Ind_Q^G {\bf 1})^\mathcal U,
    \end{equation}
 and  compatible  $\mathcal{H}_R$-embeddings  \begin{align}\label{eq:canonical2}
&v\otimes f_{Q'\mathcal U} \mapsto v\otimes f_{Q'\mathcal U}:  
 e_{\mathcal{H}_{Q'}}(\mathcal V)\otimes_R (\Ind_{Q'}^G {\bf 1})^\mathcal U \to   e_{\mathcal{H}_{Q}}(\mathcal V)\otimes_R (\Ind_Q^G {\bf 1})^\mathcal U,\\
\label{eq:canonical} &v\otimes 1_{\mathcal{H}}  \mapsto v\otimes \theta_{Q'}(e_{Q} ^{Q'}):  \Ind_{\mathcal \mathcal{H}_{Q'}}^\mathcal{H}( e_{\mathcal{H}_{Q'}}
(\mathcal V))  \xrightarrow{\iota(Q,Q') } \Ind_{ \mathcal{H}_{Q}}^\mathcal{H} ( e_{\mathcal{H}_Q}(\mathcal V)).
\end{align}
  \end{proposition}  

  \begin{proof} We have  the $\mathcal{H}_{M_{Q},R}$-embedding   
  $$v\mapsto v\otimes e_{P} ^{Q}  : e_{\mathcal{H}_Q}(\mathcal V)\to  \mathcal V \otimes_{ \mathcal{H}_{M ^+},\theta} \mathcal{H}_Q =   \Ind_{\mathcal{H}_{M}}^{\mathcal{H}_Q} (\mathcal V)$$
 by Lemma \ref{lemma:ovv} (2)    as $\Delta_M$ is orthogonal to $\Delta_{M_{Q}}\setminus \Delta_M$.  Applying the parabolic induction  which is  exact, we get the $\mathcal{H}$-embedding  $$v\otimes  1_{\mathcal{H}}\mapsto v\otimes e_{P} ^{Q} \otimes  1_{\mathcal{H}}:  \Ind_{\mathcal{H}_Q }^{\mathcal{H}}(e_{\mathcal{H}_Q }(\mathcal V))\to  \Ind_{\mathcal{H}_Q }^{\mathcal{H}}( \Ind^{\mathcal{H}_Q }_{\mathcal{H}_M } (\mathcal V)).$$
 Note that $T^{M_{Q}}_{\tilde d} \in \mathcal{H}_{M_{Q}^+} $ for $ d\in  \mathbb W_{M_{Q}} $.   By transitivity of the parabolic induction,  it is equal to the   $\mathcal{H}_R$-embedding  
\begin{equation}\label{eq:embe}v\otimes  1_{\mathcal{H}}\mapsto v\otimes \theta_{Q}(e_{P} ^{Q} ) :\Ind_{\mathcal{H}_Q}^{\mathcal{H}}(e_{\mathcal{H}_Q}(\mathcal V))\to  \Ind_{\mathcal{H}_{M}}^{\mathcal{H}} (\mathcal V).
\end{equation}
On the other  hand we have the $\mathcal{H}_R$-embedding 
\begin{equation}\label{eq:embee}
v\otimes f_{Q\mathcal U} \mapsto  v\otimes \theta_Q(e_P^Q):e(\mathcal V)\otimes_R (\Ind_Q^G {\bf 1})^\mathcal U\to  \Ind_{\mathcal{H}_{M}}^{\mathcal{H}} (\mathcal V)\end{equation}   given by the restriction  to     $e(\mathcal V)\otimes_R (\Ind_Q^G {\bf 1})^\mathcal U$  of the $\mathcal{H}_R$-isomorphism  given in  Lemma \ref{lemma:ovv} (1),  from $e(\mathcal V)\otimes_R (\Ind_P^G {\bf 1})^\mathcal U$ to $\mathcal V\otimes_{ \mathcal{H}_{M^+},\theta}\mathcal{H}$ sending $v\otimes f_{P \mathcal U}$ to $v\otimes 1_{\mathcal{H}}$,  noting that    $v\otimes f_{Q\mathcal U}  =
   (v\otimes f_{P \mathcal U}) \theta_Q ( e_P^Q)$ by  Lemma \ref{lemma:tensor},     $f_{Q \mathcal U}= f_{P \mathcal U}\theta_Q (e_ P^Q)$ and $\theta_Q (e_ P^Q)$  acts trivially on $e(\mathcal V)$ (this is true for $T_{\tilde d}$ for  $\tilde d \in {}_1 W_{M'_2}$).  Comparing the  embeddings \eqref{eq:embe} and \eqref{eq:embee}, we get the  $\mathcal{H}_R$-isomorphism  \eqref{eq:cal}.
   
We can replace $Q$ by $Q'$ in    the $\mathcal{H}_R$-homomorphisms \eqref{eq:cal},  \eqref{eq:embe} and \eqref{eq:embee}. With  \eqref{eq:embe} we see $\Ind_{\mathcal \mathcal{H}_{Q'}}^{\mathcal{H}}(e_{\mathcal \mathcal{H}_{Q'}}(\mathcal V))$ and $\Ind_{\mathcal{H}_Q}^{\mathcal{H}}(e_{\mathcal{H}_Q}(\mathcal V))$ as   $\mathcal{H}_R$-submodules  of $\Ind_{\mathcal{H}_{M}}^{\mathcal{H}} (\mathcal V)$.  As seen in \eqref{eq:eQQ'2} we have 
 $\theta_{Q'}(e_P^{Q'})=\theta_Q(e_{P} ^{Q}) \theta_{Q'} (e_{Q} ^{Q'})$. We deduce  the $\mathcal{H}_R$-embedding \eqref{eq:canonical}. 
 
 By \eqref{eq:ovvP} for $Q$ and \eqref{eq:QQ'U},  
   $$f_{Q'U}=\sum_{d \in  {}^{\mathbb W_ Q}\mathbb W_{ Q'}} f_{QU} T_{\tilde d}=  f_{QU} 
\theta_{Q'} (e_{Q} ^{Q'})$$
  in $(\ind_Q^G {\bf 1})^\mathcal{U}$.   We deduce that the $\mathcal{H}_R$-embedding corresponding to  \eqref{eq:canonical} via  $\kappa_Q$ and $\kappa_{Q'}$   is   the   $\mathcal{H}_R$-embedding  \eqref{eq:canonical2}.
     \end{proof}
   We recall that $\Delta_P$ and $\Delta \setminus \Delta_P$ are orthogonal and that $\mathcal{V}$ is extensible to $\mathcal{H}$ of extension $e(\mathcal{V})$. 
  
     \begin{corollary}\label{cor:IHPVQ} The cokernel of the  $\mathcal{H}_R$-map 
     $$\oplus_{Q\subsetneq Q'\subset G}  \Ind_{\mathcal \mathcal{H}_{Q'}}^\mathcal{H}( e_{\\mathcal{H}_{Q'}}
(\mathcal V))\to \Ind_{\mathcal{H}_{Q}}^\mathcal{H} ( e_{\mathcal{H}_{Q}}(\mathcal V))$$
defined  by the $\iota(Q,Q') $, is isomorphic to   $e(\mathcal V)\otimes_R (\St_Q^G)^\mathcal U$ via $\kappa_Q$.
     \end{corollary}
     
\subsection{Invariants in the tensor product}
We return to the setting where $P=MN$ is a standard  parabolic subgroup of $G$, $\sigma$ is a smooth $R$-representation of $M$  with $P(\sigma)=G$ of extension $e(\sigma)$ to $G$, and $Q$ a parabolic subgroup of $G$ containing $P$. We still assume that $\Delta_P$ and $\Delta\setminus \Delta_P$ are orthogonal.

The $\mathcal{H}_R$-modules $e(\sigma^{\mathcal{U}_M})=e(\sigma)^\mathcal{U}$ are equal  (Theorem \ref{thm:ouf}). We  compute $I_G(P,\sigma,Q)^\mathcal{U}=(e (\sigma)\otimes_R \St_Q^G)^{\mathcal U}$.

\begin{theorem}\label{thm:invariant}
The natural linear maps $e(\sigma)^\mathcal{U}\otimes _R (\Ind_Q^G {\bf 1})^\mathcal{U}\to (e(\sigma)\otimes _R \Ind_Q^G {\bf 1})^\mathcal{U}$ and $e(\sigma)^\mathcal{U}\otimes _R (\St_Q^G)^\mathcal{U}\to (e(\sigma)\otimes _R \St_Q^G)^\mathcal{U}$ are isomorphisms. 
  \end{theorem}
  
\begin{proof} We need some preliminaries. In  \cite{MR3263032}, \cite{MR3402357}, is introduced a finite free $\mathbb Z$-module $\mathfrak M$ (depending on $\Delta_Q$) and a $\mathcal B$-equivariant embedding $\St_Q^G \mathbb Z \xrightarrow{\iota} C_c^\infty(\mathcal B, \mathfrak M)$ (we indicate the coefficient ring in the Steinberg representation)  which induces an isomorphism  $(\St_Q^G \mathbb Z)^{\mathcal B}\simeq  C_c^\infty(\mathcal B, \mathfrak M) ^{\mathcal B}$. 
\begin{lemma}  \label{lemma:facteur}
\begin{enumerate}
\item $(\Ind_Q^G \mathbb Z)^{\mathcal B}$ is a direct factor of  $\Ind_Q^G \mathbb Z $.
\item $(\St_Q^G \mathbb Z)^{\mathcal B}$  is a direct factor of $\St_Q^G \mathbb Z$.
\end{enumerate}\end{lemma}
\begin{proof} (1)  \cite[Example~2.2]{arXiv:1703.05599}.

(2) As  $\mathfrak M$ is a free $\mathbb Z$-module, $C_c^\infty(\mathcal B, \mathfrak M) ^{\mathcal B}$ is a direct factor of  $C_c^\infty(\mathcal B, \mathfrak M) $. Consequently,  $\iota((\St_Q^G \mathbb Z)^{\mathcal B})= C_c^\infty(\mathcal B, \mathfrak M) ^{\mathcal B}$    is a direct factor of  $\iota(\St_Q^G \mathbb Z)$. As $\iota$ is injective, we get (2).
\end{proof}
We  prove now Theorem \ref{thm:invariant}. We may and do assume that $\sigma$ is $e$-minimal (because $P(\sigma)=P(\sigma_{\min}), e(\sigma)=e(\sigma_{\min})$) so that $\Delta_M$ and $  \Delta \setminus \Delta_M$ are orthogonal and we use the same notation  as  in \S \ref{S:8.0} in particular  $M_2= M_{  \Delta \setminus \Delta_M}$.  Let $V$ be the space of  $e(\sigma)$ on which  $M'_2$ acts trivially. 
The restriction of $\Ind_Q^G \mathbb Z $ to $M_2$ is $\Ind_{Q\cap M_2}^{M_2} \mathbb Z $, that of 
  $\St_{Q }^{G} \mathbb Z $  is $\St_{Q\cap M_2}^{M_2} \mathbb Z $. 
  
As in \cite[Example~2.2]{arXiv:1703.05599},$((\Ind_{Q\cap M_2}^{M_2} \mathbb Z)\otimes V)^{\mathcal{U}_{M'_2}}  \simeq (\Ind_{Q\cap M_2}^{M_2} \mathbb Z)^{\mathcal{U}_{M'_2}} \otimes V$.  We have 
$$(\Ind_{Q\cap M_2}^{M_2} \mathbb Z)^{\mathcal{U}_{M'_2}} =(\Ind_{Q\cap M_2}^{M_2} \mathbb Z)^{\mathcal{U}_{M_2}} =(\Ind_{Q}^{G} \mathbb Z)^{\mathcal{U} } .$$
The first equality follows from $M_2=(Q\cap M_2)\mathbb W_{M_2}\mathcal{U}_{M_2}$,   $\mathcal{U}_{M_2}=Z^1\mathcal{U}_{M'_2}$ and $Z^1$ normalizes $\mathcal{U}_{M'_2}$ and is normalized by $\mathbb W_{M_2}$. The second equality follows from  $\mathcal{U}=\mathcal{U}_{M'} \mathcal{U} _{M_2}$ and $\Ind_Q^G \mathbb Z$ is trivial on $M'$. Therefore 
$((\Ind_{Q }^{G} \mathbb Z)\otimes V)^{\mathcal{U}_{M'_2}}  \simeq (\Ind_{Q }^{G} \mathbb Z)^{\mathcal{U} } \otimes V$. Taking now fixed points under $\mathcal{U}_M$, as $\mathcal{U}=\mathcal{U}_{M'_2}\mathcal{U}_M$, 
$$((\Ind_{Q }^{G} \mathbb Z)\otimes V)^{\mathcal{U} }  \simeq ((\Ind_{Q }^{G} \mathbb Z)^{\mathcal{U} } \otimes V)^{\mathcal{U}_M}= (\Ind_{Q }^{G} \mathbb Z)^{\mathcal{U} } \otimes V ^{\mathcal{U}_M}$$
The equality  uses that the $\mathbb Z$-module   $\Ind_{Q }^{G} \mathbb Z$ is free. We get the first part of the theorem as $(\Ind_{Q }^{G} \mathbb Z)^{\mathcal{U} } \otimes V ^{\mathcal{U}_M}\simeq (\Ind_{Q }^{G} R)^{\mathcal{U} } \otimes _RV ^{\mathcal{U}_M}$.

 Tensoring with $R$
   the usual exact sequence defining $\St_Q^G \mathbb Z$ gives an isomorphism $\St_Q^G \mathbb Z \otimes R \simeq \St_Q^G R$ and in loc.\ cit.\ it is proved that the resulting map $\St_Q^G R \xrightarrow{\iota_R} C^\infty(\mathcal B, \mathfrak M \otimes R)$ is also injective. Their proof in no way uses the ring structure of $R$, and for any $\mathbb Z$-module $V$, tensoring with $V$ gives a $\mathcal B$-equivariant  embedding $\St_Q^G \mathbb Z \otimes V \xrightarrow{\iota_V} C_c^\infty(\mathcal B, \mathfrak M \otimes V)$. 
The natural map $(\St_Q^G \mathbb Z)^{\mathcal B}\otimes V \to \St_Q^G \mathbb Z\otimes V$ is also injective by Lemma \ref{lemma:facteur} (2). Taking $\mathcal B$-fixed points we get inclusions 
 \begin{equation}\label{eq:|B}(\St_Q^G \mathbb Z)^{\mathcal B}\otimes V \to (\St_Q^G \mathbb Z\otimes V)^{\mathcal B} \to C_c^\infty(\mathcal B, \mathfrak M \otimes V)^{\mathcal B} \simeq \mathfrak M \otimes V.
 \end{equation}
The composite map is surjective, so the inclusions are isomorphisms. The image of $\iota_V$ consists of functions which are left $Z^0$-invariant, and  $\mathcal B=Z^0 \mathcal{U}'$ where $\mathcal{U}'=G'\cap \mathcal{U}$.  It follows that $\iota$ yields an isomorphism $(\St_Q^G \mathbb Z)^{\mathcal{U}'}\simeq  C_c^\infty(Z^0\backslash \mathcal B, \mathfrak M) ^{\mathcal{U}'}$ again consisting of the constant functions. So that in particular $(\St_Q^G \mathbb Z)^{\mathcal{U}'}=(\St_Q^G \mathbb Z)^{\mathcal B}$ and reasoning as previously we get isomorphisms
\begin{equation}\label{eq:|U'}(\St_Q^G \mathbb Z)^{\mathcal{U}'}\otimes V \simeq  (\St_Q^G \mathbb Z\otimes V)^{\mathcal{U}'} \simeq \mathfrak M \otimes V.
 \end{equation}
The equality $(\St_Q^G \mathbb Z)^{\mathcal{U}'}=(\St_Q^G \mathbb Z)^{\mathcal B}$ and the  isomorphisms remain true  when we  replace $\mathcal{U}'$ by any  group between  $\mathcal B$  and  $\mathcal{U}' $. 
 We apply these  results to $\St_{Q\cap M_2}^{M_2} \mathbb Z \otimes V$ to  get that the natural map $(\St_{Q\cap M_2}^{M_2} \mathbb Z)^{\mathcal{U}_{M'_2}} \otimes V\to  (\St_{Q\cap M_2}^{M_2} \mathbb Z \otimes V)^{\mathcal{U}_{M'_2}}$ is an isomorphism and also that $(\St_{Q\cap M_2}^{M_2} \mathbb Z)^{\mathcal{U}_{M'_2}}=(\St_{Q\cap M_2}^{M_2} \mathbb Z)^{\mathcal{U}_{M_2}}$. We have $\mathcal{U}=\mathcal{U}_{M'}\mathcal{U}_{M_2}$ so $(\St_{Q}^{G} \mathbb Z)^{\mathcal{U}}=(\St_{Q\cap M_2}^{M_2} \mathbb Z)^{\mathcal{U}_{M_2}}$ and the natural map $(\St_{Q}^{G} \mathbb Z)^{\mathcal{U} } \otimes V\to  (\St_{Q }^{G} \mathbb Z \otimes V)^{\mathcal{U}_{M'_2}}$ is an isomorphism. The $\mathbb Z$-module  $(\St_{Q}^{G} \mathbb Z)^{\mathcal{U} }$ is free and the $V^{\mathcal{U}_M}= V^{\mathcal{U}}$, so taking fixed points under $\mathcal{U}_M$, we get $(\St_{Q}^{G} \mathbb Z)^{\mathcal{U} } \otimes V^\mathcal{U}\simeq   (\St_{Q }^{G} \mathbb Z \otimes V)^{\mathcal{U}}$. We have $\St_{Q }^{G} \mathbb Z \otimes V=\St_{Q }^{G} R \otimes_R V$ and $(\St_{Q}^{G} \mathbb Z)^{\mathcal{U} } \otimes V^\mathcal{U}= (\St_{Q}^{G} R)^{\mathcal{U} } \otimes_R V^\mathcal{U}$. This ends the proof of the theorem.
 \end{proof}  
\begin{theorem} \label{thm:main8} The $\mathcal{H}_R$-modules 
$(e(\sigma) \otimes_R \Ind_Q^G{\bf 1} )^{\mathcal U}=
e(\sigma)^\mathcal{U}\otimes_R(\Ind_Q^G{\bf 1} )^{\mathcal U}$  are equal. The $\mathcal{H}_R$-modules 
$(e (\sigma)\otimes_R \St_Q^G)^{\mathcal U}=e(\sigma )^\mathcal{U}\otimes_R (\St_Q^G )^\mathcal{U}$ are also equal.
\end{theorem}

\begin{proof}  We already know that the $R$-modules are equal (Theorem \ref{thm:invariant}). 
We show that they are equal as $\mathcal{H}$-modules. The $\mathcal{H}_R$-modules 
$  e(\sigma)^\mathcal{U}\otimes_R(\Ind_Q^G{\bf 1} ) =  e_\mathcal{H}(\sigma^{\mathcal{U}_M})^\mathcal{U}\otimes_R(\Ind_Q^G{\bf 1} )^{\mathcal U}$ are equal  (Theorem~\ref{thm:ouf}), they are isomorphic to 
$\Ind_{\mathcal{H}_Q}^\mathcal{H}(e_{\mathcal{H}_Q}(\sigma^{\mathcal{U}_M}))$ (Proposition \ref{prop:ovvv}), to  $(\Ind_Q^G (e_Q(\sigma)))^\mathcal{U}$ \cite[Proposition 4.4]{arXiv:1703.04921} and to $ (e(\sigma)\otimes_R \Ind_Q^G{\bf 1})^\mathcal{U}$  \cite[Lemma~2.5]{arXiv:1703.05599}). We deduce that the $\mathcal{H}_R$-modules
$e (\sigma)^{\mathcal U}\otimes_R (\Ind_{Q} ^{G} {\bf 1})^\mathcal{U}= (e (\sigma)\otimes_R \Ind_{Q} ^{G} {\bf 1})^{\mathcal U}$ are equal. The same is true when $Q$ is replaced by a parabolic subgroup $Q'$ of $G$ containing $Q$.     The representation $e (\sigma)\otimes_R \St_Q^G$  is the cokernel of the natural $R[G]$-map
  \begin{align*} \oplus _{Q\subsetneq Q'}e(\sigma) \otimes_R\Ind_{Q'}^G{\bf 1} \xrightarrow{\alpha_Q} e(\sigma) \otimes_R\Ind_Q^G{\bf 1} 
 \end{align*}
and the $\mathcal{H}_R$-module  
 $e (\sigma)^{\mathcal U}\otimes_R (\St_Q^G)^{\mathcal U}$ is the cokernel of the natural $\mathcal{H}_R$-map
  \begin{align*}  \oplus _{Q\subsetneq Q'}e(\sigma)^\mathcal{U} \otimes_R(\Ind_{Q'}^G{\bf 1})^{\mathcal U}\xrightarrow{\beta_Q} e(\sigma)^\mathcal{U}\otimes_R(\Ind_Q^G{\bf 1} )^{\mathcal U}
 \end{align*}
obtained by tensoring \eqref{eq:StQU} by $e(\sigma)^\mathcal{U}$ over $R$, because the tensor product is right exact.  
The  maps $\beta_Q=\alpha_Q^U$ are equal and the $R$-modules 
$(\sigma)^{\mathcal U}\otimes_R(\St_{Q} ^{G} )^\mathcal{U}= (e (\sigma)\otimes_R \St_{Q} ^{G})^{\mathcal{U} }$ are equal. This implies that the $\mathcal{H}_R$-modules $(\sigma)^{\mathcal U}\otimes_R(\St_{Q} ^{G} )^\mathcal{U}= (e (\sigma)\otimes_R \St_{Q} ^{G})^{\mathcal{U} }$ are equal.   \end{proof}

\begin{remark}\label{rem:surj}
The proof shows that the  representations $e(\sigma)\otimes_R \Ind_Q^G {\bf 1}$ and $e(\sigma)\otimes \St_Q^G$ of $G$ are  generated by their $\mathcal{U}$-fixed vectors if the representation $\sigma$ of $M$ is generated by its $\mathcal{U}_M$-fixed vectors.
Indeed,  the $R$-modules $e(\sigma)^\mathcal{U}=\sigma^{\mathcal{U}_M}, (\Ind_Q^G {\bf 1})^{\mathcal{U}_{M_2'}}=(\Ind_Q^G {\bf 1})^\mathcal{U}$ are equal. If $\sigma^{\mathcal{U}_M}$ generates $\sigma$, then  $e(\sigma)$   is generated by  $e(\sigma)^\mathcal{U}$. 
The representation  $\Ind_Q^G {\bf 1}|_{M'_2}$ is generated by  $(\Ind_Q^G {\bf 1})^{\mathcal{U}}$ (this follows from the lemma below), we have $G=MM'_2$ and $M'_2$ acts trivially on $e(\sigma)$. Therefore   the $R[G]$-module generated by $\sigma^{\mathcal{U}} \otimes_R (\Ind_Q^G {\bf 1})^{\mathcal{U}}$ is $e(\sigma)\otimes_R \Ind_Q^G {\bf 1}$. As $e(\sigma)\otimes_R \St_Q^G$ is a quotient of $e(\sigma)\otimes_R \Ind_Q^G {\bf 1}$,  the $R[G]$-module generated by $\sigma^{\mathcal{U}} \otimes_R (\St_Q^G)^{\mathcal{U}}$ is $e(\sigma)\otimes_R \St_Q^G$.
\end{remark}
 
 \begin{lemma}For any standard  parabolic subgroup $P$ of $G$, the representation  $\Ind_P^G {\bf 1}|_{G'}$  is generated by its $\mathcal{U}$-fixed vectors.  
  \end{lemma}
  \begin{proof}  Because $G=PG'$ it suffices to prove that if $J$ is an open compact subgroup of $\overline N$ the characteristic function $1_{PJ}$ of $PJ$ is a finite sum of translates of  $1_{P\mathcal{U}}= 1_{P\mathcal{U}_{\overline N}}$ by $G'$. For $t\in T$ we have $ P\mathcal{U} t=Pt^{-1}\mathcal{U}_{\overline N}t$ and we can choose $t\in T\cap J'$ such that $t^{-1}\mathcal{U}_{\overline N}t \subset J$.
  \end{proof}
 \subsection{General triples}
Let $P=MN$ be a standard parabolic subgroup of $G$. We now investigate situations where $\Delta_P$ and $\Delta\setminus \Delta_P$ are not necessarily orthogonal. 
Let  $ \mathcal V $ a right $\mathcal{H}_{M,R}$-module.

\begin{definition}\label{def:Htriple} Let  $P(\mathcal{V})=M(\mathcal{V}) N(\mathcal{V})$ be the  standard  parabolic subgroup of $G$ with $\Delta_{P(\mathcal{V})}= \Delta_P \cup  \Delta_{ \mathcal V}$ and  $$ \Delta_{ \mathcal V}= \{\alpha \in \Delta \ \text{ orthogonal to  $ \Delta_M$,   $ T^{M,*}(z)$ acts trivially on $\mathcal V$ for all $z\in Z\cap M'_\alpha$}\} .$$
If $Q$ is a parabolic subgroup of $G$ between $P$ and $P(\mathcal{V})$,  the triple $(P,\mathcal{V},Q)$ called an $\mathcal{H}_R$-triple, defines  a right $\mathcal{H}_R$-module $I_\mathcal{H}(P,\mathcal{V},Q)$ equal to
$$\Ind_  {\mathcal{H}_{M(\mathcal{V})}}^\mathcal{H}( e(\mathcal V)\otimes_R (\St_{Q\cap M(\mathcal{V})}^{M(\mathcal{V})})^{\mathcal U_{{M(\mathcal{V})}}})= ( e(\mathcal V)\otimes_R (\St_{Q\cap M(\mathcal{V})}^{M(\mathcal{V})})^{\mathcal U_{{M(\mathcal{V})}}}) \otimes_{\mathcal{H}_{M(\mathcal{V})^+,R},\theta} \mathcal{H}_R$$
where $e(\mathcal V)$ is the extension of $\mathcal{V}$ to  $\mathcal{H}_{M(\mathcal{V})}$. \end{definition}
 
This definition is justified by the fact that  $M(\mathcal{V})$ is the  maximal standard Levi subgroup of $G$ such that  the $\mathcal{H}_{M,R}$-module $\mathcal{V}$ is extensible to $\mathcal{H}_{M(\mathcal{V})}$:

\begin{lemma}\label{lemma:ext}   $\Delta_{\mathcal V}$ is the maximal subset of $\Delta \setminus \Delta_P$ orthogonal to $\Delta_P$ such that  $T^{M,*}_\lambda$ acts trivially on $\mathcal V$ for all $\lambda \in \Lambda(1) \cap {}_1W_{M'_{\mathcal V}}$.  \end{lemma}
\begin{proof}  For $J\subset \Delta$ let $M_J$ denote the standard Levi subgroup of $G$ with $\Delta_{M_J}=J$. The group $Z\cap M'_J$ is  generated by the  $Z\cap M'_\alpha$ for all $\alpha \in J$ (Lemma \ref{lemma:ZG'}). 
When $J$ is orthogonal to $\Delta_M$ and $\lambda\in \Lambda_{M'_J}(1)$,  $\ell_M( \lambda)=0$ where $\ell_M$ is the length associated to $S_M^{\aff}$, and   the map $\lambda \mapsto T_\lambda^{M,*}=T_\lambda^M: \Lambda_{M'_J}(1)\to \mathcal{H}_M$ is multiplicative.
 \end{proof}

The following  is the natural generalisation of Proposition \ref{prop:ovvv} and Corollary \ref{cor:IHPVQ}.  
 Let $Q'$ be  a parabolic subgroup of $G$ with $Q\subset Q'\subset P(\mathcal{V})$. Applying the results of \S \ref{S:9.1} to $M(\mathcal{V})$ and its standard parabolic subgroups $Q\cap M(\mathcal{V})\subset Q'\cap M(\mathcal{V})$,  we have an  $\mathcal{H}_{M(\mathcal{V}),R}$-isomorphism
  \begin{align*} 
  \Ind_{\mathcal{H}_{Q}}^{\mathcal{H}_{M(\mathcal{V})}}( e_{\mathcal{H}_{Q}}(\mathcal V))=
  e_{\mathcal{H}_{Q}}(\mathcal V) \otimes_{ \mathcal{H}_{M_Q^+},\theta} \mathcal{H}_{M(\mathcal{V}),R}&\xrightarrow{\kappa_{Q\cap M(\mathcal{V})}}
   e(\mathcal V)\otimes_R (\Ind_{Q\cap M(\mathcal{V})}^{M(\mathcal{V})} {\bf 1})^{\mathcal U_{M(\mathcal{V})}}\\
   v\otimes 1_{\mathcal{H}} &\mapsto v\otimes f_{Q\mathcal U\cap M(\mathcal{V})}:
 \end{align*}
and an  $\mathcal{H}_{M(\mathcal{V}),R}$-embedding
 \begin{align*}\Ind_{\mathcal{H}_{Q'}}^{\mathcal{H}_{M(\mathcal{V})}}( e_{\mathcal{H}_{Q'}}
(\mathcal V))&\xrightarrow{\iota(Q\cap M(\mathcal{V}),Q'\cap M(\mathcal{V})) }  \Ind_{\mathcal{H}_Q}^{\mathcal{H}_{M(\mathcal{V})}}( e_{\mathcal{H}_Q}
(\mathcal V))\\
v \otimes 1_{\mathcal{H}_{M(\mathcal{V})}}&\mapsto v \otimes \theta_{Q' }^{P(\mathcal{V})} (e_{ Q }^{Q' }).
 \end{align*}
 Applying the  parabolic  induction
 $\Ind_  {\mathcal{H}_{M(\mathcal{V})}}^\mathcal{H}$ which is exact and transitive, we obtain   an $\mathcal{H}_R$-isomorphism $\kappa_Q=\Ind_  {\mathcal{H}_{M(\mathcal{V})}}^\mathcal{H}(\kappa_{Q\cap M(\mathcal{V})})$, 
 \begin{align}\label{eq:kQ}
 \Ind_{\mathcal{H}_{Q}}^\mathcal{H} (e_{\mathcal{H}_{Q}}(\mathcal{V})  )\xrightarrow{\kappa_Q} \Ind_  {\mathcal{H}_{M(\mathcal{V})}}^\mathcal{H}( e(\mathcal V)\otimes_R (\Ind_{Q\cap M(\mathcal{V})}^{M(\mathcal{V})} {\bf 1}_{M_Q})^{\mathcal U_{{M(\mathcal{V})}}})
  \end{align}
 $$
 v \otimes 1_\mathcal{H}\mapsto v\otimes f_{Q\mathcal{U}_{M(\mathcal{V})}}\otimes 1_\mathcal{H} \quad \quad \quad \quad \quad \quad \quad 
 $$
and
 an $\mathcal{H}_R$-embedding $\iota(Q,Q') = \Ind_{\mathcal{H}_{M(\mathcal{V})}}^\mathcal{H} (\iota(Q,Q')^{M(\mathcal{V})})$
\begin{equation}\label{eq:iotaQQ'}v \otimes 1_{\mathcal{H} }\mapsto v \otimes \theta_{Q' } (e_{ Q }^{Q' }):\Ind_{\mathcal \mathcal{H}_{Q'}}^\mathcal{H}(e_{\mathcal{H}_{Q'}}(\mathcal V))  \xrightarrow{\iota(Q,Q') } \Ind_{\mathcal{H}_{Q}}^\mathcal{H}(e_{\mathcal{H}_{Q}}(\mathcal V)).\end{equation} 
   Applying  Corollary \ref{cor:IHPVQ} we obtain:
 
 \begin{theorem} \label{thm:IHP}Let  $(P,\mathcal V,Q)$ be an $\mathcal{H}_R$-triple. Then,  the  cokernel of the  $\mathcal{H}_R$-map
   $$\oplus_{Q\subsetneq Q'\subset P(\mathcal{V})}  \Ind_{\mathcal{H}_{Q'}}^\mathcal{H}( e_{\mathcal{H}_{Q'}}
(\mathcal V))\to \Ind_{\mathcal{H}_{Q}}^\mathcal{H} ( e_{\mathcal{H}_{Q}}(\mathcal V)),$$
defined  by the  $\iota(Q,Q') $ is isomorphic to   $ I_\mathcal{H}(P,\mathcal{V},Q)$ via the $\mathcal{H}_R$-isomorphism $\kappa_Q$.
 \end{theorem}

Let  $\sigma$  be  a   smooth $R$-representation of $M$ and  $Q$  a parabolic subgroup of $G$ with $P\subset Q\subset P(\sigma)$. 
\begin{remark}
The $\mathcal{H}_R$-module $ I_\mathcal{H} (P ,\sigma^{\mathcal{U}_M} ,Q)$ is defined if $\Delta_Q \setminus \Delta_P$ and $\Delta_P$ are orthogonal because $Q\subset P(\sigma)\subset P(\sigma^{\mathcal{U}_M})$ (Theorem \ref{thm:ouf}).
\end{remark}
We denote here by $P_{\min}=M_{\min}N_{\min}$  the minimal standard parabolic subgroup  of $G$ contained in $P$  such that $\sigma=e_P( \sigma|_{M_{\min}} )$ (Lemma  \ref {lemma:min}, we drop the index $\sigma$).  The sets of roots   $\Delta_{P_{\min}}$ and $\Delta_{P(\sigma|_{M_{\min}})}\setminus \Delta_{P_{\min}}$ are orthogonal (Lemma  \ref{lemma:2.2}). The groups $P(\sigma)=P( \sigma|_{M_{\min}})$,  the representations $e(\sigma)=e(\sigma|_{M_{\min}})$ of $M(\sigma)$, the representations $I_G(P,\sigma,Q)=I_G(P_{\min},\sigma|_{M_{\min}}  ,Q)=\Ind_{P(\sigma)}^G (e(\sigma ) \otimes_R \St_{Q }^{P(\sigma)})$ of $G$, and the $R$-modules $\sigma ^{\mathcal{U}_{M_{\min}}}=\sigma^{\mathcal{U}_M}$ are equal. From Theorem \ref{thm:ouf},   $$P(\sigma) \subset P(\sigma ^{\mathcal{U}_{M_{\min}}}), \quad e_{\mathcal{H}_{M(\sigma)}} (\sigma^{\mathcal U_{M_{\min}}})= e(\sigma)^{\mathcal U_{M(\sigma)}}, $$  and $P(\sigma^{\mathcal U_{M(\sigma)}})= P(\sigma)$ if    $\sigma^{\mathcal U_{M(\sigma)}}$  generates the representation $\sigma|_{M_{\min}}$. The $\mathcal{H}_R$-module
$$ I_\mathcal{H} (P_{\min},\sigma ^{\mathcal{U}_{M_{\min}}},Q)=\Ind_  {\mathcal{H}_{M(\sigma ^{\mathcal{U}_{M_{\min}}})}}^\mathcal{H}( e(\sigma ^{\mathcal{U}_{M_{\min}}})\otimes_R (\St_{Q } ^{P(\sigma ^{\mathcal{U}_{M_{\min}}})})^{\mathcal{U}_{M(\sigma ^{\mathcal{U}_{M_{\min}}})}})$$
is defined because $\Delta_{P_{\min}}$ and $\Delta_{P(\sigma ^{\mathcal{U}_{M_{\min}}})}\setminus \Delta_{P_{\min}}$ are orthogonal and
 $P\subset Q\subset P(\sigma) \subset P(\sigma ^{\mathcal{U}_{M_{\min}}})$. 
 
 \begin{remark} 
 If  $\sigma^{\mathcal U_{M(\sigma)}}$  generates the representation $\sigma|_{M_{\min}}$ (in particular if $R=C$ and $\sigma$ is irrreducible), then  $P(\sigma)=P(\sigma ^{\mathcal{U}_{M_{\min}}})$ hence
 $$ I_\mathcal{H} (P_{\min},\sigma ^{\mathcal{U}_{M_{\min}}},Q)= \Ind_{\mathcal{H}_{M(\sigma)}}^\mathcal{H} (e_{\mathcal{H}_{M(\sigma)}} (\sigma^{\mathcal U_{M_{\min}}})\otimes_R (\St_{Q\cap M(\sigma)}^{M(\sigma)} )^{\mathcal{U}_{M(\sigma)}}).$$
 \end{remark}  

 Applying Theorem \ref{thm:main8}  to $(P_{\min} \cap M(\sigma), \sigma|_{M_{\min}}, Q\cap M(\sigma))$, 
  the  $\mathcal{H}_{M(\sigma),R}$-modules 
  \begin{equation}\label{eq:main8}e_{\mathcal{H}_{M(\sigma)}} (\sigma^{\mathcal U_{M_{\min}}})\otimes_R (\St_{Q\cap M(\sigma)}^{M(\sigma)})^{\mathcal{U}_{M(\sigma)}}=(e_{M(\sigma)}(\sigma ) \otimes_R \St_{Q\cap M(\sigma)}^{M(\sigma)})^{\mathcal{U}_{M(\sigma)}}
  \end{equation} are equal. We have the   $\mathcal{H}_R$-isomorphism \cite[Proposition 4.4]{arXiv:1703.04921}:
\begin{align*} 
 I_G(P,\sigma,Q)^\mathcal{U}=(\Ind_{P(\sigma)}^G (e(\sigma ) \otimes_R \St_{Q }^{P(\sigma)}))^{\mathcal{U}}&\xrightarrow{ov}  \Ind_{\mathcal{H}_{M(\sigma)}}^\mathcal{H} ((e(\sigma ) \otimes_R \St_{Q\cap M(\sigma)}^{M(\sigma)})^{\mathcal{U}_{M(\sigma)}})\\
f_{P(\sigma)\mathcal{U},x}&\mapsto x\otimes 1_{\mathcal{H} } \quad (x\in (e(\sigma ) \otimes_R \St_{Q\cap M(\sigma)}^{M(\sigma)})^{\mathcal{U}_{M(\sigma)}}).\end{align*}
We deduce:
  
 \begin{theorem}\label{thm:main} Let  $(P,\sigma,Q)$ be a $R[G]$-triple.  Then, we have  the   $\mathcal{H}_R$-isomorphism
$$ I_G(P,\sigma,Q)^\mathcal{U}\xrightarrow{ov}\Ind_{\mathcal{H}_{M(\sigma)}}^\mathcal{H} (e_{\mathcal{H}_{M(\sigma)}} (\sigma^{\mathcal U_{M_{\min}}})\otimes_R (\St_{Q\cap M(\sigma)}^{M(\sigma)} )^{\mathcal{U}_{M(\sigma)}}).$$
In particular, 
$$ I_G(P,\sigma,Q)^\mathcal{U}\simeq \begin{cases} I_\mathcal{H} (P_{\min},\sigma ^{\mathcal{U}_{M_{\min}}},Q) &\ \text{if} \ P(\sigma)=P(\sigma ^{\mathcal{U}_{M_{\min}}})\\
 I_\mathcal{H} (P ,\sigma ^{\mathcal{U}_{M }},Q) &\ \text{if} \ P=P_{\min},  P(\sigma)=P(\sigma ^{\mathcal{U}_M})\end{cases}.
 $$
 \end{theorem}
  
  \subsection{Comparison of the parabolic induction and coinduction}\label{S:comp} 
 Let  $P=MN$ be a standard parabolic subgroup of $G$,   $\mathcal V$ a right  $\mathcal{H}_R$-module and $Q $ a parabolic subgroup  of $G$ with  $ Q  \subset P(\mathcal V)$.  When $R=C$, in \cite{arXiv:1406.1003_accepted}, we associated to $(P,\mathcal{V},Q)$  an  $\mathcal{H}_R$-module using   the  parabolic coinduction   $$\Coind_{\mathcal{H}_M}^\mathcal{H}(-)=\Hom_{\mathcal{H}_{ M ^-, \theta^*}}(\mathcal{H}, -):\Mod_R(\mathcal{H}_M)\to \Mod_R(\mathcal{H})$$ instead of  the parabolic induction  $\Ind_{\mathcal{H}_M}^\mathcal{H}(-)=- \otimes_{\mathcal{H}_{M^+}, \theta}\mathcal{H}$. The index $\theta^*$ in the parabolic coinduction means that $\mathcal{H}_{M_{Q}^-} $ embeds  in $\mathcal{H}$ by  $\theta_Q^*$.  Our  terminology is different from the one  in \cite{arXiv:1406.1003_accepted} where  the parabolic coinduction is called induction.
  For  a parabolic subgroup  $Q'$ of $G$ with  $ Q\subset Q'  \subset P(\mathcal V)$,  there is a natural inclusion  of $\mathcal{H}_R$-modules  \cite[Proposition 4.19]{arXiv:1406.1003_accepted} \begin{equation}\label{eq:canonicalH}
 \Hom_{\mathcal{H}_{ M_{Q'}^-, \theta^*}}(\mathcal{H}, e_{\mathcal{H}_{Q'} }(\mathcal V))\xrightarrow{i(Q,Q')} \mathcal \Hom_{\mathcal{H}_{M_{Q}^-, \theta^*}}(\mathcal{H},e_{\mathcal{H}_{Q}}(\mathcal V)).
\end{equation}
because  $\theta^*(\mathcal{H}_{ M_{Q}^-})\subset \theta^*(\mathcal{H}_{ M_{Q'}^-})$ as $W_{M_Q^-}(1)\subset W_{M_{Q'}^-}(1)$,   and  
$vT_w^{ M_{Q'}*}=v T^{M_Q *}_w$ for $w\in W_{M_Q^-}(1)$  and $v\in \mathcal V$.
   
 \begin{definition} \label{def:coind} Let   $ CI_{\mathcal{H}}(P,\mathcal V,Q)$
denote  the cokernel of the map 
$$\oplus_{Q \subsetneq  Q'\subset P(\mathcal V)} \ \Hom_{\mathcal{H}_{ M_{Q'}^-, \theta^*}}(\mathcal{H}, e_{\mathcal{H}_{Q'} }(\mathcal V))\to \mathcal \Hom_{\mathcal{H}_{M_{Q}^-, \theta^*}}(\mathcal{H},e_{\mathcal{H}_{Q}}(\mathcal V))
$$
 defined by the  $\mathcal{H}_R$-embeddings $i(Q,Q')$. 
  \end{definition}
 When $R=C$, we  showed that the $\mathcal{H}_C$-module $\mathcal CI_{\mathcal{H}}(P,\mathcal V,Q)$ is simple when  $ \mathcal V$ is simple and supersingular (Definition \ref{def:supersingularH}), and that any simple   $\mathcal{H}_C$-module is of this form for a  
  $\mathcal{H}_C$-triple $(P,\mathcal{V},Q)$ where $ \mathcal V$ is simple and supersingular, $P,Q$ and the isomorphism class of  $\mathcal{V}$ are unique~\cite{arXiv:1406.1003_accepted}.   The aim of this section is to compare the $\mathcal{H}_R$-modules $I_\mathcal{H}(P,\mathcal{V},Q)$ with the $\mathcal{H}_R$-modules $CI_\mathcal{H}(P,\mathcal{V},Q)$ and to show that the classification   is also valid with   the $\mathcal{H}_C$-modules $\mathcal I_{\mathcal{H}}(P,\mathcal V,Q)$.

\bigskip  It is already known that a parabolically coinduced module is a parabolically induced module  and vice versa \cite{arXiv:1406.1003_accepted} \cite{MR3437789}. To make it more precise we need to introduce  notations.

   We lift the  elements  $w$ of the finite Weyl group $\mathbb W$ to $\hat w \in \mathcal N_G\cap \mathcal K$ as in  \cite[IV.6]{MR3600042}, \cite[Proposition 2.7]{arXiv:1703.04921}: they satisfy the braid relations $\hat w _1\hat w_2 =(w_1 w_2)\hat {} $ when $\ell(w_1)+\ell(w_2)=\ell(w_1w_2)$ and when $s\in S$,  $\hat s$ is admissible, in particular    lies in $ {}_1 W_{G'}$.

Let $\mathbf w, \mathbf w_M,  \mathbf w^M $ denote respectively  the longest elements in $\mathbb W, \mathbb W_M$ and  $\mathbf w   \mathbf w_M$. We have 
  $\mathbf w=\mathbf w^{-1} =  \mathbf w^M \mathbf w_M, \mathbf w_M=  \mathbf w_M^{-1}$,  $\hat{ \mathbf w}= \hat{ \mathbf w}^M \hat{ \mathbf w}_M$,   $$\mathbf w^M( \Delta_M)= -  \mathbf w  ( \Delta_M) \subset \Delta, \quad\mathbf w^M(\Phi^+ \setminus \Phi^+_M)=\mathbf w (\Phi^+ \setminus \Phi^+_M) .$$ 
  Let $\mathbf w.M$ be the standard Levi subgroup of $G$ with $\Delta_{\mathbf w.M}=\mathbf w^M( \Delta_M)$ and $\mathbf w.P$ the standard parabolic subgroup of $G$ with Levi $\mathbf w.M$. 
  We have 
  $$\mathbf w.M=  \hat{ \mathbf w}^M M ( \hat{ \mathbf w}^M)^{-1}=  \hat{ \mathbf w}  M ( \hat{ \mathbf w} )^{-1}, \quad \mathbf w^{\mathbf w.M}=  \mathbf w_M\mathbf w = ( \mathbf w^M)^{-1}.$$The conjugation $w\mapsto  \mathbf w^M w (  \mathbf w^M)^{-1}$ in $W$ gives  a group isomorphism
  $ W_M\to W_{\mathbf w.M} $ sending $S^{\aff}_M$ onto $S^{\aff}_{\mathbf w.M}$,   respecting the finite Weyl subgroups   $ \mathbf w^M \mathbb W_M (  \mathbf w^M)^{-1}= \mathbb W_{\mathbf w.M}=\mathbf w \mathbb W_M\mathbf w^{-1}$,  and echanging $W_{M^+}$ and $W_{(\mathbf w.M)^-} =\mathbf w   W_{M^+}\mathbf w^{-1}$.    The conjugation by $ \tilde{ \mathbf w}^M$ restricts to a  group isomorphism $W_M(1)\to W_{\mathbf w.M}(1)$ sending $W_{M^+}(1)$ onto $ W_{(\mathbf w.M)^-}(1)$.      The linear isomorphism
\begin{equation}\label{eq:isoH} \mathcal{H}_{ M}\xrightarrow{\iota ( \tilde{ \mathbf w}^M) }\mathcal{H}_{\mathbf w. M} \quad T^M_{w} \mapsto T^{\mathbf w.M}_{ \tilde{ \mathbf w}^M w ( \tilde{ \mathbf w}^M)^{-1}} \ \text{for} \  w\in W_M(1),
\end{equation}
is a ring isomorphism between the pro-$p$-Iwahori Hecke rings of $M$ and $\mathbf w. M$.   It sends  the positive part  $\mathcal{H}_{ M^+}$ of $\mathcal{H}_M$ onto the negative part $ \mathcal{H}_{(\mathbf w. M)^-}$ of $\mathcal{H}_{\mathbf w. M}$ \cite[Proposition 2.20]{MR3437789}.  We have $\tilde{ \mathbf w}=\tilde{ \mathbf w}_M\tilde{ \mathbf w}^{\mathbf w.M}=\tilde{ \mathbf w}^{M}\tilde{ \mathbf w}_M$, 
$(\tilde{ \mathbf w}^M )^{-1}=\tilde{ \mathbf w}^{\mathbf w.M}t_M$ where $t_M= \tilde{ \mathbf w}^2\tilde{ \mathbf w}_M^{-2} \in Z_k$. 

\begin{definition}
The {\bf  twist  $\tilde{ \mathbf w} ^M.\mathcal V$  of $\mathcal{V}$ by $\tilde{ \mathbf w} ^M$} is
the right $\mathcal{H}_{\mathbf w. M}$-module deduced from  the right $\mathcal{H}_{ M}$-module $\mathcal V$ by functoriality: as $R$-modules $\tilde{ \mathbf w} ^M.\mathcal V=\mathcal V $ and  for $v\in \mathcal V, w\in W_M(1)$ we have $ vT^{\mathbf w.M}_{ \tilde{ \mathbf w}^M w ( \tilde{ \mathbf w}^M)^{-1}}=v T_w^M$. 
\end{definition} 
We can  define the  twist  $\tilde{ \mathbf w} ^M.\mathcal V$  of $\mathcal{V}$ with the $T_w^{M,*}$ instead of $T_w^{M}$.
 \begin{lemma}  For $v\in \mathcal V, w\in W_M(1)$  we have  $ vT^{\mathbf w.M,*}_{ \tilde{ \mathbf w}^M w ( \tilde{ \mathbf w}^M)^{-1}}=v T_w^{M,*} $ in  $\tilde{ \mathbf w} ^M.\mathcal V$. \end{lemma}
   \begin{proof} By the ring isomorphism $\mathcal{H}_{ M}\xrightarrow{\iota ( \tilde{ \mathbf w}^M) }\mathcal{H}_{\mathbf w. M}$,  we have $ c^{\mathbf w.M}_{\tilde{ \mathbf w}^M \tilde  s ( \tilde{ \mathbf w}^M)^{-1}}= c^{M}_{ \tilde  s  }$ when $\tilde s\in W_M(1)$ lifts  $s\in S_M^{\aff} $. So the equality of the lemma is true for $w=\tilde s $. Apply the braid relations to get the equality for all $w\in W_M(1)$.
   \end{proof} 
We return to the $\mathcal{H}_R$-module $\Hom_{\mathcal{H}_{ M ^-, \theta^*}}(\mathcal{H}, V)$ parabolically coinduced from $\mathcal{V}$. It  has a natural direct decomposition indexed by the set $ \mathbb W^{ \mathbb W _M}$ of elements $d$ in the finite Weyl group  $\mathbb W $ of minimal length in  the coset $d  \mathbb W _M$. Indeed it is known that 
the linear map   $$f\mapsto (f(T_{\tilde d}))_{d\in \mathbb W^{ \mathbb W _M}}:  \Hom_{\mathcal{H}_{ M ^-},\theta^*}(\mathcal{H}, \mathcal V)\to \oplus _{d\in \mathbb W^{ \mathbb W _M}} \mathcal V$$
 is an isomorphism. For  $v\in \mathcal V$ and $d\in \mathbb W^{ \mathbb W _M}$,   there is a unique element 
 $$f_{\tilde d,v} \in  \Hom_{\mathcal{H}_{ M ^-},\theta^*}(\mathcal{H}, \mathcal V) \ \text{ satisfying   $f(T_{\tilde d})=v$ and  $f(T_{\tilde d'} ) =0$ for $d'\in \mathbb W^{ \mathbb W _M}\setminus \{d\}$}.$$  It is  known that  the map 
$v \mapsto f_{ \tilde {\mathbf w}^M,v}: \tilde {\mathbf w}^M.\mathcal V \to  \Hom_{\mathcal{H}_{ M ^-},\theta^*}(\mathcal{H}, \mathcal V)$ 
is $\mathcal{H}_{(\mathbf w.M)^+ }$-equivariant:   $f_{ \tilde {\mathbf w}^M,vT^{\mathbf w. M}
 _  w  } = f_{ \tilde {\mathbf w}^M,v} T_ w  $ for all $v\in \mathcal{V}, w  \in W_{\mathbf w .M^+}(1) $. By adjunction, this $\mathcal{H}_{(\mathbf w.M)^+ }$-equivariant map gives an $\mathcal{H}_R$-homomorphism from  an induced module to a coinduced module: \begin{equation}\label{eq:Abe} v\otimes 1_{\mathcal{H}}\mapsto  f_{ \tilde {\mathbf w}^M,v}: \tilde {\mathbf w}^M.\mathcal V \otimes_{\mathcal{H}_{(\mathbf w.M)^+}, \theta}\mathcal{H}\xrightarrow{\mu_P}\Hom_{\mathcal{H}_{ M ^-},\theta^*}(\mathcal{H},\mathcal V).
\end{equation}
This is an isomorphism  \cite{arXiv:1406.1003_accepted}, \cite{MR3437789} .
 
 The naive guess that a variant $\mu_Q$ of $\mu_P$  induces an $\mathcal{H}_R$-isomorphism between the $\mathcal{H}_R$-modules $I_\mathcal{H}(\mathbf w.P,\tilde {\mathbf w}^M.\mathcal V,\mathbf w.Q)$ and $ CI_{\mathcal{H}}(P,\mathcal V,Q)$ turns out to be true.  The proof  is the aim of the rest of this section. 

\bigskip  The  $\mathcal{H}_R$-module 
 $I_\mathcal{H}(\mathbf w.P,\tilde {\mathbf w}^M.\mathcal V,\mathbf w.Q)$ is well defined because  the parabolic subgroups of $G$ containing  $\mathbf w. P$ and contained in $P(\tilde {\mathbf w}^M.\mathcal V) $ are $\mathbf w. Q$ for $P\subset Q \subset P(\mathcal{V})$, as follows from:

\begin{lemma}\label{lemma:delt} $\Delta_{\tilde{ \mathbf w}^M.\mathcal{V}}=- {\mathbf w} (\Delta_{\mathcal{V}})$.
\end{lemma}
\begin{proof} Recall that  $\Delta_{\mathcal V}$ is the set of simple roots $\alpha \in \Delta \setminus \Delta_M$ orthogonal to $\Delta_M$ and  $T^{M,*}(z)$ acts trivially on $\mathcal V$  for all  $z\in Z\cap M'_\alpha$, and the corresponding standard parabolic subgroup $P_{\mathcal V}=M_{\mathcal V}N_{\mathcal V}$. The $Z\cap M'_\alpha$ for $\alpha \in \Delta_{\mathcal V}$ generate the group    $ Z\cap M'_{\mathcal V}$.  A root $\alpha \in \Delta \setminus \Delta_M$ orthogonal to $\Delta_M$ is fixed by ${ \mathbf w}_M$ so ${ \mathbf w}^M(\alpha)=\mathbf w (\alpha)$ and 
 $${\hat { \mathbf w}}^M M_\mathcal{V}  ( {\hat{ \mathbf w}}^M)^{-1}= {\hat { \mathbf w}}  M_\mathcal{V}  ( {\hat{ \mathbf w}})^{-1}.$$
  The proof of Lemma \ref{lemma:delt} is straightforward as $\Delta= -{\mathbf w}(\Delta)$, $ \Delta_{ \mathbf w. M}= -{\mathbf w}(\Delta_M)$.
  \end{proof}
 
 Before going further, we check the commutativity of the extension with the twist. As $Q=M_QU$ and $M_Q$ determine each other we denote ${\mathbf w}_{M_Q}= { \mathbf w}_{ Q},  {\mathbf w}^{M_Q}= {\mathbf w}^{ Q}$  when $Q\neq P,G$.
  \begin{lemma} \label{lemma:eH} $e_{\mathcal{H}_{\mathbf w.Q}}( \tilde{ \mathbf w}^M.\mathcal V)= \tilde{ \mathbf w}^{ Q}. e_{\mathcal{H}_{Q}}(\mathcal V)$.
  \end{lemma}  
 \begin{proof} As $R$-modules $\mathcal{V}=e_{\mathcal{H}_{\mathbf w.Q}}( \tilde{ \mathbf w}^M.\mathcal V)= \tilde{ \mathbf w}^{ Q}. e_{\mathcal{H}_{Q}}(\mathcal V)$. A direct computation shows that  the Hecke element  $T^{{ \mathbf w}.Q,*}_w$ acts  in  the $\mathcal{H}_R$-module $ e_{\mathcal{H}_{\mathbf w.Q}}( \tilde{ \mathbf w}^M.\mathcal V)$, by the identity if $w\in  \tilde  { \mathbf w}^Q {}_1W_{M'_2} ({ \mathbf w}^Q)^{-1}$ and by $T^{M,*}_{(\tilde { \mathbf w}^Q)^{-1} w \tilde { \mathbf w}^Q}$ if $w\in  \tilde  { \mathbf w}^Q {}_1W_{M'_2} ({ \mathbf w}^Q)^{-1}$ where $M_2$ denotes the standard Levi subgroup with $\Delta_{M_2}=\Delta_Q\setminus  \Delta_P$. Whereas in  the $\mathcal{H}_R$-module $\tilde{ \mathbf w}^{ Q}. e_{\mathcal{H}_{Q}}(\mathcal V)$,  the Hecke element  $T^{{ \mathbf w}.Q,*}_w$ acts by the identity if $w\in {}_1 W_{{\mathbf w}.M'_2}$ and by $T^{M,*}_{(\tilde { \mathbf w}^M)^{-1} w \tilde { \mathbf w}^M}$ if  $w\in  W_{{\mathbf w}.M}(1)$. So the lemma means that 
$${}_1 W_{{\mathbf w}.M'_2}= \tilde  { \mathbf w}^Q {}_1W_{M'_2} ({ \mathbf w}^Q)^{-1}, \quad (\tilde { \mathbf w}^Q)^{-1} w \tilde { \mathbf w}^Q =(\tilde { \mathbf w}^M)^{-1} w \tilde { \mathbf w}^M \ \text{if} \ w\in  W_{{\mathbf w}.M}(1).$$
 These properties are easily proved using  that   ${}_1W_{G'}$ is normal in $W(1)$ and that the sets of roots $\Delta_P$ and $\Delta_Q\setminus  \Delta_P$  are orthogonal: 
 ${ \mathbf w}_{ Q}=
{\mathbf w}_{M_2}{\mathbf w}_M $, the elements  ${\mathbf w}_{M_2}$ and ${\mathbf w}_M$ normalise $  W_M $ and $W_{M_2}$,  the elements of  $\mathbb W_{M_2}$ commutes with the elements of $\mathbb W_M$.     \end{proof} 
We return to our guess. The variant  $\mu_Q$ of $\mu_P$ is obtained by combining the commutativity of the extension with the twist   and the isomorphism  \ref{eq:Abe}  applied to $(Q, e_{\mathcal{H}_{Q}}(\mathcal V))$ instead of $(P,\mathcal{V})$. The  $\mathcal{H}_R$-isomorphism  $\mu_Q$  is:
 \begin{equation}\label{eq:muQ} v\otimes 1_{\mathcal{H}}\mapsto  f_{ \tilde {\mathbf w}^M,v}: \Ind_{\mathcal{H}_{  \mathbf w.M_Q }}^\mathcal{H}(e_{\mathcal{H}_{\mathbf w.Q}}( \tilde{ \mathbf w}^M.\mathcal V))\xrightarrow{\mu_Q}\Hom_{\mathcal{H}_{ M_Q ^-},\theta^*}(\mathcal{H},e_{\mathcal{H}_{Q}}( \mathcal V)).
\end{equation}

Our guess is that  $\mu_Q$  induces an $\mathcal{H}_R$-isomorphism from the cokernel
 of the  $\mathcal{H}_R$-map
 $$\oplus_{Q\subsetneq  Q'\subset P( \mathcal{V})}  \Ind_{\mathcal{H}_{\mathbf w.Q'}}^\mathcal{H}( e_{\mathcal{H}_{\mathbf w.Q'}}
(\tilde {\mathbf w}^M.\mathcal{V}))\to \Ind_{\mathcal{H}_{\mathbf w.Q}}^\mathcal{H} ( e_{\mathcal{H}_{\mathbf w.Q}}(\tilde {\mathbf w}^M.\mathcal{V}))$$
defined  by the  $\mathcal{H}_R$-embeddings $\iota(\mathbf w.Q,\mathbf w.Q')$, isomorphic to $ I_{\mathcal{H}}(\mathbf w .P, \tilde { \mathbf w}^M\mathcal V,\mathbf w. \overline Q)$ via $\kappa_{{\mathbf w}.Q}$ (Theorem \ref{thm:IHP}),  onto the cokernel   $ CI_{\mathcal{H}}(P,\mathcal V,Q)$
  the $\mathcal{H}_R$-map 
$$\oplus_{Q \subsetneq  Q'\subset P(\mathcal V)} \ \Hom_{\mathcal{H}_{ M_{Q'}^-, \theta^*}}(\mathcal{H}, e_{\mathcal{H}_{Q'} }(\mathcal V))\to \mathcal \Hom_{\mathcal{H}_{M_{Q}^-, \theta^*}}(\mathcal{H},e_{\mathcal{H}_{Q}}(\mathcal V))
$$
 defined by the  $\mathcal{H}_R$-embeddings $i(Q,Q')$.
This is true if  $i(Q,Q')$  corresponds to $\iota(\mathbf w.Q,\mathbf w.Q')$ via the isomorphisms  $\mu_{Q'}$ and $\mu_Q$. This is the content of the next proposition.

\begin{proposition} \label{prop:comp} For all $Q\subsetneq Q'\subset P(\mathcal V)$ we have $$i(Q,Q') \circ \mu_{Q'} = \mu_Q \circ \iota(\mathbf w.Q,\mathbf w.Q').$$
\end{proposition}    

We postpone to section \S \ref{S:comdia} the rather long proof of the proposition.  
\begin{corollary} \label{cor:isoCI}The $\mathcal{H}_R$-isomorphism $\mu_Q\circ \kappa_{{\mathbf w}.Q}^{-1}$ induces an $\mathcal{H}_R$-isomorphism
 $$ I_{\mathcal{H}}(\mathbf w .P, \tilde { \mathbf w}^M\mathcal V,\mathbf w. \overline Q) \to  CI_{\mathcal{H}}(P,\mathcal V,Q).$$
\end{corollary}

\subsection{Supersingular $\mathcal{H}_R$-modules, classification of simple $\mathcal{H}_C$-modules}\label{S:9.5}

We   recall  first the notion of supersingularity based on the action of  center of  $\mathcal{H} $. 

The center of $\mathcal{H} $  \cite[Theorem 1.3]{MR3271250} contains a subalgebra $\mathcal Z_{T^{ +} }$ isomorphic to $\mathbb Z[T^{ +}/T_1]$ where $T^{ +} $ is the monoid of dominant elements of $T$  and $T_1$ is the pro-$p$-Sylow subgroup of the maximal compact subgroup of $T$.  

Let $t\in T$ of image $\mu_t\in W(1)$ and let $(E_o (w))_{w\in W (1)}$ denote the alcove walk basis of $\mathcal{H} $ associated to  a closed Weyl chamber $o$ of $\mathbb W $. The element  $$E_o ( C(\mu_t))= \sum_{\mu'} E_o(\mu') $$  is the sum over the elements in $ \mu'$ in the  conjugacy class $ C(\mu_t)$ of $\mu_t$ in $W(1)$. It is a central element of $\mathcal{H} $ and does not depend on the  choice of $o$.  We write also $z (t)=E_o ( C(\mu_t))$.    
   
 \begin{definition} \label{def:supersingularH} A non-zero right $\mathcal{H}_{ R}$-module $\mathcal{V}$ is called supersingular when,  for any $v\in \mathcal{V}$ and any non-invertible $t\in T^{ +}$, there exists a positive integer $n \in \mathbb N$ such that $ v(z (t))^n=0$.  If one can choose $n $  independent on  $(v,t)$, then $\mathcal{V}$ is called uniformly supersingular.\end{definition}

\begin{remark} 
 One can choose $n $  independent on  $(v,t)$  when $\mathcal{V}$ is finitely generated as a right $\mathcal{H}_{ R}$-module. If $R$ is a field and  $\mathcal{V}$ is simple we can take $n=1$.
 
   When $G$ is compact modulo the center, $T^+=T$, and any non-zero $\mathcal{H}_{ R}$-module is supersingular.  
   \end{remark}

  The induction functor $\Ind_{\mathcal{H}_M}^\mathcal{H}:\Mod (\mathcal{H}_{M,R})\to \Mod (\mathcal{H}_R)$ has a left 
 adjoint $\mathcal L_{\mathcal{H}_M}^\mathcal{H}$ and a right adjoint $\mathcal R_{\mathcal{H}_M}^\mathcal{H}$ \cite{MR3437789}: for $\mathcal{V}\in \Mod (\mathcal{H}_R)$, 
\begin{equation}\label{eq:LR}\mathcal L_{\mathcal{H}_M}^\mathcal{H}(\mathcal{V})=  \tilde { \mathbf w}^{ \mathbf w.M} \circ (\mathcal{V}\otimes_{\mathcal{H}_{( \mathbf w.M)^-}, \theta^*}\mathcal{H}_{ \mathbf w.M}), \quad \mathcal R_{\mathcal{H}_M}^\mathcal{H}(\mathcal{V})=\Hom_{\mathcal{H}_{M^+},\theta}(\mathcal{H}_M,\mathcal{V}).
\end{equation}

In the  left adjoint, $\mathcal{V}$ is seen as a   right $\mathcal{H}_{( \mathbf w.M)^- }$-module via the ring homomorphism $\theta^*_{\mathbf w.M}\colon \mathcal{H}_{ (\mathbf w.M)^-}\to\mathcal{H}$; in the right adjoint, $\mathcal{V}$ is seen as a   right  $\mathcal{H}_{ M^+}$-module via the  ring homomorphism $\theta_M \colon\mathcal{H}_{ M^+}\to\mathcal{H}$ (\S \ref{S:2.3}).  

\begin{proposition}\label{pro:ssLR} Assume   that  $\mathcal{V}$  is a supersingular right $\mathcal{H}_{R}$-module and that $p$ is nilpotent in $\mathcal{V}$. Then $\mathcal L_{\mathcal{H}_M}^\mathcal{H} (\mathcal{V})= 0$, and if $\mathcal{V}$ is uniformly supersingular   $\mathcal R_{\mathcal{H}_M}^\mathcal{H} (\mathcal{V})=0$.\end{proposition}  
\begin{proof}  This is a  consequence  of  three known properties:  
\begin{enumerate}
\item $\mathcal{H}_M$ is the  localisation of $\mathcal{H}_{M^+}$ (resp.\ $\mathcal{H}_{M^-}$) at $T^M_\mu$ for any element $\mu  \in \Lambda_T(1)$, central in $W_M(1)$  and strictly $N$-positive (resp.\ $N$-negative), and $T^M_\mu=T^{M,*}_\mu$. See  \cite[Theorem  1.4]{MR3437789}. 
\item When $o$ is anti-dominant, $E_o(\mu)=T_{\mu}$ if $\mu\in \Lambda^+(1)$ and  $E_o(\mu)=T^*_{\mu}$ if $\mu\in \Lambda^-(1)$.
\item Let an integer   $n>0$ and   $\mu\in \Lambda(1)$ such that  the $\mathbb W$-orbit of $v(\mu)\in X_*(T)\otimes \mathbb Q$ (Definition in \S \ref{S:2.1})   and  of $\mu$ have the same number of elements. Then  $$(E_{o} ( C(\mu)))^{n} E_o (\mu) - E_o (\mu)^{n+1} \in p \mathcal{H}. $$
 See \cite[Lemma 6.5]{Vigneras-prop-III}, where the  hypotheses  are  given  in the proof (but not written in the lemma). 
\end{enumerate}

 Let $\mu\in  \Lambda_T^+(1)$ satisfying  (1) for $M^+$ and (3),  similarly let  $\mathbf w. \mu\in \Lambda_T^-(1)$ satisfying (1) for $(\mathbf w.M)^-$ and (3).   For $(R,\mathcal{V})$ as in the proposition, let $v\in \mathcal{V}$  and $n>0$ such that $vE_{o} ( C(\mu))^n= vE_{o} ( C(\mathbf w.\mu))^n=0$. Multiplying by $E_o (\mu)$ or $E_o(\mathbf w.\mu)$, and applying (3)  and (2) for $o$ anti-dominant  we get: $$vE_o(\mu^{n+1})= vT_\mu^{n+1} \in p \mathcal{V}, \quad vE_o((\mathbf w.\mu)^{n+1})= v(T^*_{\mathbf w.\mu})^{n+1}\in p \mathcal{V} .$$
The  proposition follows from: $vT_\mu^{n+1}, v(T^*_{\mathbf w.\mu})^{n+1}$ in $ p \mathcal{V}$ (as explained in  \cite[Proposition 5.17]{arXiv:1612.01312} when $p=0$ in $R$). From $ v(T^*_{\mathbf w.\mu})^{n+1}$ in $p \mathcal{V} $, we get   $v \otimes (T^{\mathbf w.M,*}_{\mathbf w.\mu})^{n+1}=v(T^*_{\mathbf w.\mu})^{n+1}\otimes 1_{\mathcal{H}_{ \mathbf w.M}}$  in $p\mathcal{V}\otimes_{\mathcal{H}_{( \mathbf w.M)^-}, \theta^*}\mathcal{H}_{ \mathbf w.M}$. As $T^{\mathbf w.M,*}=T^{\mathbf w.M}$ is invertible in $\mathcal{H}_{ \mathbf w.M}$ we get  $v \otimes 1_{\mathcal{H}_{ \mathbf w.M}}$ in $  p\mathcal{V}\otimes_{\mathcal{H}_{( \mathbf w.M)^-}, \theta^*}\mathcal{H}_{ \mathbf w.M}$. As $v$ was arbitrary, $\mathcal{V}\otimes_{\mathcal{H}_{( \mathbf w.M)^-}, \theta^*}\mathcal{H}_{ \mathbf w.M}\subset p \mathcal{V}\otimes_{\mathcal{H}_{( \mathbf w.M)^-}, \theta^*}\mathcal{H}_{ \mathbf w.M}$. If $p$ is nilpotent in $\mathcal{V}$, then $ \mathcal{V}\otimes_{\mathcal{H}_{( \mathbf w.M)^-}, \theta^*}\mathcal{H}_{ \mathbf w.M}=0$. Suppose now that  there exists  $n>0$ such that  $\mathcal{V}(z (t))^n=0$   for any   non-invertible $t\in T^{ +}$, then $\mathcal{V} T_\mu^{n+1} \subset p \mathcal{V} $ where $\mu = \mu_t$; hence $\varphi (h)=\varphi (h T^M_{\mu^{-n-1}} )T_\mu^{n+1}$ in $p\mathcal{V}$ for an arbitrary
 $\varphi\in    \Hom_{\mathcal{H}_{M^+},\theta}(\mathcal{H}_M,\mathcal{V})$ and an arbitrary $h\in \mathcal{H}_M$.  We deduce $ \Hom_{\mathcal{H}_{M^+},\theta}(\mathcal{H}_M,\mathcal{V}) \subset    \Hom_{\mathcal{H}_{M^+},\theta}(\mathcal{H}_M,p\mathcal{V})$. If $p$ is nilpotent in $\mathcal{V}$, then $ \Hom_{\mathcal{H}_{M^+},\theta}(\mathcal{H}_M,\mathcal{V})=0$.
  \end{proof}

 Recalling that  $\tilde { \mathbf w}^M.\mathcal V$ is obtained  by functoriality from $\mathcal{V}$ and the  ring isomorphism  $\iota ( \tilde{ \mathbf w}^M)$ defined in  \eqref{eq:isoH}, the equivalence between $\mathcal{V}$ supersingular   and  $\tilde { \mathbf w}^M\mathcal V$  supersingular follows from:

\begin{lemma}
\begin{enumerate}
\item Let $t\in T$. Then $t$ is dominant for $U_M$ if and only  if $\hat{ \mathbf w}^M t ( \hat{ \mathbf w}^M)^{-1} \in T $ is dominant for $U_{\mathbf w.M}$.
\item The $R$-algebra isomorphism   $\mathcal{H}_{M,R} \xrightarrow {\iota ( \tilde{ \mathbf w}^M)}\mathcal{H}_{ \mathbf w .M, R}$, 
 $ T^M_{w}\mapsto T^{\mathbf w.M}_{ \tilde{ \mathbf w}^M w ( \tilde{ \mathbf w}^M)^{-1}}$ for $w\in W_M(1) $ sends $z^M(t)$ to $z^{\mathbf w.M}(\hat{ \mathbf w}^M t ( \hat{ \mathbf w}^M)^{-1})$ for $t\in T$ dominant for $U_M$.\end{enumerate}
\end{lemma}
   
\begin{proof}    The conjugation by $\hat{ \mathbf w}^M $ stabilizes $T$, sends $U_M$ to $U_{\mathbf w.M}$ and sends the $\mathbb W_{M}$-orbit  of $t\in T$ to the 
$\mathbb W_{{\mathbf w}.M}$-orbit of $\hat{ \mathbf w}^M t ( \hat{ \mathbf w}^M)^{-1}$, as  ${\mathbf w}^M  \mathbb W_M  ({\mathbf w}^M)^{-1}=\mathbb W_{{\mathbf w}.M}$.
It is known that $\iota ( \tilde{ \mathbf w}^M)$ respects the antidominant  alcove walk bases    \cite[Proposition~2.20]{MR3437789}:   it sends $E^M(w)$ to $E^{\mathbf w.M} ( \tilde{ \mathbf w}^M w ( \tilde{ \mathbf w}^M)^{-1})$  for $w\in W_M(1) $.   \end{proof}

  We deduce:  
  \begin{corollary} \label{cor:9.29}Let $\mathcal{V}$ be a right $\mathcal{H}_{M,R}$-module. Then  $\mathcal{V}$ is supersingular if and only if the right $\mathcal{H}_{{ \mathbf w}.M,R}$-module $\tilde { \mathbf w}^M\mathcal V$ is supersingular.
  \end{corollary}
  
 Assume $R=C$. The supersingular simple $\mathcal{H}_{M,C}$-modules  are classified in \cite{Vigneras-prop-III}. By Corollaries \ref{cor:isoCI} and  \ref{cor:9.29},  the classification of the simple  $\mathcal{H}_C $-modules in \cite{arXiv:1406.1003_accepted} remains valid with the $\mathcal{H}_C$-modules $I_{\mathcal{H}}(P,\mathcal V,Q)$ instead of $ \mathcal CI_{\mathcal{H}}(P,\mathcal V,Q)$:
  
 \begin{corollary}[Classification of simple $\mathcal{H}_C$-modules] \label{cor:suH} Assume $R=C$.  Let   $(P,\mathcal V,Q)$ be a $\mathcal{H}_C$-triple where $ \mathcal V $  is simple and supersingular. Then, the $\mathcal{H}_C$-module $ \mathcal I_{\mathcal{H}}(P,\mathcal V,Q)$ is  simple. A simple $\mathcal{H}_C$-module is isomorphic to
$ \mathcal I_{\mathcal{H}}(P,\mathcal{V},Q)$ for  a  $\mathcal{H}_C$-triple  $(P,\mathcal{V},Q)$ where $ \mathcal V $  is simple and supersingular,  $P,Q$ and the isomorphism class of $\mathcal{V}$ are unique.
\end{corollary}

\subsection{A commutative diagram} \label{S:comdia}
We prove in this section Proposition \ref{prop:comp}. For  $Q\subset Q'\subset P(\mathcal{V})$
 we show by an explicit computation that
\begin{equation*} \mu_Q^{-1}\circ i(Q,Q') \circ \mu_{Q'}:
\Ind_{\mathcal{H}_{  \mathbf w.Q '}}^\mathcal{H}(e_{\mathcal{H}_{\mathbf w.Q'}}( \tilde{ \mathbf w}^M.\mathcal V))\to\Ind_{\mathcal{H}_{  \mathbf w.Q }}^\mathcal{H}(e_{\mathcal{H}_{\mathbf w.Q}}( \tilde{ \mathbf w}^M.\mathcal V)).
\end{equation*}
is equal to $\iota(\mathbf w.Q,\mathbf w.Q')$. The  $R$-module $e_{\mathcal{H}_{\mathbf w.Q'}}( \tilde{ \mathbf w}^M.\mathcal V)\otimes 1_\mathcal{H}$ generates the $\mathcal{H}_R$-module $e_{\mathcal{H}_{\mathbf w.Q'}}( \tilde{ \mathbf w}^M.\mathcal V)\otimes_{\mathcal{H}_{{ \mathbf w}.Q', R}, \theta^+} \mathcal{H}_R=\Ind_{\mathcal{H}_{  \mathbf w.Q '}}^\mathcal{H}(e_{\mathcal{H}_{\mathbf w.Q'}}( \tilde{ \mathbf w}^M.\mathcal V))$  and by \eqref{eq:iotaQQ'})
\begin{equation}\label{iota(v1)} \iota(\mathbf w.Q,\mathbf w.Q') (v \otimes 1_{\mathcal{H}} ) = v \otimes   \sum_{d \in  {}^{\mathbb W_{M_{\mathbf w.  Q}}}\mathbb W_{M_{\mathbf w. Q'}}} T_{\tilde d}\end{equation} for $v\in \mathcal{V}$ seen as an element of 
 $e_{\mathcal{H}_{\mathbf w.Q'}}( \tilde{ \mathbf w}^M.\mathcal V)$ in the LHS and an element of $e_{\mathcal{H}_{\mathbf w.Q}}( \tilde{ \mathbf w}^M.\mathcal V)$ in the RHS.

  \begin{lemma}\label{lemma:jeannette}  
$( \mu_Q^{-1}\circ i(Q,Q') \circ \mu_{Q'})(v \otimes 1_{\mathcal{H}} )  =v \otimes  \sum_{d\in \mathbb W_{M_{Q'}}^{\mathbb W_{M_Q}}}  \, q_d \,T^*_{ \tilde {\mathbf w}^{Q} (\tilde{ \mathbf w}^{Q' }\tilde d)^{-1}}$.
  \end{lemma}

\begin{proof}   $\mu_{Q'}(v \otimes 1_{\mathcal{H}} )$ is the unique homomorphism
 $ f_{\tilde {\mathbf w}^{M_{Q'}}, v}
\in  \Hom_{\mathcal{H}_{ M_{Q'} ^-},\theta^*}(\mathcal{H}, e_{\mathcal{H}_{Q'}}( \mathcal V))$ sending $T_{\tilde {\mathbf w}^{Q'}}$ to  $v$ and vanishing on $T_{\tilde d'}$ for $d'\in \mathbb W^{ \mathbb W _{M_{Q'}}}\setminus \{ {\mathbf w}^{Q'} \}$ by  \eqref{eq:muQ}. By \eqref{eq:canonicalH},
$i(Q,Q')$ is the natural embedding of $ \Hom_{\mathcal{H}_{ M_{Q'}^-, \theta^*}}(\mathcal{H}, e_{\mathcal{H}_{Q'} }(\mathcal V)) $ in $ \mathcal \Hom_{\mathcal{H}_{M_{Q}^-, \theta^*}}(\mathcal{H},e_{\mathcal{H}_{Q}}(\mathcal V))$  therefore $i(Q,Q')( f_{\tilde {\mathbf w}^{M_{Q'}}, v})$  is the unique homomorphism $ \mathcal \Hom_{\mathcal{H}_{M_{Q}^-, \theta^*}}(\mathcal{H},e_{\mathcal{H}_{Q}}(\mathcal V))$  sending $T_{\tilde {\mathbf w}^{Q'}}$ to  $v$ and vanishing on $T_{\tilde d'}$ for $d'\in \mathbb W^{ \mathbb W _{M_{Q'}}}\setminus \{ {\mathbf w}^{Q'} \}$. As $\mathbb W^{\mathbb W_{M_Q}}= \mathbb W^{\mathbb W_{ Q'}}\mathbb W_{M_{Q'}}^{\mathbb W_{M_Q}}$, this homomorphism vanishes on  $T_{\tilde w} $ for 
$w$ not in  $ \mathbf w^{M_{Q'} }\mathbb W_{M_{Q'}}^{\mathbb W_{M_Q}}$.
By \cite[Lemma 2.22]{arXiv:1406.1003_accepted}, the inverse of $\mu_Q$ is the $\mathcal{H}_R$-isomorphism:
 \begin{equation}\label{eq:muQ-1}
 \Hom_{\mathcal{H}_{ M_Q ^-},\theta^*}(\mathcal{H},e_{\mathcal{H}_{Q}}( \mathcal V))\xrightarrow{\mu_Q^{-1}} \Ind_{\mathcal{H}_{  \mathbf w.M_Q }}^\mathcal{H}(e_{\mathcal{H}_{\mathbf w.Q}}( \tilde{ \mathbf w}^M.\mathcal V)) \end{equation}
$$\quad  \quad  \quad  \quad  \quad  \quad  \quad  \quad  \quad f\mapsto \sum_{d\in \mathbb W^{\mathbb W_{M}}} f (T_{\tilde d})\otimes T^*_{ \tilde {\mathbf w}^M \tilde d^{-1}} , 
$$
where $ \mathbb W^{\mathbb W_{M}}$ is the set of $d\in  \mathbb W$ with minimal length in the coset $d\mathbb W_{M}$. We deduce the explicit formula:
 $$( \mu_Q^{-1}\circ i(Q,Q') \circ \mu_{Q'})(v \otimes 1_{\mathcal{H}} )  =  \sum_{w\in \mathbb W^{\mathbb W_{M_Q}}}  i(Q,Q')(f^{Q'}_{\tilde {\mathbf w}^{M_{Q'}}, v})(T_{\tilde w})\otimes T^*_{ \tilde {\mathbf w}^{M_Q} \tilde w^{-1}} .
 $$
 Some terms are zero:  the terms for $w\in  \mathbb W^{\mathbb W_{M_Q}}$ not in  $\mathbf w^{M_{Q'} }\mathbb W_{M_{Q'}}^{\mathbb W_{M_Q}}$. We analyse the other terms for
$ w$ in $  \mathbb W^{\mathbb W_{M_Q}}\cap \mathbf w^{M_{Q'} }\mathbb W_{M_{Q'}}^{\mathbb W_{M_Q}}$; this set is  $\mathbf w^{M_{Q'} } \mathbb W_{M_{Q'}}^{\mathbb W_{M_Q}}$. Let $w=\mathbf w^{M_{Q'} } d, d\in \mathbb W_{M_{Q'}}^{\mathbb W_{M_Q}}$, and 
 $\tilde w=\tilde{ \mathbf w}^{M_{Q'} }\tilde d$ with $ \tilde d\in {}_1 W_{G'}$ lifting $d $. 
By the braid relations $T_{\tilde w}=T_{\tilde {\mathbf w}^{M_{Q'}}  }T_{\tilde d}$. We have      $T_{ \tilde d} = \theta^*(T^{M_{Q'}}_{ \tilde d} )$ by the braid relations because  $d\in \mathbb W_{ M_{Q'}}$,  $S_{M_{Q'}} \subset S^{\aff}$ and $\theta^*( c^{M_{Q'}}_{ \tilde s}) =c _{ \tilde s}$  for $s\in S_{M_{Q'}} $.  
As $\mathbb W_{ M_{Q'}}\subset  W_{ M_{Q'} ^-}\cap W_{ M_{Q'} ^+}$, we deduce:   \begin{align*} i(Q,Q')(f^{Q'}_{\tilde {\mathbf w}^{M_{Q'}}, v})(T_{\tilde w})&= i(Q,Q')(f^{Q'}_{\tilde {\mathbf w}^{M_{Q'}}, v})(T_{\tilde {\mathbf w}^{M_{Q'}}}T_{ \tilde d}) = i(Q,Q')(f^{Q'}_{\tilde {\mathbf w}^{M_{Q'}}, v})(T_{\tilde {\mathbf w}^{M_{Q'}}  }) T^{M_{Q'}}_{\tilde d}\\
&= vT^{M_{Q'}}_{\tilde d} = q_d v.
\end{align*}
Corollary \ref{cor:extT} gives the last equality.
  \end{proof}
  
The  formula  for $( \mu_Q^{-1}\circ i(Q,Q') \circ \mu_{Q'})(v \otimes 1_{\mathcal{H}} )$ given in Lemma \ref{lemma:jeannette} is different  from the formula  \eqref{iota(v1)} for $\iota(\mathbf w.Q,\mathbf w.Q') (v \otimes 1_{\mathcal{H}} )$.
 It  needs  some work to prove that they are equal.
   
 A first reassuring  remark is that  $  {}^{\mathbb W_{M_{\mathbf w.  Q}}}\mathbb W_{M_{\mathbf w. Q'}}= \{ \mathbf w  d^{-1} \mathbf w \ | \ d\in \mathbb W_{M_{Q'}}^{\mathbb W_{M_Q}}\}$, so the two summation sets have the same number of elements. But better,
$$  {}^{\mathbb W_{M_{\mathbf w.  Q}}}\mathbb W_{M_{\mathbf w. Q'}}= \{  \mathbf w^{Q}   (\mathbf w^{Q'}d)^{-1} \ | \ d\in \mathbb W_{M_{Q'}}^{\mathbb W_{M_Q}}\}$$
because  $ \mathbf w_{Q'}  \mathbb W_{M_{Q'}}^{\mathbb W_{M_Q}}  \mathbf w_{Q} =  \mathbb W_{M_{Q'}}^{\mathbb W_{M_Q}}$. To prove the latter equality, we apply  the criterion:   $w\in  \mathbb W_{M_{Q'}}$ lies in $ \mathbb W_{M_{Q'}}  {}^{\mathbb W_{M_Q }}$ if and only if  $w(\alpha)>0$ for all $\alpha \in \Delta_Q$ noticing that  $ d\in  \mathbb W_{M_{Q'}}^{\mathbb W_{M_Q}} $ implies $ \mathbf w_{Q} (\alpha)\in - \Delta_Q$, $d \mathbf w_{Q} (\alpha) \in -\Phi_{M_{Q'}}$,  $\mathbf w_{Q'} d \mathbf w_{Q} (\alpha)>0$.
Let $x_d = \mathbf{w}^Q   (\mathbf{w}^{Q'} d)^{-1}$. We have  $\tilde {\mathbf w}^{M_Q} (\tilde{ \mathbf w}^{M_{Q'} }\tilde d)^{-1}=\tilde x_d$  because the lifts $\tilde w$ of the elements $w\in \mathbb W$ satisfy the braid relations  and $\ell(x_d)=\ell( \mathbf w_{Q} d^{-1} \mathbf w_{Q'} )= \ell (\mathbf w_{Q'} )- \ell ( \mathbf w_{Q} d^{-1})= \ell (\mathbf w_{Q'} )- \ell ( \mathbf w_{Q}) -\ell ( d^{-1}) =  \ell (\mathbf w_{Q'} )- \ell ( \mathbf w_{Q}) -\ell ( d )=  -\ell (\mathbf w^{Q'})+ \ell ( \mathbf w^{Q}) -\ell ( d )$. 
We have $q_d=  q_{ \mathbf w_{\mathbf w.Q}x_d\mathbf w_{\mathbf w.Q'}} $ because  $\mathbf w d^{-1} \mathbf w= \mathbf w_{\mathbf w.Q}x_d\mathbf w_{\mathbf w.Q'}$, and 
 $q_d=q_{ d^{-1}}= q_{\mathbf w d^{-1} \mathbf w}$.
 So
 $$ \sum_{d\in \mathbb W_{M_{Q'}}^{\mathbb W_{M_Q}}}   q_d T^*_{ \tilde {\mathbf w}^{Q} (\tilde{ \mathbf w}^{Q' }\tilde d)^{-1}} = \sum_{x_d \in  {}^{\mathbb W_{M_{\mathbf w.  Q}}}\mathbb W_{M_{\mathbf w. Q'}}}  q_{ \mathbf w_{\mathbf w.Q}x_d\mathbf w_{\mathbf w.Q'}}T^*_{ \tilde x_d}.$$ 
In the RHS, only $ \tilde {\mathbf w}^M.\mathcal{V},{\mathbf w}.Q,{\mathbf w}.Q'$ appear. The same holds true in the formula  \eqref{iota(v1)}. The map
 $(P,\mathcal{V},Q,Q') \mapsto ({\mathbf w}.P,\tilde {\mathbf w}^M.\mathcal{V},{\mathbf w}.Q,{\mathbf w}.Q')$ is a bijection of the set of triples $(P,\mathcal{V},Q,Q')$ where $P=MN,Q,Q'$ are standard parabolic subgroups of $G$, $\mathcal{V}$ a right $\mathcal{H}_R$-module, $Q\subset Q'\subset P(\mathcal{V})$ by Lemma  \ref{lemma:delt}.  So we can  replace $({\mathbf w}.P,\tilde {\mathbf w}^M.\mathcal{V},{\mathbf w}.Q,{\mathbf w}.Q')$   by $(P,\mathcal{V},Q,Q')$.  Our task is  reduced to prove   in  $e_{\mathcal{H}_{Q}}(\mathcal{V}) \otimes_{\mathcal{H}_{M^+_Q}, \theta} \mathcal{H}_R$: 
\begin{equation}\label{eq:finalrs}
 v \otimes   \sum_{d \in  {}^{\mathbb W_{M_{  Q}}}\mathbb W_{M_{ Q'}}} T_{\tilde d}= v \otimes   \sum_{d \in  {}^{\mathbb W_{M_{  Q}}}\mathbb W_{M_{  Q'}}}q_{\mathbf w_{Q}d\mathbf w_{Q'}} T^*_{\tilde d}. 
\end{equation}
  A second  simplification is possible: we can replace   $Q\subset Q'$  by the standard  parabolic subgroups
  $Q_2\subset Q'_2 $   of $G$  with $\Delta_{Q_2}= \Delta_{Q}\setminus \Delta_P$ and $\Delta_{Q'_2}= \Delta_{Q'}\setminus \Delta_P$, because  $\Delta_P$ and  $\Delta_{P(\mathcal{V})}\setminus \Delta_P$ are orthogonal.  Indeed, 
  $\mathbb W_{M_{ Q'}} = \mathbb W_{M} \times  \mathbb W_{M_{Q'_2}}$ and $\mathbb W_{M_{ Q}}=\mathbb W_{M} \times  \mathbb W_{M _{Q_2}}$ are direct products, the longest elements $\mathbf w_{Q'}=\mathbf w_{M}\mathbf w_{Q'_2}, \mathbf w_{Q}=\mathbf w_{M}\mathbf w_{Q_2}$ are direct products and 
  $${}^{\mathbb W_{M_{  Q}}}\mathbb W_{M_{ Q'}}={}^{\mathbb W_{M_{  Q _2}}}\mathbb W_{M_{ Q' _2}}, \quad \mathbf w_{Q}d\mathbf w_{Q'}=   \mathbf w_{Q_2 } d \mathbf w_{ Q'_2} .$$

Once this is done, we use the properties of $e_{\mathcal{H}_{Q}}(\mathcal{V})$:    $vh\otimes 1_\mathcal{H}= v\otimes \theta_Q(h)$ for $h\in \mathcal{H}_{M^+_{Q_2}}$, and    $T^{Q,*}_w$  acts trivially on $e_{\mathcal{H}_{Q}}(\mathcal{V})$ for $w\in {}_1W_{M_{Q_2}'} \cup (\Lambda(1)\cap {}_1 W_{M_{Q'_2}'})$.
Set ${}_1\mathbb{W}_{M'_{Q'_2}} = \{w\in {}_1W_{M'_{Q'_2}}\mid \text{$w$ is a lift of some element in $\mathbb{W}_{M_{Q'_2}}$}\}$ and ${}_1\mathbb{W}_{M'_{Q_2}}$ similarly.
Then $Z_k\cap {}_1\mathbb{W}_{M'_{Q'_2}}\subset (\Lambda(1)\cap {}_1W_{M'_{Q'_2}})\cap {}_1W_{M^+_{Q_2}}$ and ${}_1\mathbb{W}_{M'_{Q_2}}\subset {}_1W_{M'_{Q_2}}\cap {}_1W_{M^+_{Q_2}}$.
This implies that  \eqref{eq:finalrs}  where $Q\subset Q'$ has been replaced by $Q_2\subset Q'_2$ follows from a congruence
\begin{equation}\label{eq:finalrss} \sum_{d \in  {}^{\mathbb W_{M_{  Q_2}}}\mathbb W_{M_{ Q'_2}}} T_{\tilde d}\equiv \sum_{d \in  {}^{\mathbb W_{M_{  Q_2}}}\mathbb W_{M_{  Q'_2}}}q_{\mathbf w_{Q_2}d\mathbf w_{Q'_2}} T^*_{\tilde d}.
\end{equation}
in the finite subring $H({}_1\mathbb{W}_{M_{ Q'_2}})$ of $\mathcal{H}$ generated by $\{T_w\mid w\in {}_1\mathbb{W}_{M'_{Q'_2}}\}$ modulo the  the right ideal  $\mathcal J_2$ with generators  $\{\theta_Q(T^{Q,*}_w)-1\mid  w\in (Z_k\cap {}_1\mathbb{W}_{M'_{Q'_2}}) \cup {}_1 \mathbb{W}_{M'_{Q_2}} \}$.
 
Another simplification  concerns  $T^*_{\tilde d}$ modulo $\mathcal J_2$ for $d\in \mathbb W_{M_{  Q'_2}}$. We recall that for any reduced decomposition $d=s_1\ldots s_n$ with $s_i \in S\cap  \mathbb W_{M_{Q'_2}}$ we have $T^*_{\tilde d}=  (T_{\tilde s_1}-c_{\tilde s_1} ) \ldots (T_{\tilde s_n}-c_{\tilde s_n} )$  where  the  $\tilde s_i$ are admissible. For $\tilde s$ admissible, by \eqref{eq:cs}
$$ c_{\tilde s}\equiv  q_{s}-1 .$$
Therefore $$ T_d^* \equiv  (T_{\tilde s_1}-q_{s_1}+1 ) \dotsm (T_{\tilde s_n}-q_{ s_n}+1 ).$$
Let $\mathcal{J}'\subset \mathcal{J}_2$ be the ideal of $H({}_1\mathbb{W}_{M'_{Q'_2}})$ generated by $\{T_t - 1\mid t\in Z_k\cap {}_1W_{M'_{Q'_2}}\}$.
Then the ring  $H({}_1\mathbb{W}_{M'_{ Q'_2}})/\mathcal{J}'$ and its right  ideal $\mathcal{J}_2/\mathcal{J}'$ are the specialisation of    the generic finite ring $H(\mathbb W_{M_{ Q'_2}})^{g}$  over $\mathbb Z[(q_s)_{s\in S_{M_{  Q'_2}}}]$ where the $q_s$ for $s\in S_{M_{  Q'_2}}=S\cap \mathbb W_{M_{ Q'_2}} $ are indeterminates, and of its right ideal  $\mathcal J_2^{g}$ with the same generators. The similar congruence modulo $\mathcal J_2^{g}$ in $H(\mathbb W_{M_{ Q'_2}})^{g}$ (the generic congruence)  implies the congruence \eqref{eq:finalrss} by specialisation.

We will prove the generic congruence in a  more general setting where $H$ is the generic Hecke ring of  a finite Coxeter system$( \mathbb W, S)$ and   parameters  $(q_s)_{s\in S}$ such that $q_s=q_{s'}$ when $s,s'$ are conjugate in  $\mathbb W$. The Hecke ring $H $ is a $\mathbb Z[(q_s)_{s\in S}] $-free module of basis $(T_w)_{w\in  \mathbb W}$ satisfying the braid relations and the quadratic relations $T_s^2=q_s+(q_s-1)T_s$ for $s\in S$. The other basis $(T^*_w)_{w\in  \mathbb W}$ satisfies the braid relations and the quadratic relations $(T^*_s)^2=q_s-(q_s-1)T^*_s$ for $s\in S$, and is related to the first basis by $T_s^*= T_s-(q_s-1)$  for $s\in S$, and more generally $T_w T^*_{w^{-1}}= T^*_{w^{-1}} T_w= q_w$ for $w\in \mathbb W$ \cite[Proposition 4.13]{MR3484112}. 

Let $J\subset S$ and $\mathcal J$ is the right ideal of $H$ with generators  $T_w^*-1$ for all $w$ in the group $ \mathbb W_J$ generated by $J$. 

\begin{lemma}\label{lemma:ba}  A basis of $\mathcal J$ is $(T^*_{w_1}-1)T^*_{w_2}$ for $w_1\in \mathbb W_J\setminus \{1\}, w_2\in {}^{\mathbb W_J}\mathbb W $, and  adding $T^*_{w_2}$ for $w_2\in {}^{\mathbb W_J}\mathbb W$ gives a basis of $H$. In particular, 
$\mathcal J$ is a direct factor of $\mathcal{H}$.
\end{lemma}
\begin{proof} The elements $(T^*_{w_1}-1)T^*_{w }$ for $w_1\in \mathbb W_J, w \in {} \mathbb W $ generate $\mathcal J$.   We write  $w=u_1w_2$ with unique elements $u_1\in \mathbb W_J, w_2\in {}^{\mathbb W_J}\mathbb W $, and $T^*_w=T^*_{u_1}T^*_{w_2}$. Therefore,  $(T^*_{w_1}-1)T^*_{u_1 }T^*_{w_2 }$. 
By an induction on the length of $u_1$, one proves that $(T^*_{w_1}-1)T^*_{u_1 }$ is a linear combination of $(T^*_{v_1}-1)$ for $v_1\in  \mathbb W_J$ as in the proof of Proposition \ref{prop:ideal}.
It is clear that the elements $(T^*_{w_1}-1)T^*_{w_2}$ and $ T^*_{w_2}$ for  $w_1\in \mathbb W_J\setminus \{1\}, w_2\in {}^{\mathbb W_J}\mathbb W $ form a basis of $H$. 
 \end{proof}
    
 Let $\mathbf w_J$ denote the longest element of  $ \mathbb W_J$ and  $\mathbf w=  \mathbf w_S$. 
 \begin{lemma} In the generic Hecke ring $H$,  the congruence modulo  $\mathcal J$ $$ \sum_{d \in  {}^{\mathbb W_J}\mathbb W } T_{d}\equiv  \sum_{d \in  {}^{\mathbb W_J}\mathbb W }q_{\mathbf w_Jd\mathbf w}  T^*_{d} $$
holds true. 
\end{lemma}

\begin{proof}   
Step 1. We show:
$$ {}^{\mathbb W_J}\mathbb W={ \mathbf w_J  } {}^{\mathbb W_J}\mathbb W  \,\mathbf w,\quad  q_{\mathbf  w_J}q_{\mathbf w_Jd\mathbf w}T_d^*=   T_{\mathbf  w_J}T_{\mathbf  w_J d{\mathbf w} } T^*_{\mathbf w }.$$
The equality between the groups  follows from the characterisation of $ {}^{\mathbb W_J}\mathbb W$ in $\mathbb W$: an element $d\in \mathbb W$ has minimal length in $\mathbb W_J d$ if and only if  $\ell(u d)=\ell (u)+\ell(d)$ for  all $u\in \mathbb W_J$. An easy computation shows that  $\ell(u { \mathbf w_J  }d\mathbf w)=\ell (u)+\ell({ \mathbf w_J  }d\mathbf w)$ for all $u\in \mathbb W_J, d\in  {}^{\mathbb W_J}\mathbb W$ (both sides are equal to $\ell (u)+\ell(\mathbf w)- \ell(({ \mathbf w_J  })-\ell(d)$). The second equality  follows from $q_{\mathbf  w_J} q_{\mathbf  w_J d\mathbf w}= q_{ d\mathbf w}$ because 
$(\mathbf  w_J)^2=1  $ and  $\ell (\mathbf w_J) +\ell (\mathbf  w_J d\mathbf w )=\ell( d\mathbf w)$ (both sides are $\ell( \mathbf w) -\ell( d)$) and from
   $ q_{ d\mathbf w}T^*_d = T_{ d\mathbf w}T^*_{\mathbf w d^{-1}} T^*_d=  T_{ d\mathbf w} T^*_{\mathbf w }$.
   We also have $T_{d\mathbf{w}} = T_{\mathbf{w}_J}T_{\mathbf{w}_Jd\mathbf{w}}$.
   
Step 2.    The multiplication by $q_{\mathbf  w_J}$ on the quotient $H /\mathcal J $ is injective (Lemma \ref{lemma:ba}) and $q_{\mathbf  w_J}   \equiv T _{\mathbf  w_J} $. By Step 1, 
$q_{\mathbf w_Jd\mathbf w}T_d^* \equiv  T_{\mathbf  w_J d{\mathbf w} } T^*_{\mathbf w }$ and 
$$ \sum_{d \in  {}^{\mathbb W_J}\mathbb W }q_{\mathbf w_Jd\mathbf w}  T^*_{d}  \equiv   \sum_{d \in  {}^{\mathbb W_J}\mathbb W } T_{d}T^*_{\mathbf w }.$$
 The congruence 
\begin{equation}\label{eq:congJ} \sum_{d \in  {}^{\mathbb W_J}\mathbb W } T_{d} \equiv   \sum_{d \in  {}^{\mathbb W_J}\mathbb W } T_{d}T^*_{s }
\end{equation}
 for all  $s\in S$ implies the lemma because $T^*_{\mathbf w }=T^*_{s_1 }\ldots T^*_{s_n}$  for any reduced decomposition  $\mathbf w=s_1\ldots s_n$ with $s_i\in S$.

Step 3. When $J=\emptyset$,  the congruence \eqref{eq:congJ}  is an equality:  \begin{equation}\label{eq:cong} \sum_{w \in  \mathbb W } T_{w} =   \sum_{w \in  \mathbb W } T_{w}T^*_{s }.
\end{equation}
 It holds true because
 $
  \sum_{w\in \mathbb{W}}T_w = \sum_{w < ws}T_w(T_s + 1)
 $
 and
 $
  (T_s + 1)T_s^* = T_sT_s^* + T_s^* = q_s + T_s^* = T_s + 1
 $.
   
 Step 4.  Conversely the congruence \eqref{eq:congJ} follows from \eqref{eq:cong} because
 $$\sum_{w \in  \mathbb W } T_{w}=(\sum_{u\in W_J} T_u)\sum_{d \in  {}^{\mathbb W_J}\mathbb W } T_{d} \equiv (\sum_{u\in W_J} q_u)\sum_{d \in  {}^{\mathbb W_J}\mathbb W } T_{d}$$
(recall $  q_u=T_{u^{-1}}^*T_u \equiv T_u$) and we can simplify by $\sum_{u\in W_J} q_u$ in $H/\mathcal J$. 
  \end{proof}
This ends the proof of Proposition \ref{prop:comp}.

   \section{Universal  representation  $I_\mathcal{H}(P,\mathcal V,Q)\otimes_\mathcal{H} R[\mathcal U\backslash G] $ }\label{S:10}
  The invariant functor $(-)^\mathcal{U}$ by the pro-$p$ Iwahori subgroup $\mathcal U$ of $G$ has a left adjoint
    $$- \otimes_{\mathcal{H}_R} R[\mathcal U\backslash G] : \Mod_R( \mathcal{H}) \to \Mod_R^\infty (G).$$ 
The smooth $R$-representation   $\mathcal{V} \otimes_{\mathcal{H}_R}R[\mathcal U\backslash G] $ of $G$ constructed from the right $\mathcal{H}_R$-module $\mathcal{V}$ is called universal.  We write 
$$R[\mathcal U\backslash G]=\mathbb X .$$

\begin{question}\label{que:non0} Does $\mathcal{V}\neq 0$ implies $\mathcal{V} \otimes_{\mathcal{H}_R} \mathbb X\neq 0$   ? or does $ v\otimes 1_{\mathcal{U} }=0 $ for $v\in \mathcal{V}$ implies $v=0$ ?  
 We have no counter-example.  If $R$ is a field and the $\mathcal{H}_{R}$-module $\mathcal{V}$ is simple, the two questions are equivalent:  $\mathcal{V} \otimes_{\mathcal{H}_R} \mathbb X\neq 0$ if and only if the map $v\mapsto   v\otimes 1_{\mathcal{U}}$ is injective. When $R=C$,    $\mathcal{V} \otimes_{\mathcal{H}_R} \mathbb X\neq 0$ for all simple $\mathcal{H}_{C}$-modules $\mathcal{V}$ if this is true for $\mathcal{V}$ simple supersingular (this is a consequence of Corollary \ref{cor:10.13}). 
\end{question}

The functor $- \otimes_{\mathcal{H}_R} \mathbb X$ satisfies a few good properties:   it has a right adjoint and is compatible with the parabolic induction and  the left adjoint (of the parabolic induction). Let  $P=MN$ be a  standard parabolic subgroup and $\mathbb X_M=R[\mathcal{U}_M \backslash  M]$. We have  functor isomorphisms 
\begin{align}&(- \otimes_{\mathcal{H}_R} \mathbb X)\circ   \mathcal Ind_{\mathcal{H}_M}^\mathcal{H} \to \Ind_P^G \circ (-\otimes_{\mathcal{H}_R} \mathbb X_M),\\
 \label{eq:JT}&(-)_N\circ (- \otimes_{\mathcal{H}_R} \mathbb X) \to(-\otimes_{\mathcal{H}_R} \mathbb X_M) \circ  \mathcal L_{\mathcal{H}_M}^\mathcal{H}.
\end{align}
The first one is \cite[formula 4.15]{arXiv:1703.04921}, the second one is obtained by left adjunction from the isomorphism $ \mathcal Ind_{\mathcal{H}_M}^\mathcal{H} \circ (-)^{\mathcal{U}_M}\to (-)^\mathcal{U}\circ \Ind_P^G $ \cite[formula (4.14)]{arXiv:1703.04921}.  If  $\mathcal{V}$ is a right $\mathcal{H}_R$-supersingular module and $p$ is nilpotent in $\mathcal{V}$, then $ \mathcal{L}_{\mathcal{H}_M}^\mathcal{H} (\mathcal{V})=0$ if $M\neq G$ (Proposition \ref{pro:ssLR}). Applying \eqref{eq:JT} we deduce:

 \begin{proposition}\label{pro:p0} If $p $  is nilpotent in $\mathcal{V}$ and  $\mathcal{V}$  supersingular,
 then $\mathcal{V} \otimes_{\mathcal{H}_R} \mathbb X$ is left cuspidal.
 \end{proposition}

\begin{remark}\label{rem:or} For a non-zero smooth $R$-representation   $\tau $ of $M$, $\Delta_\tau$ is  orthogonal to $\Delta_P$ if $\tau$ is left cuspidal. 
  Indeed, we recall from \cite[II.7 Corollary~2]{MR3600042} that  $\Delta_{ \tau}$  is not orthogonal to $\Delta_P$ if and only if it exists a proper standard parabolic subgroup $X $  of $M$ such that $\sigma$ is trivial on the unipotent radical of $X$; moreover $\tau$ is a subrepresentation of $\Ind_X^M (\tau|_X)$, so  the image of $\tau$ by the left adjoint of $\Ind_X^M$ is not $0$. 
\end{remark} 
 From now on,     $\mathcal{V}$ is a non-zero right $\mathcal{H}_{M,R}$-module and $$\sigma=\mathcal{V} \otimes_{\mathcal{H}_{M,R}} \mathbb X_M.$$
In general,  when $\sigma\neq 0$, let  $P_{\perp}(\sigma)$ be the standard parabolic subgroup    of  $   G$  with  $\Delta_{P_\perp(\sigma)}= \Delta_P\cup  \Delta_{\perp,\sigma}$  where $ \Delta_{\perp,\sigma}$ is the set of simple roots $\alpha \in \Delta_\sigma$ orthogonal to $\Delta_P$.
 
\begin{proposition}\label{prop:orthtriple} 
\begin{enumerate}
\item $P(\mathcal{V})\subset P_{\perp}(\sigma)$ if $\sigma\neq 0$. 
\item $   P(\mathcal{V})= P_{\perp}(\sigma)$ if the map $v\mapsto v\otimes 1_{\mathcal{U}_M}$ is injective.
\item $   P(\mathcal{V})=P(\sigma)$ if the map $v\mapsto v\otimes 1_{\mathcal{U}_M}$ is injective, $p $  nilpotent in $\mathcal{V}$ and $\mathcal{V}$  supersingular.
\item $   P(\mathcal{V})=P(\sigma)$ if $\sigma\neq 0 $, $R$ is a field of characteristic $p$ and $\mathcal{V}$ simple supersingular.
\end{enumerate}\end{proposition}

 \begin{proof} (1) $   P(\mathcal{V})\subset P_{\perp}(\sigma)$ means that
$ Z\cap M'_{\mathcal V} \ \text{ acts trivially on  } \ \mathcal V\otimes  1_{\mathcal{U}_M} , $
where $M_{\mathcal{V}}$ is the standard Levi subgroup such that $\Delta_{M_{\mathcal{V}}} = \Delta_{\mathcal{V}}$.
Let  $z\in Z\cap M'_{\mathcal{V}}$ and $v\in \mathcal V$. As  $\Delta_M$ and $\Delta_{\mathcal V}$  are orthogonal, we have $T ^{M,*}(z)= T ^{M}(z)$  and $\mathcal{U}_Mz\mathcal{U}_M=\mathcal{U}_M z$.  We have $v\otimes   1_{\mathcal{U}_M}= v T ^{M}(z)\otimes   1_{\mathcal{U}_M}=v\otimes T ^{M}(z)  1_{\mathcal{U}_M}=v\otimes  {\bf 1}_{\mathcal{U}_Mz}= v\otimes  z^{-1} 1_{\mathcal{U}_M}= z^{-1}(v\otimes   1_{\mathcal{U}_M})$.
 
(2)   If $ v\otimes  1_{\mathcal{U}_M}=0 $ for $v\in \mathcal{V}$ implies $v=0$, then $\sigma\neq 0$ because $\mathcal{V}\neq 0$. By (1) $P(\mathcal V) \subset  P_{\perp}(\sigma) $. As in  the proof of (1), for $z\in  Z\cap M'_{\perp, \sigma}$ we have  $ v T ^{M,*}(z)\otimes   1_{\mathcal{U}_M}= v T ^{M}(z)\otimes   1_{\mathcal{U}_M}= v \otimes   1_{\mathcal{U}_M}$ and our hypothesis implies  $v T ^{M,*}(z)=v$ hence $P(\mathcal V) \supset  P_{\perp}(\sigma) $.

(3)  Proposition \ref{pro:p0}, Remark \ref{rem:or} and (2).

(4) Question \ref{que:non0} and (3).
   \end{proof}
 Let  $Q$ be a parabolic subgroup of $G$   with $P\subset Q\subset P(\mathcal{V})$.  
 In this chapter  we   will compute  $I_\mathcal{H}(P,\mathcal V,Q)\otimes_\mathcal{H} R[\mathcal{U}\backslash  G]$ where   $I_\mathcal{H}(P,\mathcal{V},Q)=\Ind_{\mathcal{H}_{M(\mathcal{V})}}^\mathcal{H} (e(\mathcal{V}) \otimes (\Ind_Q^{P(\mathcal{V})} {\bf 1})^{\mathcal{U}_{M(\mathcal{V})} } )$ (Theorem \ref{thm:10.10}).
 The smooth $R$-representation $I_G(P,\sigma,Q)$  of $G$ is well defined: it is $0$ if $\sigma=0$ and  $\Ind_{P(\sigma)}^G (e(\sigma)\otimes \St_Q^{P(\sigma)})$ if $\sigma\neq 0$ because $(P,\sigma,Q)$ is an $R[G]$-triple by Proposition \ref{pro:p0}. We   will  show that  the universal representation  $I_\mathcal{H}(P,\mathcal V,Q)\otimes_\mathcal{H} R[\mathcal{U}\backslash  G]$ is  isomorphic to   $I_G(P,\sigma,Q)$,  if  $P(\mathcal{V})=P(\sigma)$  and $p=0$, or if $\sigma=0$ (Corollary \ref{cor:10.12}). In particular, when $R=C$ and $I_\mathcal{H}(P,\mathcal V,Q)\otimes_\mathcal{H} R[\mathcal{U}\backslash  G] \simeq I_G(P,\sigma,Q)$ when $\mathcal{V}$ is supersingular

\subsection{$Q=G$} \label{S:Q=G}  We consider first the case  $Q=G$.   We are in the simple situation where $\mathcal{V}$ is extensible to $\mathcal{H}$ and $P(\mathcal V)=P(\sigma)=G$,  $I_\mathcal{H}(P,\mathcal V,G) = e(\mathcal V)$ and  $I_G(P,\sigma,G) =e(\sigma)$. We recall that $\Delta \setminus \Delta_P$ is orthogonal to $\Delta_P$ and that  $M_2$
denotes  the  standard Levi subgroup   of $G$ with $\Delta_{M_2}=\Delta \setminus \Delta_P$.

The $\mathcal{H}_{ R}$-morphism $e (\mathcal V)\to e ( \sigma)^\mathcal{U}=\sigma^{\mathcal{U}_M}$ sending $v $ to $v\otimes  1_{\mathcal{U}_M}$ for $v\in  \mathcal V$, gives by adjunction an $R[G]$-homomorphism 
$$
v\otimes {\bf 1}_{\mathcal{U}} \mapsto v\otimes  1_{\mathcal{U}_M} :  e (\mathcal V) \otimes_{\mathcal{H}_{ R}} \mathbb X   \xrightarrow{\Phi^G}  e ( \sigma),$$
  If $\Phi^G$ is an isomorphism, then $e (\mathcal V) \otimes_{\mathcal{H}_{ R}} \mathbb X  $ is the extension to $G$ of  $ (e (\mathcal V) \otimes_{\mathcal{H}_{ R}} \mathbb X  )|_M$, meaning that $
M'_2  $  acts trivially on $e (\mathcal V) \otimes_{\mathcal{H}_{ R}} \mathbb X $.    The converse is true:
 \begin{lemma}  If $
M'_2  $  acts trivially on $e (\mathcal V) \otimes_{\mathcal{H}_{ R}} \mathbb X,$ then $\Phi^G$  is an  isomorphism.\end{lemma} 

 \begin{proof}  Suppose that  $
M'_2  $  acts trivially on $e (\mathcal V) \otimes_{\mathcal{H}_{ R}} \mathbb X $. Then $e (\mathcal V) \otimes_{\mathcal{H}_{ R}} \mathbb X  $ is the extension   to $G$ of   $ (e (\mathcal V) \otimes_{\mathcal{H}_{ R}} \mathbb X  )|_M$,  and by Theorem \ref{thm:ouf}, $(e (\mathcal V) \otimes_{\mathcal{H}_{ R}} \mathbb X )^\mathcal{U}$ is the extension of  $(e (\mathcal V) \otimes_{\mathcal{H}_{ R}} \mathbb X  )^{\mathcal{U}_M}$. Therefore \begin{equation*} (v  \otimes  1_{\mathcal{U}})T^{ *}_w=
(v  \otimes  1_{\mathcal{U}})T^{M, *}_w \quad 
\text{for all} \ v\in \mathcal V, w\in W_M(1).
\end{equation*}
 As $\mathcal{V}$ is extensible to $\mathcal{H}$, the natural map
 $v\mapsto  v\otimes   1_{\mathcal{U}} :\mathcal V\xrightarrow{\Psi} (e (\mathcal V) \otimes_{\mathcal{H}_{ R}} \mathbb X  )^{\mathcal{U}_M} $  
is $\mathcal{H}_M$-equivariant, i.e.:
 $$ v T^{M,*}_w\otimes   1_{\mathcal{U}}= (v  \otimes 1_\mathcal{U})T^{M, *}_w \quad 
\text{for all} \ v\in \mathcal V, w\in W_M(1).$$
because   (\eqref{eq:structure})  $v T^{M,*}_w\otimes   1_{\mathcal{U}}= v T^{ *}_w \otimes   1_{\mathcal{U}}  =  v \otimes  T^{ *}_w = (v  \otimes   1_{\mathcal{U}})T^{ *}_w$ in $e (\mathcal V) \otimes_{\mathcal{H}_{ R}} \mathbb X $. 

We recall that $- \otimes_{\mathcal{H}_{ M,R}} \mathbb X_M $ is the left adjoint of $(-)^{\mathcal{U}_M}$.  
  The adjoint $R[ M]$-homomorphism 
$\sigma=\mathcal{V} \otimes_{\mathcal{H}_{ M,R}} \mathbb X_M   \to e (\mathcal V) \otimes_{\mathcal{H}_{ R}} \mathbb X  $
 sends  $ v\otimes  1_{\mathcal{U}_M}$ to $v\otimes {\bf 1}_{\mathcal{U}} $ for all $v\in  \mathcal V$. 
 The $R[M]$-module generated by the $v\otimes {\bf 1}_{\mathcal{U}} $ for all $v\in  \mathcal V$ is equal to $e (\mathcal V) \otimes_{\mathcal{H}_{ R}} \mathbb X  $ because $M'_2$ acts trivially. Hence we obtained  an inverse of  $\Phi^G$.
   \end{proof}
 
Our next move is to determine if $M'_2$ acts trivially on $e (\mathcal V) \otimes_{\mathcal{H}_{ R}} \mathbb X $. It is equivalent to  see if  $M'_2$ acts trivially on $e (\mathcal V)   \otimes  1_{\mathcal{U}}$ as this set generates the representation $e (\mathcal V) \otimes_{\mathcal{H}_{ R}} \mathbb X $ of $G$ and 
$M'_2$  is a normal subgroup of $G$ as $M'_2$ and $M$ commute and   $G=ZM' M'_2$. Obviously,
 $\mathcal{U}\cap M'_2$ acts trivially on $e (\mathcal V)   \otimes  1_{\mathcal{U}}$. The group of double classes  $(\mathcal{U}\cap M'_2)\backslash M'_2/(\mathcal{U}\cap M'_2)$  is generated by the lifts  $\hat s\in \mathcal N \cap M'_2$ of the  simple affine roots $s $ of $W_{M'_2}$. Therefore,   $M'_2$ acts trivially on $e (\mathcal V) \otimes_{\mathcal{H}_{ R}} \mathbb X $ if and only if for any simple affine root $s \in  S_{M'_2}^{\aff}$ of $W_{M'_2}$, any $\hat s\in \mathcal N \cap M'_2$ lifting $s $ acts trivially on  $e (\mathcal V)   \otimes  1_{\mathcal{U}}$. 
 
\begin{lemma}Let $v\in \mathcal V, s \in  S_{M'_2}^{\aff}$ and $\hat s\in \mathcal N \cap M'_2$ lifting  $s $. We have 
$$(q_s+1)( v\otimes  1_{\mathcal{U}} - \hat s ( v\otimes  1_{\mathcal{U}}) )= 0.$$ \end{lemma}

 \begin{proof} We compute:
\begin{align*}T_{s} (\hat s  {\bf 1}_{\mathcal{U}}) &= \hat s (T_{ s}{\bf 1}_{\mathcal{U}}) = 1_{\mathcal{U} \hat s\mathcal{U} (\hat s)^{-1}}      =\sum _{u }\hat s u (\hat s)^{-1}{\bf 1}_{\mathcal{U}} = \sum _{u^{op} }u^{op} {\bf 1}_{\mathcal{U}} ,\\
   T_{ s} (\hat s^2  {\bf 1}_{\mathcal{U}})&= \hat s^2 (T_{ s}{\bf 1}_{\mathcal{U}})  =  {\bf 1}_{\mathcal{U} \hat s\mathcal{U} (\hat s)^{-2}}={\bf 1}_{\mathcal{U}( \hat s )^{-1}\mathcal{U}}  =\sum _{u } u\hat s{\bf 1}_{\mathcal{U}}.\end{align*}  
   for $u$ in the group $ \mathcal{U}/(\hat s^{-1}\mathcal{U} \hat s \cap \mathcal{U})$ and $u^{op}$ in the group $\hat s \mathcal{U} (\hat s)^{-1}/(\hat s \mathcal{U} (\hat s)^{-1} \cap \mathcal{U})$; the reason is that $\hat s^2$ normalizes $\mathcal{U}$, $\mathcal{U} \hat s\mathcal{U} \hat s^{-1}$ is the disjoint union of   the sets $\mathcal{U} \hat s u^{-1} (\hat s)^{-1}$   and  $\mathcal{U}( \hat s)^{-1}\mathcal{U} $  is the disjoint  union of the sets $\mathcal{U} ( \hat s)^{-1} u^{-1} $.  
   We introduce now    a natural  bijection 
   \begin{equation}\label{eq:bij}u\to u^{op}:  \mathcal{U}/(\hat s^{-1}\mathcal{U} \hat s \cap \mathcal{U})\to  \hat s \mathcal{U} (\hat s)^{-1}/(\hat s \mathcal{U} (\hat s)^{-1} \cap \mathcal{U})
   \end{equation}  which is not a group homomorphism. We recall the finite reductive group $G_{k,s}$  quotient of the parahoric subgroup $\mathfrak K_s$ of $G$ fixing the face  fixed by $s$ of the alcove $\mathcal C$. The Iwahori groups $Z^0\mathcal{U}$ and $Z^0\hat s\mathcal{U} (\hat s)^{-1} $  are contained in $ \mathfrak K_s$ and  their images in $G_{s,k}$ are opposite Borel subgroups $Z_k U_{s,k}$ and $Z_kU_{s,k}^{op}$. Via the  surjective maps 
 $u\mapsto \overline u: \mathcal{U}\to U_{s,k}$ and $u^{op}\mapsto \overline u^{op}:\hat s\mathcal{U} (\hat s)^{-1} \to U_{s,k}^{op}$   we   identify the groups   $\mathcal{U}/(\hat s^{-1}\mathcal{U} \hat s \cap \mathcal{U}) \simeq U_{s,k}$ and similarly $\hat s \mathcal{U} (\hat s)^{-1}/ (\hat s \mathcal{U} (\hat s)^{-1}\cap \mathcal{U}) \simeq U_{s,k}^{op}$.  Let $G'_{k,s}$ be the group generated by $U_{s,k}$ and $U_{s,k}^{op}$, and let $B'_{s,k}=G'_{k,s}\cap  Z_k  U_{s,k}=(G'_{k,s}\cap  Z_k)  U_{s,k} $.   We   suppose  (as we can) that  $\hat s\in \mathfrak K_s$ and that its image $\hat s_k$ in $G_{s,k}$  lies in    $G'_{k,s}$. We have  $\hat s_kU_{s,k} (\hat s_k) ^{-1} =U_{s,k}^{op}$  and the  Bruhat decomposition  $G'_{k,s}= B'_{k,s} \sqcup  U_{k,s}\hat s_k B'_{k,s} $ implies the existence of a canonical bijection  
 $\overline u^{op} \to \overline u   : ( U_{k,s}^{op} - \{1\})\to (U_{k,s}- \{1\}) $ respecting the cosets $\overline u^{op}B'_{k,s}=\overline u \hat s_k B'_{k,s}$.  Via the preceding identifications we get the wanted bijection \eqref{eq:bij}.

For $v\in  e(\mathcal{V})$ and  $z\in Z^0\cap M'_2$ we have $vT_z=v$, 
 $z  1_{\mathcal{U}}= T_z  1_{\mathcal{U}}$ and $v\otimes T_z  1_{\mathcal{U}}=vT_z\otimes  1_\mathcal{U}$ therefore $Z^0\cap M'_2$ acts trivially on $\mathcal V\otimes   1_{\mathcal{U}}$. The action of the group $(Z^0\cap M'_2)\mathcal{U}$ on  $\mathcal V\otimes   1_{\mathcal{U}}$ is also trivial.   As the image of $Z^0\cap M'_2$ in $G_{s,k}$ contains $Z_k\cap G'_{s,k}$,
 $$u\hat s (v\otimes    1_{\mathcal{U}})=
u ^{op}(v\otimes   1_{\mathcal{U}})$$ when $u$ and $u^{op}$ are not units and correspond via the bijection  \eqref{eq:bij}. 
   So we have \begin{align}\label{eq:for}v \otimes T_{ s} (\hat s   1_{\mathcal{U}}) - (v\otimes  1_{\mathcal{U}}) = v\otimes T_{ s} (\hat s^2    1_{\mathcal{U}}) -  v\otimes \hat s  1_{\mathcal{U}} 
  \end{align}
  We can move $T_s$ on the other side of $\otimes$ and as $vT_s=q_s v$ (Corollary \ref{cor:extT}), we can replace $T_s$ by $q_s$.
 We have $v  \otimes \hat s^2 1_\mathcal{U}=v  \otimes T_{s^{-2}} 1_\mathcal{U}$ because  $\hat s^2\in Z^0\cap M'_2$ normalizes $\mathcal{U}$; as we can move $T_{s^{-2}}$ on the other side of $\otimes$ and as $vT_{s^{-2}}=v$ we can forget $\hat s^2 $.  So \eqref{eq:for} is equivalent to
  $(q_s+1)( v\otimes 1_\mathcal{U} - \hat s ( v\otimes 1_\mathcal{U}) )= 0.$ \end{proof}
 
 Combining  the two lemmas we obtain:

\begin{proposition} \label{prop:QG} When $ \mathcal V $ is extensible to $\mathcal{H}$ and  has no $q_{s}+1$-torsion for any  $s \in S_{M'_2}^{\aff}$, then $
M'_2  $  acts trivially on $e (\mathcal V) \otimes_{\mathcal{H}_{ R}} \mathbb X $ and 
  $\Phi^G $  is an $R[G]$-isomorphism. 
 \end{proposition}

  Proposition \ref{prop:QG}  for the trivial character ${\bf 1}_\mathcal{H}$, says that 
 ${\bf 1}_\mathcal{H}\otimes_{\mathcal{H}_{ R}} \mathbb X $ is the trivial representation $ {\bf 1}_G$ of $G$   when $ q_{s}+1$ has no torsion in  $R$ for all $s\in S^{\aff}$. This is  proved  in \cite[Lemma 2.28]{arXiv:1703.04921} by a different method. 
 The following counter-example shows that this is not true for all $R$. 
 
  \begin{example}\label{ex:contr}  Let $G=GL(2,F)$ and $R$ an algebraically closed field where $q_{s_0}+1 = q_{s_1} + 1=0$ and $S_\aff = \{s_0,s_1\}$. (Note that $q_{s_0} = q_{s_1}$ is the order of the residue field of $R$.)
Then the dimension of   ${\bf 1}_\mathcal{H}\otimes_{\mathcal{H}_{ R}} \mathbb X $ is infinite, in particular ${\bf 1}_\mathcal{H}\otimes_{\mathcal{H}_{ R}} \mathbb X \neq {\bf 1}_G$.

Indeed,   the  Steinberg representation $\St_G=(\Ind_B^G{\bf 1}_Z)/ {\bf 1}_G$ of $G$ is an  indecomposable representation of length $2$ containing an irreducible infinite dimensional representation $\pi$   with $\pi^{\mathcal{U}}=0$  of quotient the character $(-1)^{\val \circ \det }$.  This follows from the proof of Theorem 3 and from Proposition  24 in \cite{MR1026328}. The kernel of the quotient map $\St_G \otimes (-1)^{\val \circ \det }\to {\bf 1}_G $ is infinite dimensional without a non-zero $\mathcal{U}$-invariant vector.  As the characteristic of $R$ is not $p$, the functor of $\mathcal{U}$-invariants is exact hence $(\St_G  \otimes (-1)^{\val \circ \det })^\mathcal{U} = {\bf 1}_\mathcal{H}$. As  $- \otimes_{\mathcal{H}_{ R}} R[\mathcal{U}  \backslash  G]$ is the left adjoint of $(-)^{\mathcal{U} }$ there is a non-zero homomorphism $${\bf 1}_\mathcal{H}\otimes_{\mathcal{H}_{ R}} \mathbb X \to  \St_G \otimes (-1)^{\val \circ \det } $$ with image  generated by its $\mathcal{U}$-invariants. The homomorphism is therefore surjective.  
    \end{example}

\subsection{$\mathcal{V}$ extensible to $\mathcal{H}$}
Let $P=MN$ be a standard parabolic subgroup of $G$ with $\Delta_P$ and $\Delta\setminus \Delta_P$ orthogonal. We  still suppose that  the $\mathcal{H}_{M,R}$-module $\mathcal{V}$ is extensible to $\mathcal{H}$, but  now $P\subset Q \subset G$. So we have  $I_\mathcal{H}(P,\mathcal V,Q)= e(\mathcal V)\otimes_R (\St_Q^G)^\mathcal{U} $ and $I_G(P,\sigma,Q)= e(\sigma)\otimes_R \St_Q^G $ where $\sigma=\mathcal V \otimes_{\mathcal{H}_{M,R}} \mathbb X_M$.  We compare the images by $-\otimes_{\mathcal{H}_R} \mathbb X$ of 
the $\mathcal{H}_R$-modules  $ e(\mathcal V)\otimes_R (\Ind_Q^G {\bf 1})^\mathcal{U}$ and $  e(\mathcal V)\otimes_R (\St_Q^G)^\mathcal{U}$  with the smooth $R$-representations $e(\sigma)\otimes \Ind_Q^G {\bf 1}$ and  $ e(\sigma)\otimes \St_Q^G$ of $G$.

As  $-\otimes_{\mathcal{H}_R} \mathbb X$ is left adjoint of $(-)^\mathcal{U}$, the   $\mathcal{H}_R$-homomorphism $v\otimes f \mapsto v\otimes  1_{\mathcal{U}_M}\otimes f: e(\mathcal{V})\otimes_R (\Ind_Q^G {\bf 1})^\mathcal{U}\to  (e(\sigma)\otimes_R \Ind_Q^G {\bf 1})^\mathcal{U}$ gives by adjunction  an $R[G]$-homomorphism
 $$v\otimes f\otimes 1_\mathcal{U} \mapsto v\otimes 1_{\mathcal{U}_M}\otimes f:(e(\mathcal{V})\otimes_R (\Ind_Q^G {\bf 1})^\mathcal{U}) \otimes_{\mathcal{H}_{ R}} \mathbb X   \xrightarrow{\Phi_Q^G} e(\sigma)\otimes_R \Ind_Q^G {\bf 1}.$$ 
 When $Q=G$  we have $\Phi_G^G=\Phi^G$. By Remark \ref{rem:surj}, $\Phi_Q^G$ is surjective.  Proposition \ref{prop:QG} applies with $M_Q$ instead of $G$ and gives
 the  $R[M_Q]$-homomorphism 
 $$v\otimes 1_{\mathcal{U}_{M_Q}} \mapsto v\otimes 1_{\mathcal{U}_M}: e_{\mathcal{H}_Q}(\mathcal{V})\otimes_{\mathcal{H}_{Q, R}} \mathbb X_{M_Q} \xrightarrow{\Phi^Q}e_Q(\sigma).$$

\begin{proposition} \label{prop:fourhom} The $R[G]$-homomorphism $\Phi_Q^G$ is an isomorphism if $\Phi^Q$ is an isomorphism, in particular if $ \mathcal V $  has no $q_{s}+1$-torsion for any  $s \in S_{M'_2 \cap M_Q}^{\aff}$.
\end{proposition}
\begin{proof}  The proposition follows from another construction of $\Phi_Q^G$  that we now describe. 
Proposition \ref{prop:ovvv} gives the   $\mathcal{H}_R$-module isomorphism 
$$v\otimes f_{Q\mathcal{U}}\mapsto v\otimes 1_{\mathcal{H}}:(e(\mathcal{V})\otimes_R (\Ind_Q^G {\bf 1})^\mathcal{U})\to  \Ind_{\mathcal{H}_Q}^\mathcal{H} (e_{\mathcal{H}_Q}(\mathcal{V}))=e_{\mathcal{H}_Q}(\mathcal{V}) \otimes_{\mathcal{H}_{M_{Q,R}^+},\theta} \mathcal{H}.$$  
We have the  $R[G]$-isomorphism \cite[Corollary~4.7]{arXiv:1703.04921}
 $$v\otimes 1_{\mathcal{H}} \otimes1_{\mathcal{U}}\mapsto f_{Q\mathcal{U}, v\otimes 1_{\mathcal{U}_{M_Q}}}:\Ind_{\mathcal{H}_Q}^\mathcal{H} (e_{\mathcal{H}_Q}(\mathcal{V})\otimes_{\mathcal{H}_{ R}} \mathbb X ) \to \Ind_Q^G(e_{\mathcal{H}_Q}(\mathcal{V})\otimes_{\mathcal{H}_{Q, R}} \mathbb X_{M_Q} )$$ 
 and the   $R[G]$-isomorphism ***
 $$f_{Q\mathcal{U},v\otimes 1_{\mathcal{U}_M}} \mapsto v\otimes 1_{\mathcal{U}_M}\otimes f_{Q\mathcal{U}}:\Ind_Q^G(e_Q(\sigma))\to e(\sigma)\otimes \Ind_Q^G {\bf 1} .$$  
 From   $\Phi^Q$ and these three homomorphisms, there exists a unique $R[G]$-homomorphism 
$$(e(\mathcal{V})\otimes_R (\Ind_Q^G {\bf 1})^\mathcal{U}) \otimes_{\mathcal{H}_{ R}} \mathbb X   \to  e(\sigma)\otimes_R \Ind_Q^G {\bf 1}$$  sending   $v\otimes f_{Q\mathcal{U}} \otimes1_{\mathcal{U}}$ to $v\otimes 1_{\mathcal{U}_M}\otimes f_{Q\mathcal{U}}$.  We deduce: this homomorphism is equal to $\Phi_Q^G$,  $\mathcal{V} \otimes 1_{Q\mathcal{U}} \otimes1_{\mathcal{U}}$ generates  $(e(\mathcal{V})\otimes_R (\Ind_Q^G {\bf 1})^\mathcal{U}) \otimes_{\mathcal{H}_{ R}} \mathbb X $, 
if $\Phi^Q$ is an isomorphism then $\Phi_Q^G$ is an isomorphism. 
By Proposition \ref{prop:QG}, if $\mathcal{V}$ has no $q_{s}+1$-torsion for any  $s \in S_{M'_2 \cap M_Q}^{\aff}$, then  $\Phi^Q$ and $\Phi_Q^G$ are   isomorphisms. \end{proof}

We recall that the $\mathcal{H}_{M,R}$-module $ \mathcal V $ is extensible to $\mathcal{H}$.
\begin{proposition}\label{prop:extX} The $R[G]$-homomorphism $\Phi_{Q}^G$ induces an $R[G]$-homomorphism
$$(e(\mathcal V)\otimes_R (\St_Q^G)^\mathcal{U}) \otimes_{\mathcal{H}_{ R}} \mathbb X  \to e(\sigma)\otimes_R \St_Q^G,$$
It is an isomorphism if $\Phi_{Q'}^G$  is an $R[G]$-isomorphism for all  parabolic subgroups $Q'$ of $G$ containing $Q$, in particular if $\mathcal{V}$ has no $q_{s}+1$-torsion for any  $s \in S_{M'_2 }^{\aff}$.\end{proposition}
 
\begin{proof} The proof is straightforward, with the arguments already developped for Proposition \ref{prop:ovvv} and Theorem \ref{thm:main8}. The representations $e(\sigma)\otimes_R \St_Q^G$ and $(e(\mathcal V)\otimes_R (\St_Q^G)^\mathcal{U}) \otimes_{\mathcal{H}_{ R}} \mathbb X $ of $G$ are the cokernels of the natural $R[G]$-homomorphisms
\begin{align*}\oplus_{Q\subsetneq Q' }e(\sigma)\otimes_R \Ind_{Q'}^G {\bf 1} &\xrightarrow {\id\otimes \alpha} e(\sigma)\otimes_R \Ind_Q^G {\bf 1} ,\\
\oplus_{Q\subsetneq Q' }
(e(\mathcal{V})\otimes_R (\Ind_{Q'}^G {\bf 1})^\mathcal{U}) \otimes_{\mathcal{H}_{ R}} \mathbb X  & \xrightarrow{\id\otimes \alpha^\mathcal{U} \otimes \id} (e(\mathcal{V})\otimes_R (\Ind_{Q}^G {\bf 1})^\mathcal{U}) \otimes_{\mathcal{H}_{ R}} \mathbb X .
\end{align*}
These $R[G]$-homomorphisms make a commutative diagram with the   $R[G]$-homomorphisms $\oplus_{Q\subsetneq Q' }\Phi_{Q'}^G$ and $\Phi_{Q}^G$ going from the lower line to the upper line. Indeed,  let $ v \otimes f_{Q'\mathcal{U}} \otimes 1_\mathcal{U} \in (e(\mathcal{V})\otimes_R (\Ind_{Q'}^G {\bf 1})^\mathcal{U}) \otimes_{\mathcal{H}_{ R}} \mathbb X  $. One one hand, it  goes to   $ v \otimes f_{Q\mathcal{U}} \theta_{Q'}(e_Q^{Q'}) \otimes  1_\mathcal{U}\in (e(\mathcal{V})\otimes_R (\Ind_{Q}^G {\bf 1})^\mathcal{U}) \otimes_{\mathcal{H}_{ R}} \mathbb X  $  by the horizontal map, and then   to $ v  \otimes  1_{\mathcal{U}_M}\otimes f_{Q\mathcal{U}} \theta_{Q'}(e_Q^{Q'})$ by the vertical map. On  the other hand,   it goes to  $ v  \otimes  1_{\mathcal{U}_M}\otimes f_{Q'\mathcal{U}}$ by the vertical map,  and then  to $ v  \otimes  1_{\mathcal{U}_M}\otimes f_{Q\mathcal{U}} \theta_{Q'}(e_Q^{Q'})$ by the horizontal map. 
One deduces that  $\Phi_{Q}^G$ induces an $R[G]$-homomorphism 
$( e(\mathcal V)\otimes_R (\St_Q^G)^\mathcal{U}) \otimes_{\mathcal{H}_{ R}} \mathbb X \to e(\sigma)\otimes_R \St_Q^G$, which is an isomorphism if  $\Phi_{Q'}^G$  is an $R[G]$-isomorphism for all $Q\subset Q'$.
  \end{proof}
  
\subsection{General}

We consider now the general case: let  $P=MN\subset Q$ be two standard parabolic subgroups of $G$ and $\mathcal{V}$  a non-zero right $\mathcal{H}_{M,R}$-module  with $Q\subset P(\mathcal{V})$.  
 We recall   $ I_\mathcal{H}(P,\mathcal{V},Q)= \Ind_{\mathcal{H}_{M(\mathcal{V})}}^{\mathcal{H}}((e(\mathcal{V})\otimes_R (\St_{Q  }^{P(\mathcal{V})})^{\mathcal{U}_{M(\mathcal{V})}}) $ 
and  $\sigma=\mathcal V \otimes_{\mathcal{H}_{M,R}} \mathbb X_M$ (Proposition\ref{prop:orthtriple}).  There is a natural $R[G]$-homomorphism  
 $$ I_\mathcal{H}(P,\mathcal{V},Q)\otimes _{\mathcal{H}_R} \mathbb X \xrightarrow{\Phi_{I}^G} \Ind_{P(\mathcal{V})}^G(e_{M(\mathcal{V})}(\sigma)\otimes_R \St_{Q }^{P(\mathcal{V})} )$$
 obtained by composition of the 
 $R[G]$-isomorphism \cite[Corollary~4.7]{arXiv:1703.04921} (proof of Proposition \ref{prop:fourhom}):
$$ I_\mathcal{H}(P,\mathcal{V},Q)\otimes _{\mathcal{H}_R} \mathbb X \to \Ind_{P(\mathcal{V})}^G((e(\mathcal{V})\otimes_R (\St_{Q\cap M(\mathcal{V}) }^{M(\mathcal{V})})^{\mathcal{U}_{M(\mathcal{V})}}) \otimes_{\mathcal{H}_{M(\mathcal{V}), R}} \mathbb X_{M(\mathcal{V})} ),$$
with the  $R[G]$-homomorphism
 $$\Ind_{P(\mathcal{V})}^G((e(\mathcal{V})\otimes_R (\St_{Q }^{P(\mathcal{V})})^{\mathcal{U}_{M(\mathcal{V})}}) \otimes_{\mathcal{H}_{M(\mathcal{V}), R}} \mathbb X_{M(\mathcal{V})} )  \to \Ind_{P(\mathcal{V})}^G(e_{M(\mathcal{V})}(\sigma)\otimes_R \St_{Q }^{P(\mathcal{V})}),$$
 image by the parabolic induction $\Ind_{P(\mathcal{V})}^G$ of the  homomorphism
 $$(e(\mathcal{V})\otimes_R (\St_{Q}^{P(\mathcal{V})})^{\mathcal{U}_{M(\mathcal{V})}}) \otimes_{\mathcal{H}_{M(\mathcal{V}), R}} \mathbb X_{M(\mathcal{V})}   \to e_{M(\mathcal{V})}(\sigma)\otimes_R \St_Q^{P(\mathcal{V})} .$$
induced by  the  $R[M(\mathcal{V})]$-homomorphism $\Phi_Q^{P(\mathcal{V})}=\Phi_{Q\cap M(\mathcal{V})}^{M(\mathcal{V})}$ of Proposition \ref{prop:extX} applied to $M(\mathcal{V})$ instead of $G$.

 This homomorphism $\Phi_{I}^G$ is an isomorphism if   $\Phi_{Q}^{P(\mathcal{V})}$ is an isomorphism, in particular if  $\mathcal{V}$ has no $q_{s}+1$-torsion for any  $s \in S_{M'_2 }^{\aff}$ where $\Delta_{M_2}= \Delta_{M(\mathcal{V})}\setminus\Delta_M$ (Proposition \ref{prop:extX}). We get the main theorem of this section:
 \begin{theorem} \label{thm:10.10} Let $(P=MN,\mathcal{V},Q)$ be an $\mathcal{H}_R$-triple and  $\sigma=\mathcal V \otimes_{\mathcal{H}_{M,R}}R[\mathcal{U}_M\backslash M]$. Then, $(P,\sigma,Q)$ is an $R[G]$-triple. The  $R[G]$-homomorphism  $$I_\mathcal{H}(P,\mathcal V,Q)\otimes_{\mathcal{H}_R}R[\mathcal{U} \backslash G]  \xrightarrow{\Phi_{I}^G} \Ind_{P(\mathcal{V})}^G(e_{M(\mathcal{V})}(\sigma)\otimes_R \St_{Q }^{P(\mathcal{V})} )$$
  is an isomorphism  if   $\Phi_{Q}^{P(\mathcal{V})}$ is an isomorphism. In  particular $\Phi_{I}^G$ is an isomorphism if  $\mathcal{V}$ has no $q_{s}+1$-torsion for any  $s \in S_{M'_2 }^{\aff}$. \end{theorem}
 
 Recalling $I_G(P,\sigma,Q)  = \Ind_{P(\sigma)}^G( e (\sigma)\otimes_R \St_Q^{P(\sigma)} ) $ when $\sigma\neq 0$, we deduce:
\begin{corollary} \label{cor:10.12}We have:

 $I_\mathcal{H}(P,\mathcal V,Q)\otimes_{\mathcal{H}_R} R[\mathcal{U} \backslash G] \simeq I_G(P, \sigma,Q) $, if  $\sigma\neq 0$, $P(\mathcal V)=P( \sigma)$ and $\mathcal{V}$ has no $q_{s}+1$-torsion for any  $s \in S_{M'_2 }^{\aff}$.

  $I_\mathcal{H}(P,\mathcal V,Q)\otimes_{\mathcal{H}_R} R[\mathcal{U} \backslash G]=I_G(P,\sigma,Q)  =0$,  if $\sigma=0$.

\end{corollary}
Recalling $P(\mathcal V)=P( \sigma)$ if $\sigma\neq 0$, $R$ is a field of characteristic $p$ and $\mathcal{V}$  simple supersingular  (Proposition \ref{prop:orthtriple} 4)), we deduce:
\begin{corollary} \label{cor:10.13}
$I_\mathcal{H}(P,\mathcal V,Q)\otimes_{\mathcal{H}_R} R[\mathcal{U} \backslash G] \simeq I_G(P, \sigma,Q) $ if $R$ is a field of characteristic $p$ 
 and $\mathcal{V}$ simple supersingular.
\end{corollary} 
  \section{Vanishing of the smooth dual }\label{S:9}
 
 Let $V$ be an $R[G]$-module.  The dual $\Hom _R (V,R)$ of $V$ is an $R[G]$-module for  the contragredient action: $gL (gv)=L(v)$  if $g\in G$, $L\in \Hom _R (V,R)$ is a linear form  and $v\in V$.  When $V\in \Mod_R^\infty(G)$ is a smooth $R$-representation of $G$, the dual of $V$ is not necessarily smooth.  A linear form $L$ is smooth   if there exists an open  subgroup $H\subset G$ such that $L(hv)=L(v)$ for all $h\in H, v\in V$; the space $\Hom _R (V,R)^\infty$of smooth linear forms   is a smooth $R$-representation of $G$,  called the {\bf smooth dual} (or  smooth contragredient) of $V$. The  smooth dual of $V$  is contained in the dual of $V$.

 \begin{example} 
  When  $R$ is a field and  the dimension of $V$ over $R$ is finite, the  dual   of   $V$ is equal to the smooth dual of $V$ because the kernel of the action of $G$ on $V$ is an open normal subgroup $H\subset G$; the  action of $G$ on the dual $\Hom _R (V,R)$  is  trivial on  $H$.
\end{example}
 
We assume in this section that $R$ is a field of characteristic $p$. 
Let $P=MN$ be  a  parabolic subgroup of $G$ and $V\in \Mod_R^\infty (M)$.  Generalizing the proof given  in \cite[8.1]{MR2392364}  when $G=GL(2,F)$ and the dimension of $V$ is $1$, we show: 

\begin{proposition}\label{He} If $P\neq G$, the smooth dual of  $\Ind_P^G (V)$ is $0$.
 \end{proposition}

\begin{proof} Let $L$ be a smooth linear form on $\Ind_P^G (V)$ and   $K$   an open pro-$p$-subgroup
of $G$ which fixes $L$. Let $J$ an arbitrary open subgroup  of $K$, $g\in G$ and $f\in  (\Ind_P^G (V))^J$ with support $PgJ$. We want to show that $L(f)=0$. Let $J' $ be any open
normal subgroup of $ J$ and let $\varphi$ denote  the function in $(\Ind_P^G (V))^{J'}$ with support 
$PgJ'$ and value $\varphi(g)=f(g)$ at $g$. For $j\in J$ we have $L(j\varphi)=L(\varphi)$, and the support of $ j\varphi(x)= \varphi (xj)$ is $PgJ'j^{-1}$. 
The function $f$ is the sum of translates $ j \varphi$, 
where $ j$ ranges through the left cosets of the image $X$ of 
 $g^{-1}Pg\cap  J$ in $J/J' $, so that $L(f)= r L(\varphi)$ where $r$ is the order of $X$ in $J/J'$. 
We can certainly find $J'$ such that $r\neq 1$, and then $r$ is a positive power of $p$. As  the characteristic of $C$ is $p$ we  have $L(f)=0$.
\end{proof}
 
The  module  $R[\mathcal{U}\backslash G]$ is contained in the module $R^{\mathcal{U}\backslash G}$ of functions $f: \mathcal{U}\backslash G\to R$. The actions of $\mathcal{H}$ and of $G$ on $R[\mathcal{U}\backslash G]$ extend to $R^{\mathcal{U}\backslash G}$ by the same formulas. The pairing 
$$(f,\varphi)\mapsto \langle f,\varphi\rangle=\sum_{g\in \mathcal{U}\backslash G} f(g) \varphi (g):R^{\mathcal{U}\backslash G}\times R[\mathcal{U}\backslash G]\to R$$
identifies $R^{\mathcal{U}\backslash G}$ with the dual of $R[\mathcal{U}\backslash G]$. Let $h\in \mathcal{H}$ and $\check h\in \mathcal{H}, \ \check h (g)=h(g^{-1})$ for $g\in G$. We have 
$$\langle f,h\varphi\rangle= \langle \check hf,\varphi\rangle.
$$

 \begin{proposition}\label{arXiv:1406.1003_accepted}
When $R$ is an algebraically closed field of characteristic $p$, $G$ is not compact modulo the center  and  $\mathcal V$ is a simple supersingular right $\mathcal{H}_R$-module,  the  smooth dual of $\mathcal{V}\otimes_{\mathcal{H}_R}R[{\mathcal{U}}\backslash G]$ is 
$0$. 
\end{proposition}

 \begin{proof}   
  Let  $\mathcal{H}_R^{\aff}$ be the subalgebra of $\mathcal{H}_R$ of basis $(T_w)_{w\in W'(1)}$ where $W'(1)$ is the inverse image of $W'$ in $W(1)$.
 The   dual of $\mathcal{V}\otimes_{\mathcal{H}_R}R[{\mathcal{U}}\backslash G]$  is contained in the   dual of $\mathcal{V}\otimes_{\mathcal{H}_R^{\aff}}R[{\mathcal{U}}\backslash G]$;   the   $\mathcal{H}_R^{\aff}$-module  $\mathcal{V}|_{ \mathcal{H}_R^{\aff}}$  is a finite sum of supersingular characters \cite{Vigneras-prop-III}. Let $\chi: \mathcal{H}_R^{\aff}\to R$ be a 
supersingular character. The  dual of  $\chi\otimes_{\mathcal{H}_R^{\aff}} R[{\mathcal{U}}\backslash G]$ is contained in the dual of $R[{\mathcal{U}}\backslash G]$   isomorphic to $R^{\mathcal{U}\backslash G}$. It is the space of $f\in R^{\mathcal{U}\backslash G}$ with $\check h f =\chi(h)f$ for all  $h\in  \mathcal{H}_R^{\aff} $.
The smooth dual of $\chi\otimes_{\mathcal{H}_R^{\aff}} R[{\mathcal{U}}\backslash G]$  is $0$ if the 
dual of $\chi\otimes_{\mathcal{H}_R^{\aff}} R[{\mathcal{U}}\backslash G]$ has no non-zero element   fixed by $\mathcal{U}$.
Let us take $f\in R^{\mathcal{U}\backslash G/\mathcal{U}}$ 
 with $\check h f =\chi(h)f$ for all  $h\in  \mathcal{H}_R^{\aff} $. We shall prove that $f=0$.  We have $\check T_w=T_{w^{-1}}$ for $w\in W(1)$.

 The elements  $(T_t)_{t\in Z_k}$ and $(T_{\tilde s})_{s\in S^{\aff}}$ where $\tilde s$ is an admissible  lift of  $s$ in $ W^{\aff}(1)$, generate the algebra $ \mathcal{H}_R^{\aff} $ and 
  $$T_t T_w=T_{tw},\quad T_{\tilde s}T_w=\begin{cases} T_{\tilde s w}  & \tilde s w>w,\\
  c_{\tilde s}T_w& \tilde s w<w.
  \end{cases}
  $$
with $c_{\tilde s}= - |Z_{k,s}'|\sum _{t\in Z'_{k,s}}T_t$ because the characteristic of $R$ is $p$ \cite[Proposition 4.4]{MR3484112}. 
Expressing $ f = \sum_{ w \in W(1)}a_{ w} T_{ w} $, $a_w\in R$,  as an infinite sum,  we have 
  $$T_{t}f=\sum_{w \in W(1)}a_{t^{-1} w} T_{ w}, \quad T_{\tilde s}f=\sum_{w \in W(1), \tilde s w <  w}  (a_{(\tilde s)^{-1} w} + a_{ w} c_{\tilde s})T_{ w} ,$$
where  $<$ denote  the Bruhat order  of  $W(1)$ associated to $S^{\aff}$ \cite{MR3484112} and  \cite[Proposition 4.4]{MR3484112}. 
A  character $\chi$ of $ \mathcal{H}_R^{\aff} $ is associated to a character $\chi_k:Z_k\to R^*$ and a subset $J $ of  
$$S_{\chi_k}^{\aff}=\{ s\in  S ^{\aff}  \ | \ (\chi_k)|_ {Z'_{k,s}} \ \text{ trivial} \ \} $$ 
 \cite[Definition 2.7]{Vigneras-prop-III}. We have 
      \begin{equation}\label{cf} \begin{cases} \chi(T_t )=\chi_k(t)  \quad \quad \  t\in Z_k,\\
      \chi(T_{\tilde s}) = \begin{cases} 0 &  s\in S^{\aff}\setminus J  , \\  -1 & s\in  J.
   \end{cases}  
   \end{cases}  \quad  (\chi_k)(c_{\tilde s})= \begin{cases} 0 &  s\in S^{\aff}\setminus S_{\chi_k}^{\aff}  , \\  -1 & s\in  S_{\chi_k}^{\aff}. 
   \end{cases} 
\end{equation}
 Therefore $ \chi_k(t)f= \check T_{t}f=T_{t^{-1}}f$ hence $\chi_k(t) a_w= a_{tw}$. 
 We have  $\chi(T_{\tilde s})f=\check T_{\tilde s}f= T_{(\tilde s)^{-1}}f = T_{\tilde s} T_{(\tilde s)^{-2}} f=\chi_k((\tilde s)^2) T_{\tilde s}f  $; as    $(\tilde s) ^2 \in Z'_{k,s}$ \cite[three lines before Proposition 4.4]{MR3484112} and $J\subset S_{\chi_k}^{\aff}$, we obtain \begin{equation}\label{eq:Tsf}T_{\tilde s}f=\begin{cases} 0 &  s\in S^{\aff}\setminus J  , \\  -f & s\in  J.
   \end{cases}
   \end{equation}  Introducing $\chi_k(t) a_w= a_{tw}$ in the formula for  $T_{\tilde s}f $, we get 
\begin{align*}\sum_{w \in W(1), \tilde s w <  w}  a_{ w} c_{\tilde s} T_{ w} &=-|Z'_{k,s}|^{-1}\sum_{w \in W(1), \tilde s w <  w, t\in Z'_{k,s}}   a_{ w} T_{t w}\\
& =-  |Z'_{k,s}|^{-1}\sum_{w \in W(1), \tilde s w <  w, t\in Z'_{k,s}} a_{t^{-1} w} T_{w}\\
&=-  |Z'_{k,s}|^{-1}\sum_{t\in Z'_{k,s}}\chi_k(t^{-1})\sum_{w \in W(1), \tilde s w <  w} a_{ w} T_{w}\\
& =\chi_k(c_{\tilde s})\sum_{w \in W(1), \tilde s w <  w} a_{ w} T_{w}.
\end{align*}
\begin{align*}
T_{\tilde s}f & =\sum_{w \in W(1), \tilde s w <  w}  (a_{(\tilde s)^{-1} w}+  a_{ w} \chi_k(c_{\tilde s}))T_{ w}\\
&= \begin{cases}\sum_{w \in W(1), \tilde s w <  w}   a_{(\tilde s)^{-1} w}T_{ w} & s\in S^{\aff}\setminus  S_{\chi_k}^{\aff},
\\
\sum_{w \in W(1), \tilde s w <  w}  (a_{(\tilde s)^{-1} w}-  a_{ w})T_{ w} & s\in  S_{\chi_k}^{\aff}
.
\end{cases}
\end{align*} From  the last  equality and \eqref{eq:Tsf} for  $T_{\tilde s}f $, we get:     \begin{equation}\label{df}
 a_{\tilde s  w} = \begin{cases}0 &s \in J \cup (S^{\aff}\setminus  S_{\chi_k}^{\aff}),  \tilde s  w < w
 ,\\
  a_{w} & s\in S_{\chi_k}^{\aff} \setminus J .
\end{cases}
\end{equation}
  Assume that $a_ {w}\ne 0$. By the first condition, we know that $w>\tilde s w$ for $s \in J \cup (S^{\aff}\setminus  S_{\chi_k}^{\aff})$.
The character $\chi$ is supersingular  if  for each irreducible component $X$ of $S^{\aff}$, the intersection  $X\cap J$ is not  empty and different from $X$  \cite[Definition 2.7, Theorem  6.18]{Vigneras-prop-III}. This implies that  the group generated by the $s\in S_{\chi_k}^{\aff} \setminus J$ is finite.
If $\chi$ is supersingular, by the second condition   we can suppose $w>\tilde s w$ for any  $s\in S^{\aff}$. But there 
is no such element if $S^{\aff}$ is not empty.
  \end{proof}
  
 \begin{theorem}\label{contra}  Let  $\pi$ be an irreducible admissible $R$-representation of $G$ with a non-zero smooth dual where $R$ is an algebraically closed field of characteristic $p$. Then $\pi$ is finite dimensional. \end{theorem}
 \begin{proof} Let  $(P,\sigma,Q)$ be   a  $R[G]$-triple with $\sigma$ supercuspidal such that $\pi \simeq I_G(P,\sigma,Q)$.  The representation $I_G(P,\sigma,Q)$ is a quotient of  
 $\Ind_Q^G e_Q(\sigma)$  hence  the smooth  dual of  $\Ind_Q^G e_Q(\sigma)$ is not  zero. From  Proposition \ref{He}, $Q=G$.  We have $I_G(P,\sigma,G)=e(\sigma)$.   
 The smooth dual of 
  $\sigma$ contains the  smooth linear dual of $e(\sigma)$ hence is not zero. 
As $\sigma$ is  supercuspidal, the $\mathcal{H}_{M}$-module $\sigma^{\mathcal{U}_M}$ contains a simple supersingular submodule $\mathcal V$ \cite[Proposition 7.10, Corollary~7.11]{Vigneras-prop-III}.  The functor  $-\otimes_{\mathcal{H}_{M,R}} 
R[{\mathcal{U}_M}\backslash M]$ being the right adjoint of  $(-)^{\mathcal{U}_M}$,  the irreducible representation $\sigma$ is a quotient of $\mathcal V\otimes_{\mathcal{H}_{M,R}} R[{\mathcal{U}_M}\backslash M]$,  hence the smooth  dual of $\mathcal V\otimes_{\mathcal{H}_{M,R}} R[{\mathcal{U}_M}\backslash M]$   is not  zero.  By Proposition \ref{arXiv:1406.1003_accepted},  $M=Z $. Hence $\sigma$ is finite dimensional and the same is true for $e(\sigma)= I_G(B,\sigma,G) \simeq \pi$.
 \end{proof}
  
\begin{remark}
  When  the characteristic of $F$ is $0$,  Theorem \ref{contra} was proved by Kohlhaase for a field $R$ of characteristic $p$. He gives two proofs  \cite[Proposition 3.9, Remark 3.10]{kohlhasse-smooth-duality}, but none of them  extends to $F$ of characteristic $p$.
Our proof  is valid without restriction on the characteristic of $F$ and does not use the results of Kohlhaase. Our assumption that $R$ is an algebraically closed field of characteristic $p$ comes  from the classification theorem in \cite{MR3600042}.
\end{remark}

\end{document}